\documentclass[11pt]{article}
\usepackage[ansinew]{inputenc}
\usepackage{graphicx}
\usepackage{color}
\usepackage[colorlinks]{hyperref}

\usepackage{enumerate,latexsym}
\usepackage{latexsym}
\usepackage{amsmath,amssymb}
\usepackage{graphicx}
\usepackage{times}
\usepackage{soul}

\newfont{\bb}{msbm10 at 11pt}
\newfont{\bbsmall}{msbm8 at 8pt}
\def\cL{\mathcal{L}}
\def\cS{\mathcal{S}}
\def\cF{\mathcal{F}}
\def\R{\mathbb{R}}
\def\N{\mathbb{N}}
\def\B{\mathbb{B}}
\def\D{\mathbb{D}}
\def\esf{\mathbb{S}}
\newcommand{\ds}{\displaystyle}

\newcommand{\Int}{\mbox{\rm Int}}
\newcommand{\C}{\mbox{\bb C}}

\newcommand{\Z}{\mbox{\bb Z}}

\newcommand{\Nsmall}{\mbox{\bbsmall N}}

\newcommand{\Zsmall}{\mbox{\bbsmall Z}}

\newcommand{\rth}{\R^3}
\newcommand{\wt}{\widetilde}
\newcommand{\wh}{\widehat}
\newcommand{\ov}{\overline}
\newcommand{\ben}{\begin{enumerate}}
\newcommand{\een}{\end{enumerate}}

\def\a{{\alpha}}
\def\lc{{\cal L}}
\def\t{{\theta}}

\def\g{{\gamma}}
\def\G{{\Gamma}}
\def\l{{\lambda}}

\def\de{{\delta}}
\def\be{{\beta}}
\def\ve{{\varepsilon}}

\newtheorem{theorem}{Theorem}[section]
\newtheorem{lemma}[theorem]{Lemma}
\newtheorem{proposition}[theorem]{Proposition}

\newtheorem{remark}[theorem]{Remark}
\newtheorem{corollary}[theorem]{Corollary}
\newtheorem{definition}[theorem]{Definition}
\newtheorem{conjecture}[theorem]{Conjecture}
\newtheorem{assertion}[theorem]{Assertion}

\newtheorem{claim}[theorem]{Claim}

\definecolor{pp}{rgb}{.5,0,.7}

\definecolor{rrr}{rgb}{.9,0,.1}

\newenvironment{proof}{\smallskip\noindent{\it Proof.}\hskip \labelsep}
                         {\hfill\penalty10000\raisebox{-.09em}{$\Box$}\par\medskip}

\usepackage{pdfsync}

\begin{document}

  \begin{title}
{The embedded {C}alabi-{Y}au conjecture for finite genus}
\end{title}

\begin{author}
{William H. Meeks III\thanks{This material is based upon
 work for the NSF under Award No. DMS -
  1309236. Any opinions, findings, and conclusions or recommendations
 expressed in this publication are those of the authors and do not
 necessarily reflect the views of the NSF.}
 \and Joaqu\'\i n P\' erez
\and Antonio Ros\thanks{Research partially supported by
MINECO/FEDER grants no. MTM2014-52368-P and MTM2017-89677-P.} }
\end{author}

\maketitle
\begin{abstract} Suppose $M$ is a complete,
embedded minimal surface in $\R^3$ with an infinite number of ends,
finite genus and compact boundary. We prove that
the simple limit ends of $M$ have properly embedded representatives
with compact boundary, genus zero and with constrained geometry.  We
use this result to show that if $M$ has at least two simple limit
ends, then $M$ has exactly two simple limit ends. Furthermore, we
demonstrate that $M$ is properly embedded in $\rth$ if and only if
$M$ has at most two limit ends if and only if $M$ has a countable
number of limit ends.
\par
\vspace{.17cm} \noindent{\it Mathematics Subject Classification:}
Primary 53A10, Secondary 49Q05, 53C42

\noindent{\it Key words and phrases:} Proper minimal surface,
embedded Calabi-Yau problem,
minimal lamination, limit end, injectivity radius function, locally simply connected.
\end{abstract}

\section{Introduction.}
The Calabi-Yau conjectures refer to a series of conjectures
concerning the nonexistence of a complete, minimally immersed
surface $f \colon M \rightarrow \rth$ whose image $f(M)$ is
constrained to be contained in a particular region of $\rth$
(see Calabi~\cite{ca4}, page 212 in Chern~\cite{che4}, problem~91 in
Yau~\cite{yau1} and page 360 in Yau~\cite{yau2}).  Calabi's original
conjectures~\cite{ca4} state that a complete, nonflat minimal
surface cannot be contained  in the unit ball $\B(1)=\{x \in \rth
\mid |x| < 1\}$  or even in a halfspace of $\R^3$.
Among the positive results on the Calabi-Yau conjectures, we mention that the Strong
Halfspace Theorem~\cite{hm10} implies the validity of the
conjectures for properly immersed minimal surfaces in a closed
halfspace. A spectacular positive result by Colding and
Minicozzi~\cite{cm35} is that any complete, embedded minimal surface $M$
in $\R^3$ with finite topology is proper, and so the Halfspace Theorem (Hoffman and Meeks~\cite{hm10})
implies that $M$ cannot be contained in a halfspace unless it is a finite number of parallel planes.
In contrast to Colding and Minicozzi's properness result for the
finite topology embedded Calabi-Yau
problem, Ferrer, Mart\'\i n,
Meeks and Nadirashvili have conjectured that there is a particular bounded domain
$\Omega$ in $\rth$  (see~\cite{fmm1} for a description of $\Omega$),
which is smooth except at one point and  satisfies the following
property:
{\em Every open surface  with compact (possibly empty) boundary whose ends have infinite genus
admits a complete, proper minimal {\em embedding} into $\Omega$.}
We refer the reader to
Section~\ref{secprelim} for a brief elementary topological discussion of
the notions of end, the genus of an end, limit end, simple limit end
and end representative
for any noncompact surface, terms  we will use freely in this manuscript.

The theory developed in this paper represents the first step in resolving
the following fundamental conjecture, which gives a strong converse to the just
mentioned existence conjecture of Ferrer, Mart\'\i n,
Meeks and Nadirashvili for open surfaces with compact boundary.

\begin{conjecture} [Embedded Calabi-Yau  Conjecture for Finite Genus]\label{conj}
Every connected, complete
embedded minimal surface $M \subset \rth$ of finite genus and compact (possibly empty)
boundary is properly embedded in $\rth$.
\end{conjecture}

Corollary 1 in~\cite{mpr3} implies Conjecture~\ref{conj} under
the additional hypothesis that $M$ is a leaf of a minimal lamination
$\lc$ of $\rth$, or equivalently, when $M$ has locally bounded
Gaussian curvature in $\rth$. As mentioned above, Colding and
Minicozzi~\cite{cm35} have proved Conjecture~\ref{conj}
under the additional assumption that $M$ has finite topology.
In~\cite{mr13}, Meeks and Rosenberg  proved that connected,
complete embedded minimal surfaces in $\R^3$ with positive injectivity radius are proper;
their theorem is a generalization of the properness result of Colding
and Minicozzi since complete, embedded finite topology minimal surfaces in $\rth$
have positive injectivity radius.

These results, together with others by Bernstein and  Breiner~\cite{bb1},
Collin~\cite{col1}, Meeks and
P\'erez~\cite{mpe3} and  Meeks and Rosenberg~\cite{mr8},  imply that  a complete,
nonflat embedded minimal surface $M\subset \R^3$ with finite
topology has annular ends which are asymptotic to ends of planes,
catenoids or $M$ has just one end which is asymptotic to the end of a helicoid.
In all of these  cases, $M$ is proven to be conformally a compact Riemann surface
$\overline{M}$ punctured in a finite number of points (in
particular, $M$ is recurrent for Brownian motion), and the embedding
of $M$ into $\rth$   can be expressed analytically in terms
of meromorphic data defined on $\overline{M}$. In the case that
a complete embedded minimal surface $M$ of finite topology in $\R^3$ has nonempty compact boundary,
a similar
description of its conformal structure ($\partial M$ has full harmonic measure)
and of its  asymptotic behavior (a few more asymptotic types arise than in the case
without boundary) hold, see~\cite{mpe3} for details.
Concerning conformal questions, one consequence of the results in this paper is
Corollary~\ref{corolnew}, which states that if a properly embedded
minimal surface in $\R^3$ has a  limit end of genus zero, then it
is recurrent; this can be viewed as a generalization of our previous
result~\cite{mpr4} that any properly embedded minimal surface of
finite genus in $\R^3$ is recurrent.

Using the techniques developed by Colding and Minicozzi
\cite{cm21,cm22,cm24,cm23,cm35,cm25}, Meeks and Rosenberg
~\cite{mr13} and those in our papers \cite{mpr8,mpr11,mpr3,mpr4,mpr10},
we shall prove here that if a complete, embedded minimal surface of
finite genus in $\rth$ has 
a countable number of limit ends, then it is properly
embedded in $\rth$ (see Theorem~\ref{thm1.2} below). By the main result of
Collin, Kusner, Meeks and Rosenberg in~\cite{ckmr1}, any properly
embedded minimal surface in $\rth$ must have a countable number of
ends, even if it does not have finite genus. More generally, the
results in~\cite{ckmr1} imply that a properly embedded minimal
surface with compact boundary in $\rth$ can have at most two limit ends, and that
if it has empty boundary and two limit ends, then it is recurrent.

Our first key partial result on Conjecture~\ref{conj} is the following theorem, which  is the main
result in Section~\ref{sec3} (see Remark~\ref{properness-rem}).

\begin{theorem} \label{proper1end} Let $M\subset \rth$ be a complete embedded minimal surface
of finite genus with compact
boundary and exactly one limit
end. 
Then $M$ is properly embedded in $\rth$.
\end{theorem}

More generally, we have the following extension of the above result,
which is proved in Section~\ref{sec5}.

\begin{theorem}
\label{thm1.2}
Suppose $M \subset \rth$ is a complete, connected, embedded minimal surface
of finite genus, an infinite number of ends and compact boundary
(possibly empty).  Then:
\begin{enumerate}
\item $M$ has at most two simple limit ends.
\item $M$ has exactly one or two limit ends if and only if $M$ is proper in $\rth$.
\item Suppose $M$ has a countable number of limit ends. Then:
\begin{description}
\item[{\it 3-A.}] $M$ has one or two limit ends.
\item[{\it 3-B.}] $M$ is proper in $\rth$.
\item[{\it 3-C.}] If $M$ has two limit ends, then its annular ends are planar.
\item[{\it 3-D.}] If $\partial M=\mbox{\rm \O }$, then $M$ has exactly two
limit ends and  $M$ is recurrent for Brownian motion.
\item[{\it 3-E.}] If  $\partial M \neq \mbox{\rm \O }$,
then $\partial M$ has full harmonic measure.
\end{description}
\end{enumerate}
\end{theorem}

\begin{remark}
{\rm
In contrast to item~3-D of Theorem~\ref{thm1.2}, we note that Traizet~\cite{tra8}
has constructed a complete embedded minimal surface in $\rth$ of infinite genus,
with empty boundary, one limit end and infinitely many catenoidal type ends.
}
\end{remark}

The proof of Theorem~\ref{thm1.2}  depends on
Theorem~\ref{thm1.3} below, which
describes the geometry,
topology and conformal structure of certain representatives
of a simple limit end of genus zero for a complete embedded
minimal surface in $\R^3$; see Figure~\ref{figure1new} for a suggestive picture
describing the
key geometric features of such a representative.

Before stating Theorem~\ref{thm1.3}, we will need the following definition.

\begin{definition}
\label{def2.1}
{\rm
Let $E$ be a complete embedded minimal surface in $\rth$ with
nonempty compact boundary $\partial E$.
We define the {\it flux  vector of $E$} as
\begin{equation}
\label{FM}
F_E=\int _{\partial E}\eta \, \in \R^3,
\end{equation}
where $\eta $ denotes the inward pointing unit conormal vector to $E$ along $\partial E$.
}
\end{definition}

\begin{figure}
\begin{center}
\includegraphics[width=13cm]{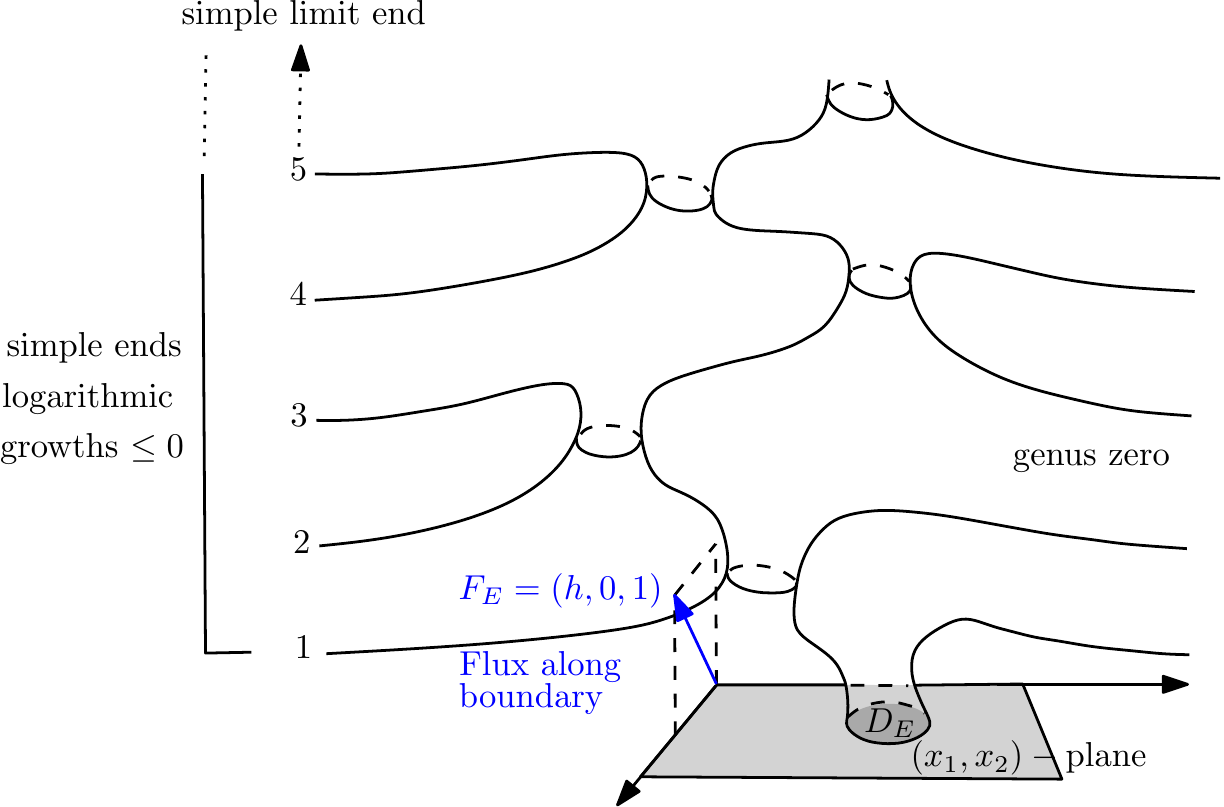}
\caption{A graphical representation of  the end representative $E$ in
Theorem~\ref{thm1.3}.} \label{figure1new}
\end{center}
\end{figure}

 \begin{theorem}
\label{thm1.3}
Suppose ${\bf e}$ is a simple limit end of  genus zero of a complete, connected, embedded
minimal surface $M \subset \rth$  with compact (possibly empty) boundary. Then
${\bf e}$ can be represented by a subdomain $E\subset \Int( M)$ that is properly
embedded in $\rth$ and, after a translation, rotation and homothety of $M$,
$E$ satisfies the following statements:
\begin{description}
\item[{\it 1.}] \label{thm1.3-1} The annular ends of $E$ have nonpositive logarithmic growths.
\item[{\it 2.}] \label{thm1.3-2} $E$ has genus zero and one limit end, which, in the
natural ordering of the ends of $E$ given by the
Ordering Theorem\footnote{Observe that the Ordering Theorem stated in~\cite{fme2}
also holds for properly embedded minimal surfaces in $\R^3$ with compact boundary.}
in~\cite{fme2}, is the top end of $E$.
\item[{\it 3.}]
\label{thm1.3-3} The boundary  $\partial E$ is a  simple closed curve
in the $(x_1, x_2)$-plane, and the flux vector $F_E$ of $E$ defined as in~(\ref{FM})
is $(h,0,1)$ for some $h>0$. Furthermore, $\partial E$ bounds
a convex disk $D_E \subset \{ x_3=0\} $
whose interior is disjoint from $E$, see Figure~\ref{figure1new}.
\item[{\it 4.}]
\label{thm1.3-4} There exists an orientation preserving
diffeomorphism $f \colon \rth \rightarrow \rth$ such that $f ({\cal R}_+) = E$,
where ${\cal R}_+$ is the top half of a Riemann minimal
example\footnote{See~\cite{mpr1, mpr6} for a
discussion of the singly-periodic, genus-zero, Riemann minimal
examples.} with boundary circle in the  $(x_1, x_2)$-plane.
\item[{\it 5.}]
\label{thm1.3-4'5}
$E$ has bounded Gaussian curvature.
\item[{\it 6.}]
\label{thm1.3-5} $E$ is conformally diffeomorphic to the closed punctured
disk $\{ z\in \C \ | \ 0<|z|\leq 1\}$ minus a sequence of points converging to
$0$. In particular, $\partial E$ has full harmonic
measure.
\end{description}
\end{theorem}

\begin{remark}
{\rm
If $M\subset \R^3$ is a properly embedded minimal surface of
finite genus and infinite topology, then $M$ has exactly two limit
ends $e_{-\infty },e_{\infty }$ which are simple limit ends of genus
zero, and which admit representatives $E_{-\infty },E_{\infty }$
that satisfy the conclusions of Theorem~\ref{thm1.3},
see~\cite{mpr3,mpr4}.
In this case where $M$ has no boundary, then
Theorem~8.1 in~\cite{mpr6} implies that the asymptotic behavior
of each of its two limit ends can be
described by the geometry of the ends of a Riemann minimal example.
}
\end{remark}

Crucial ingredients in the proof of Theorem~\ref{thm1.3} are the Limit Lamination Closure Theorem
(Theorem~1 in~\cite{mr13}),  the Local Picture Theorem on the Scale of Topology
(Theorem~1.1 in~\cite{mpr14}) 
and Theorem~2.2 in~\cite{mpr8} on the structure of certain possibly
singular lamination limits of certain sequences of minimal surfaces in $\rth$.
These ingredients, as well as many arguments in this paper, rely heavily on Colding-Minicozzi theory.

Theorems~\ref{thm1.2} and~\ref{thm1.3} not only play an important theoretical role in
our strategy  to prove Conjecture~\ref{conj},
but they also have important consequences for properly embedded minimal
surfaces, such as the one given in the next corollary; this corollary
follows from the more general result Corollary~\ref{corolnew2}.

\begin{corollary}
\label{corolnew}
If $M\subset \R^3$ is a properly embedded minimal
surface with a limit end of genus zero,
then $M$ is recurrent for Brownian motion.
\end{corollary}

Some of the results in this paper were announced by the authors at a conference in Paris in 2004.
Our proofs use results of Colding-Minicozzi theory that led us to develop a detailed study of minimal 
laminations with singularities and subsequent applications. The present paper can be considered 
as a culmination of a long term project in the understanding of complete embedded minimal surfaces of
finite genus in $\R^3$.

\section{Preliminaries on the  ends of a noncompact surface.}
\label{secprelim}
Given $p\in \R^3$ and $R>0$, we denote by $\B
(p,R)$ the open ball centered at $p$ of radius $R$. When $p=\vec{0}$, we let
$\B (R)=\B (\vec{0},R)$. If $\Sigma
\subset \R^3$ is a surface and $p\in \Sigma $, then $K_{\Sigma
},d_{\Sigma }, I_{\Sigma }, B_{\Sigma }(p,R)$ and $T_p\Sigma $
respectively stand for the Gaussian curvature function of $\Sigma $,
its intrinsic distance function, its injectivity radius function,
the intrinsic ball centered at $p$ of radius $R>0$ and the tangent plane to $\Sigma $ at $p$.
Also, $\overline{\D}=\{ z\in \C \ : \ |z| \leq 1\}$ stands for the closed unit disk.

{\bf Throughout the paper, $M\subset \R^3$ will denote a connected, complete
embedded minimal surface with compact boundary (possibly empty).}

We next recall the notion of end of $M$. Consider the set
\[
{\mathcal A}=\{ \a \colon [0,\infty )\to M\ \mid \ \a \mbox{
 is a proper arc}\} .
 \]
In ${\cal A}$, we define the equivalence relation $\a _1\sim \a _2$
if for every compact set $C \subset M$,  $\a
_1,\a _2$ lie eventually (outside a compact subset of the parameter domain $[0,\infty )$)
in
the same component of $M-C$.

\begin{definition}
\label{defend}
{\rm Each equivalence class in ${\mathcal E}(M)= {\mathcal A}/_\sim
$ is called an {\it end} of $M$. If ${\bf e}\in {\mathcal E}(M)$, $\a \in
{\bf e}$ is a  proper arc and $E\subset M$ is a proper connected
subdomain with  compact boundary such that $\alpha ([t_0,\infty ))\subset E$ for
some $t_0\geq 0$, then we
say that $E$ {\it represents} the end ${\bf e}$.}
\end{definition}

The space
${\mathcal E}(M)$ has the following natural Hausdorff topology.
  For each proper subdomain
$E\subset M$ with compact boundary, we define the basis open
set $B(E) \subset {\mathcal E} (M)$ to be those equivalence
classes in ${\mathcal E}(M)$ which have representative proper arcs contained in
$E$. With this topology, ${\mathcal E}(M)$ is a
 totally disconnected compact space which embeds topologically
 as a subspace of $[0,1]\subset \R $ (see pages 288-289 of \cite{mpe1} for a proof of this property).
In the sequel, we will
view ${\cal E}(M)$ as a subset of $[0,1]$ endowed with the induced metric topology.

Note that every {\it simple end} ${\bf x}$ of $M$ (i.e., ${\bf x}$ is an isolated point of ${\cal E}(M)$)
with genus zero can be represented by a proper annulus $E_{\bf x} \subset M$ which is homeomorphic to
$\esf^1 \times [0, \infty)$.

Next consider a {\em simple limit end}  ${\bf e}$ of $M$, i.e., there exists a
neighborhood $O({\bf e}) \subset {\cal E}(M)$ such that $ O({\bf e}) - \{ {\bf e} \}$ consists of
simple ends and ${\bf e}$ is a limit point of a sequence of
simple ends $\{ {\bf x}_n \}_n
\subset O({\bf e}) - \{ {\bf e} \}$. Suppose that the  simple limit end ${\bf e}$ has genus zero,
i.e., ${\bf e}$ admits a representative of genus zero. By the
 classification of genus-zero surfaces and after a possible replacement by a smaller
neighborhood $O({\bf e})$ of ${\bf e}$ in
${\cal E}(M)$, there
  exists a proper subdomain $E$ of $M$ satisfying:
\begin{enumerate}
\item[(A1)] $E$ is diffeomorphic to $\overline{\D}(\ast)=\overline{\D}
- [\left\{ 0\} \cup \{ \frac{1}{2n}\right\} _{n\in \N}]$. Furthermore,
$\partial E\cap \partial M=\mbox{\O }$.
\item[(A2)] $E$ represents all the ends in $O({\bf e})$, and the equivalence
class under $\sim $ of every proper arc
in $E$ represents a unique end in $O({\bf e})$.
\end{enumerate}

\section{Simple limit ends of genus zero can be represented by
properly embedded surfaces.}
\label{sec3}

We begin this section with several key notions that are closely tied
to obtaining properness results for minimal surfaces, including   Theorem~\ref{proper1end}
which will be proved here.

\begin{definition}
\label{defLSC}
{\rm
\begin{enumerate}[1.]
\item An embedded surface with boundary (possibly empty)
$\Sigma \subset \R^3$ is said to be {\it locally simply connected in $\rth$} if
 for every $p\in \R^3$, there exists $r=r(p)>0$
such that 
the closure of each  component of $\Sigma\cap \B (p,r)$ that is
disjoint from $\partial \Sigma$, is a compact disk with boundary in
$\partial \B (p,r)$. $\Sigma $ has {\it locally positive injectivity radius
away from $\partial \Sigma $} if
 for every $p\in \R^3$, there exists $r=r(p)>0$
such that 
the injectivity radius function $I_{\Sigma }$ of $\Sigma$ is bounded away from zero
on the union of  the components of $\Sigma\cap \overline{\B} (p,r)$ that are
disjoint from $\partial \Sigma$.

\item Let $A\subset \R^3$ be an open set and $\{ \Sigma _n\} _{n\in \N}\subset \R^3$
be a sequence of embedded surfaces
 (possibly with boundary). The sequence $\{ \Sigma_n\}_n$ is called
{\it locally simply connected} in $A$ if
for every $p\in A$, there exists $r=r(p)>0$ such that
$\B (p,r)\subset A$ and for $n$ sufficiently large,  $\B (p,r)$ intersects $\Sigma _n$
in components that are disks with boundaries in $\partial \B (p,r)$.
$\{ \Sigma _n\} _n$ is said to have
{\em locally positive injectivity radius in $A$,} if for every
$p\in A$, there exists $\ve _p>0$ and $n_p\in \N $ such that for
$n>n_p$, the restricted functions $(I_{\Sigma _n})|_{\Sigma _n\cap \B (p,\ve_p)}$ are
uniformly bounded away from zero.
\end{enumerate}
}
\end{definition}

\begin{remark}
\label{remark3.2}
{\rm
With the notation of item 2 of Definition~\ref{defLSC},
if the surfaces $\Sigma _n$ have nonempty boundaries and $\{ \Sigma _n\} _n$ has locally
positive injectivity radius in $A$, then for any $p\in A$
there exists $\varepsilon _p>0$ and $n_p\in \N$ such that $\partial
\Sigma _n\cap \B (p,\ve _p)=\mbox{\rm \O}$ for  $n>n_p$, i.e., points in the boundary
of $\Sigma_n$ must eventually diverge in space or converge to a subset
of~$\R^3-A$.
}
\end{remark}

By Proposition~1.1 in~\cite{cm35}, if $M\subset \R^3$
is an embedded minimal surface, then the property that  $M$ is locally simply
connected in $\R^3$ is equivalent to the property
that $M$ has locally positive injectivity radius away from
$\partial M$. The same proposition gives that a sequence
 of embedded minimal surfaces $\{M_n\}_n$ has locally positive
injectivity radius in an open set $A\subset \R^3$ if and only if $\{M_n\}_n$ is
locally simply connected in $A$.

Theorem~2 in~\cite{mr13} implies that if an  embedded, complete,
nonflat minimal surface in $\rth$ (with empty boundary) has positive injectivity
radius, then it is proper. Although not stated explicitly in~\cite{mr13},
the following result is an immediate consequence of the proof of Theorem~2
in~\cite{mr13} and other arguments therein.

\begin{theorem}
\label{thmmr}
Let $M\subset \R^3$ be  a complete, connected, 
embedded minimal surface with compact boundary.
If the injectivity radius function $I_M$ of $M$ is bounded away from zero outside of 
some intrinsic $\ve$-neighborhood  of $\partial M$, then $M$ is proper in $\rth$.
Furthermore, if $M$ has finite topology, then
$I_M$ is bounded away from zero outside of some intrinsic
$\ve$-neighborhood of
$\partial M$, and so, $M$ is proper in $\rth$.
\end{theorem}

\begin{remark}
\label{properness-rem}
{\em Theorem~\ref{propos3.4} below
implies the main properness statement in Theorem~\ref{thm1.3}.
It also
implies Theorem~\ref{proper1end} by the following reasoning.  Suppose that
$M\subset \rth$ is a complete embedded minimal surface
of finite genus with compact
boundary and exactly one limit
end  ${\bf e}$, which must therefore be a simple limit end.
If  $E$ is the proper representative
of  ${\bf e}$  given in the next
theorem, then the surface $ M-\Int(E)$ has finite topology and
must therefore be proper by Theorem~\ref{thmmr};
hence, $M=E\cup ( M-\Int(E))$ is also proper in $\rth$. }
\end{remark}

\begin{theorem}
\label{propos3.4}
Suppose that
$M\subset \rth$ is a complete embedded minimal surface
with possibly empty compact
boundary. 
Every simple limit end ${\bf e}\in {\cal E}(M)$ of genus zero
can be represented by a subdomain $E\subset M$ with compact boundary whose injectivity
radius is bounded away from zero outside each compact neighborhood of its boundary. In
particular, $E$ is proper in $\R^3$.
\end{theorem}
\begin{proof}
Let ${\bf e}\in {\cal E}(M)$ be a simple limit end of genus zero. Consider a proper
subdomain $E\subset M$ satisfying properties (A1) and (A2) stated in the preliminaries section.
With a slight abuse of notation, we identify
 $E$ with the parameter domain $\overline{\D}(\ast)=
\overline{\D}- [\left\{ 0\} \cup \{ \frac{1}{2n}\right\} _{n\in \N}]$, see property (A1).
The proof of Theorem~\ref{propos3.4} will be divided into several statements; more
precisely, Lemmas~\ref{flux} and~\ref{pg1} and Propositions~\ref{2limit} and~\ref{cat}.
As the proof develops, we will replace $E$ by similar proper subdomains of $E$ and
$O({\bf e})$ by the open subset of ends of the replaced $E$,  but will continue  to label these objects
by the same letters.

We first deal with the (simple) annular ends in $E$.
For $n \in \N$, let $\esf^{1}_{n}$ denote the
circle of center $0\in \C$ and radius $\frac{1}
{2n+1}$. Let $E_n$ be the proper subdomain of $E$ bounded by $\partial E\cup
\esf^{1}_{n}$.
 Since $E_n$ has finite topology and compact boundary,
 then Theorem~\ref{thmmr} applied to $E_n$ insures that $E_n$ is proper in $\R^3$.
As each of the (finitely many) ends of $E_n$ is an annular end, then
 Collin's theorem~\cite{col1} implies that each end of $E_n$ has
finite total curvature and is asymptotic to an
 end of a plane or catenoid. After a
rigid motion in $\R^3$, we may assume that:
\begin{enumerate}[(B1)]
\item \label{B1}
The annular ends of $E$ are represented by graphs over their projections to
 $\{ x_3=0\} $ with logarithmic growth (which is zero when the end is
asymptotic to the end of a plane).
\end{enumerate}

The proof of Theorem~\ref{propos3.4} is by contradiction. Hence {\bf assume there is no
end representative $E$ of ${\bf e}$ which is a proper surface}.
By Theorem~\ref{thmmr}, the injectivity radius function $I_E$ of every such
a representative has the property that $I_E$ fails to be bounded away from zero
outside some small $\ve$-neighborhood of $\partial E$.
Therefore, there exists a sequence of points $q_n\in E$ such that
$d_E(q_n,\partial E)$ is bounded away from zero and $I_E(q_n)\to 0$
as $n\to \infty $. Clearly, the $q_n$ diverge in $E$. As $I_E$ becomes unbounded
when approaching each of the simple ends of $E$, we deduce that the $q_n$ converge
to the origin when viewed inside $\ov{\D}(*)$.

Since $E$ has genus zero, then
the Local Picture Theorem on the Scale of Topology
(see Theorem~1.1, Proposition~4.20 and Remark~4.32 in~\cite{mpr14})
implies that we can find a divergent sequence of points $p_n\in E$ (called {\it points of almost minimal
injectivity radius for $E$}) and positive numbers $\ve _n \to 0$, such that
$d_E(p_n,q_n)\to 0$ as $n\to \infty $ and:
\begin{enumerate}[(C1)]
\item \label{D1}
 The closure $M_n$ of the component of $\B (p_n,\ve _n)\cap E$ that contains $p_n$ is compact with
boundary $\partial M_n\subset \partial \B (p_n,\ve _n)$. Furthermore, $M_n$ is disjoint from $\partial E$
for $n$ large enough (this follows from the fact that $p_n$ is divergent in $E$).

\item  \label{D2}
Let $\l_n=1/I_{M_n}(p_n)$, where $I_{M_n}$ denotes the injectivity radius function of
$E$ restricted to $M_n$. Then, $\l _nI_{M_n}\geq 1-\frac{1}{n}$ in $M_n$
and $\l _n\ve _n \to \infty$.
\end{enumerate}
Furthermore, exactly one of the following two cases occurs after extracting a subsequence.
\begin{enumerate}
\item[(C3)] \label{D3}
The surfaces $\l_n(M_n -p_n)$ have
uniformly bounded Gaussian curvature
on compact subsets of $\R^3$. In this case, there exists a connected, properly
embedded minimal surface $M_{\infty}\subset \R^3$
with $\vec{0}\in M_{\infty }$, $I_{M_{\infty}}\geq 1$
and $I_{M_{\infty}}(\vec{0})=1$,
 such that for any
$k \in \N$, the surfaces  $\l_n(M_n -p_n)$ converge $C^k$
on compact subsets of $\rth$  to $M_{\infty}$ with
multiplicity one as $n \to \infty$.

\item[(C4)] \label{D4}
After possibly a rotation in $\R^3$, the surfaces
$\l_n(M_n -p_n)$ converge to a minimal parking
garage structure\footnote{
We refer the reader to Section~3 in~\cite{mpr14}
for the definition of parking garage structure of $\R^3$.} of
$\R^3$ consisting of a foliation ${\cal F}$ of $\R^3$ by horizontal planes, with
two columns $l_1,l_2$
such that the associated highly sheeted, double multivalued graphs forming
in  $\l_n(M_n -p_n)$ around $l_1,l_2$
for $n$ sufficiently large, are oppositely handed. Furthermore, after relabeling,
$l_1$ intersects $\overline{\B}(1)$ and $l_2$ is at distance 1 from $l_1$.
\end{enumerate}

In order to finish the proof of Theorem~\ref{propos3.4}, we must find a contradiction
in each of the Cases (C3), (C4) above.

Suppose first that Case~(C3) holds. As the $\l_n(M_n -p_n)$ all have genus zero, then
$M_{\infty }$ has genus zero as well. By classification results for properly embedded minimal
surfaces of genus zero (Collin~\cite{col1}, L\'opez-Ros~\cite{lor1},
Meeks-P\'erez-Ros~\cite{mpr6}), $M_\infty$ is a catenoid or a
Riemann minimal example. Let $\gamma \subset
M_\infty$ be the waist
  circle if $M_\infty$ is a catenoid, and in the case $M_\infty$
is a Riemann minimal example, then let $\gamma $
be a simple closed planar curve (actually a circle) which separates the
  two limit ends of $M_\infty$.

\begin{remark}
\label{rem3.6}
{\rm
In the sequel, we will need the notion of {\em flux vector} of a minimal surface along a closed curve $\G $ once we have
chosen a unit conormal vector $\eta $  along $\G $; this flux is the vector in $\R^3$ given by the
integral of $\eta $ along $\G $, which clearly is defined up to a sign. This ambiguity still lets us make sense of when this
flux is nonzero, or when it is vertical.
}
\end{remark}

Since $\gamma$ has nonzero flux,  then for $n$
large, $\gamma$ is approximated
by the image by the composition of a translation by vector $-p_n$ with a homothety by $\l _n$ of a simple closed
planar curve $\gamma_n \subset M_n$
also with nonzero flux.

\begin{lemma}
\label{flux}
  $\gamma \subset M_\infty$ has vertical flux. Furthermore, after choosing a
subsequence, each curve $\gamma_n$ also has vertical flux.
\end{lemma}
\begin{proof}
It suffices to prove that for $n$ large, $\gamma_n \subset E$
has vertical flux. If the subdisk in $\overline{\D }$ bounded by $\g _n$ does not contain
$0 \in \overline{\D}$, then $\gamma_n$
is homologous to a finite number of loops around the annular ends of $E$,
and so, $\gamma_n$ has vertical flux by property (B1). Otherwise, after replacing by a subsequence,
we may assume that $\gamma_n$ is topologically
parallel to $\gamma_{n+k}$ and to $\partial E$ in $\overline{\D} - \{ 0 \}$
for $n,k \in \N$, $n$ large. Hence for any $k\in \N$, $\gamma_n$ is
homologous in $E$ to the union of $\gamma_{n+k}$ with a finite number of
loops around annular ends of $E$, and so, the flux
along $\gamma_n$ is equal to a vertical vector minus the
flux along $\gamma_{n+k}$. Since the flux along
$\gamma_{n+k}$ goes to zero as
$k \rightarrow \infty$ (because length$(\g _{n+k})\to 0$),
then  the flux of $E$ along $\gamma_n$ is vertical.
\end{proof}

\begin{proposition}
 \label{2limit}
  {$M_\infty$} is not a Riemann minimal example.
\end{proposition}
\begin{proof}
Arguing by contradiction, assume $M_{\infty }$ is a
Riemann minimal example. Since $\g \subset M_{\infty }$
has nonzero vertical flux by Lemma~\ref{flux},
then Theorem~6 in~\cite{mpr3} implies that the planar ends of $M_\infty$
are not horizontal.

Let $Q$ be the plane passing through the origin in $\rth$ that is parallel
to the planar ends of $M_\infty$ (equivalently, $Q$ is
the limit tangent plane at infinity of $M_\infty$).
Observe that planes parallel to $Q$
at heights (with respect to $Q$) different from the heights corresponding to the
planar ends of $M_\infty$, intersect $M_\infty$ transversely in simple
closed curves (actually in circles).
  Let $\Gamma _1, \Gamma _2, \Gamma _3, \Gamma _4$ be four such
circles 
on $M_\infty$, 
chosen so that the cycles $\Gamma_1 \cup \Gamma_2$,
$\Gamma_2 \cup \Gamma_3$ and $\Gamma_3 \cup \Gamma_4$ each bound a
noncompact subdomain $\Omega _{1,2},\Omega _{2,3},\Omega _{3,4}$ respectively of
$M_\infty$, each containing exactly two planar ends and such that
$\Omega _{1,2}\cap \Omega _{2,3}
=\Gamma_2$ and $\Omega _{2,3}\cap \Omega _{3,4}= \Gamma_3$,
see Figure~\ref{figure2} top. Observe that there exists a compact arc $c \colon [1,4]\to M_{\infty }$ such that
$c (i)\in \G _i$, $i=1,2,3,4$, and $c$ intersects transversely exactly once each of the curves $\G _i$.
\begin{figure}
\begin{center}
\includegraphics[width=11cm]{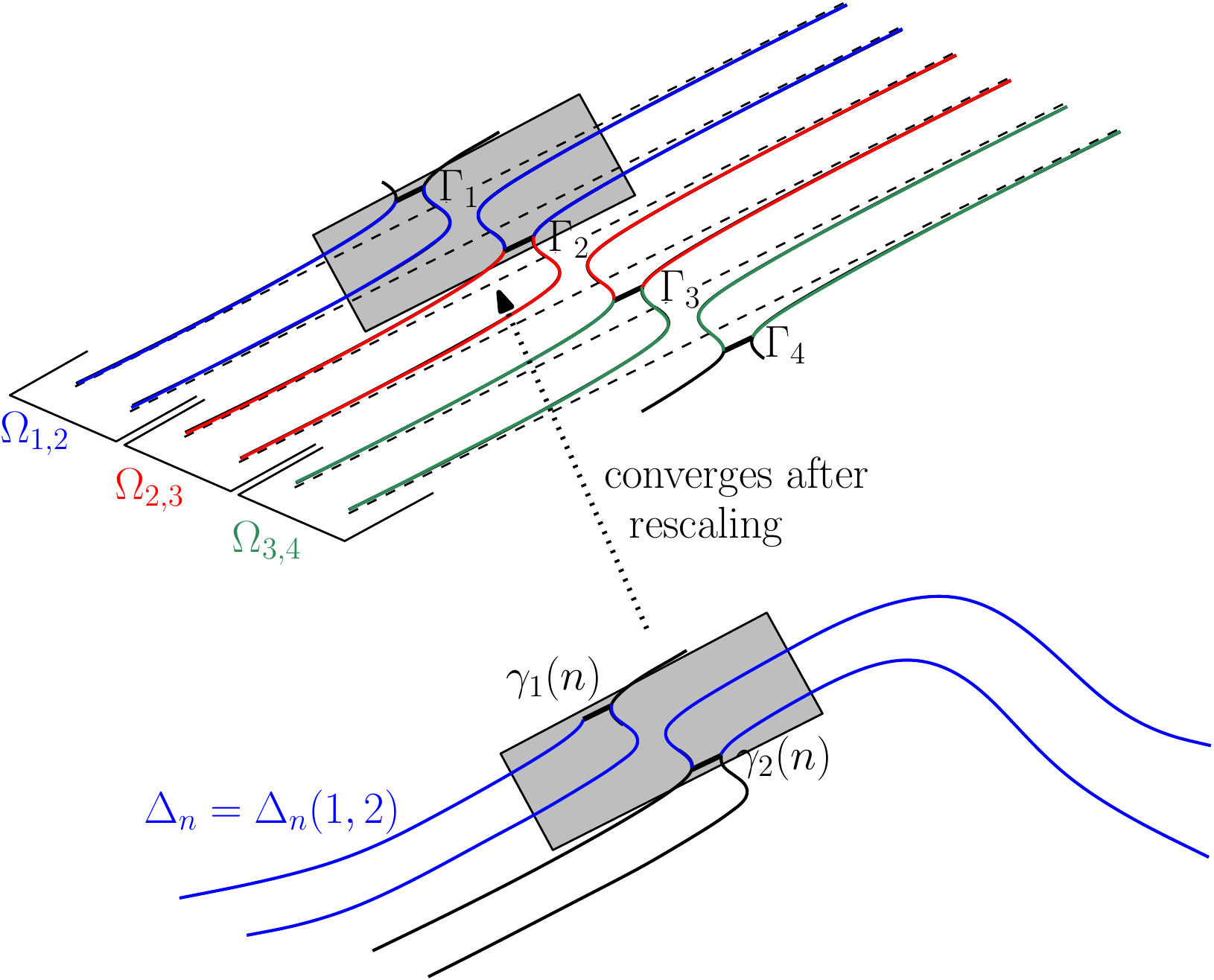}
\caption{Top: The (tilted) limit minimal example $M_{\infty }$. Below:
compact portions $\Delta_n$ of $E$ inside the shaded box which
converges after expanding to
a related compact portion of $M_{\infty }$.}
\label{figure2}
\end{center}
\end{figure}

For $n$ large, let $\gamma_1(n), \gamma_2(n), \gamma_3(n),
\gamma_4(n)$ be related simple closed planar curves in $M_n$, in the
sense that $\l _n(\g _i(n)-p_n)$ converges as $n\to \infty $ to $\G _i$,
$i=1,2,3,4$.  We can also assume that the $\g _i(n)$ are
contained in planes parallel
to $Q$. Similarly, $c$ is the limit of related compact arcs
$\l _n(c_n-p_n)$, where $c_n\colon  [1,4]\to E$ satisfies
$c_n(i)\in \g _i(n)$, $i=1,2,3,4$, and $c_n$ intersects exactly once each of the curves $\g _i(n)$.

To proceed with the proof of Proposition~\ref{2limit}, we
will need two assertions.

\begin{assertion}
\label{asser3.7}
After possibly reindexing, there is a domain
$\Delta_n \subset E$ of finite
topology such that $\partial \Delta_n = \gamma_1(n) \cup \gamma_2(n)$.
\end{assertion}
\begin{proof}
When considered to be curves in $\overline{\D }$, the closed
curves $\g _i(n)$ all separate $\overline{\D }$. Therefore,
$\overline{\D }-\cup _{i=1}^4\g _i(n)$ consists of five components. As the compact arc
$c_n([1,4])$ intersects exactly three of these five components
in open intervals of the form $c_n((j,j+1))$ ($j=1,2,3$),
then at least two of these components
are annuli disjoint from $\partial \overline{\D }$; of these two annuli, at least one,
called $A$, is disjoint from the limit end $0$ of $E$.
Hence, if we remove from $A$ the annular ends of $E$, then we obtain a planar domain
with finite topology, which we take as $\Delta _n$.
Now the assertion follows.
\end{proof}

For $i=1,2$, let $\gamma'_i(n) \subset
E- \Delta_n$ be planar curves (contained in planes parallel to
$Q$) close to and topologically parallel to $\gamma_i(n)$,
let $A_i(n) \subset E - \Delta_n$ be the open annulus with compact closure
bounded by $\gamma_i(n) \cup \gamma'_i(n)$,
and let $D_i(n), D'_i(n)\subset \R^3$ be the corresponding compact planar
disks bounded by $\gamma_i(n), \gamma'_i(n)$ respectively.
Finally, define $B_i(n)$ to be the
compact domain in $\rth$ with boundary $A_i(n) \cup D_i(n) \cup
D'_i(n)$, for $i = 1, 2$, see Figure~\ref{figure6}.
\begin{figure}
\begin{center}
\includegraphics[width=15cm]{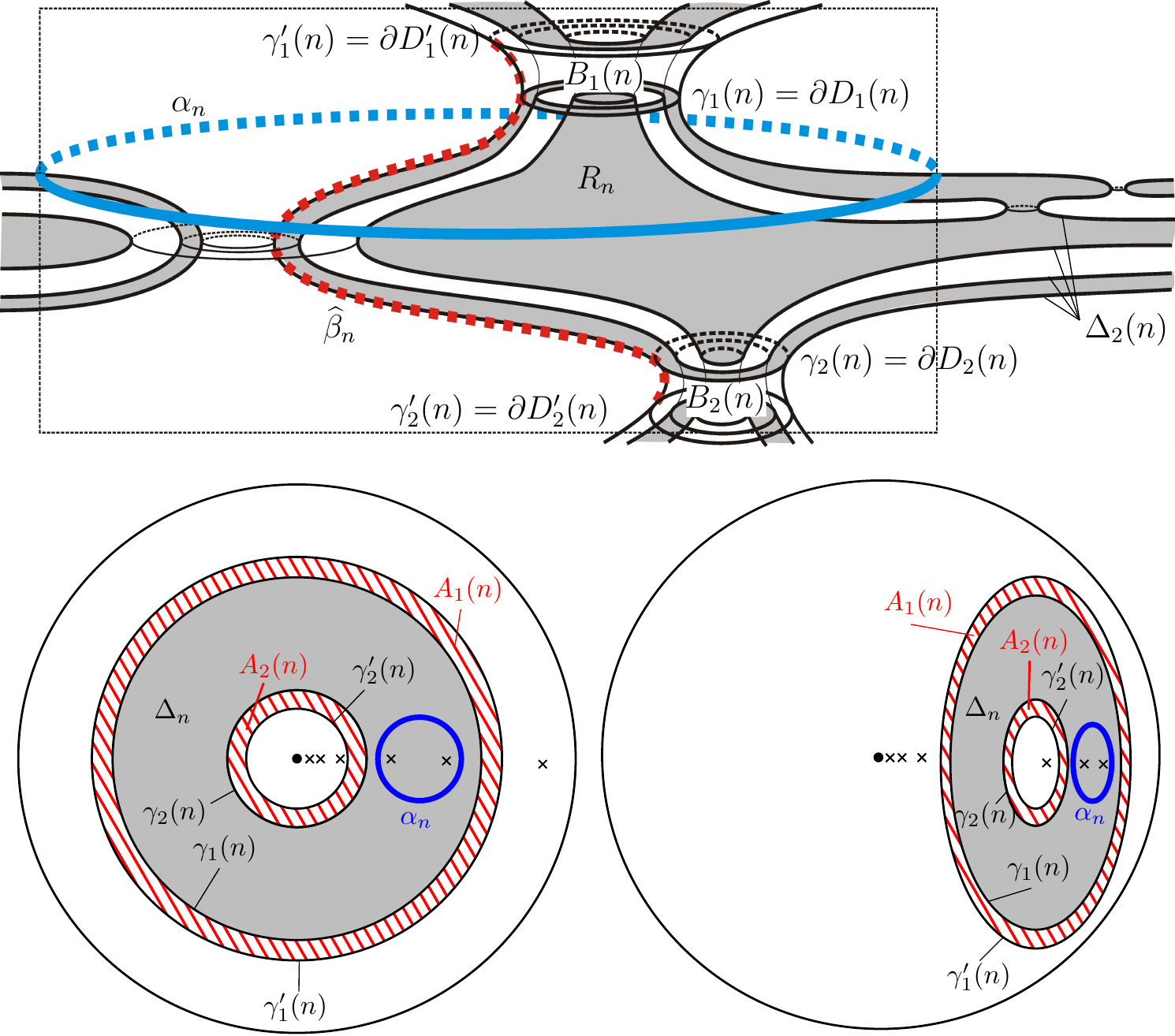}
\caption{Top: The portion of $E$ inside the dotted box
converges as $n\to \infty $ to a compact portion of a Riemann minimal example
$M_{\infty }$; note that the limit tangent plane of $M_{\infty }$
is represented as horizontal in the figure. Bottom: Two topological
configurations for the subdomain with finite topology $\Delta _n\subset E$,
depending on whether or not the curves $\g _i(n)$
wind around the limit end of $E$, $i=1,2$.
}
 \label{figure6}
\end{center}
\end{figure}

Since $\Delta_n$ has finite topology and compact boundary, then it is
properly embedded in $\rth$ (note that $\Delta _n$ is not compact
by the convex hull property). Without loss of generality, we may assume
that for $i= 1, 2$, each of the interiors of $D_i(n), D'_i(n)$ intersects
$\Delta_n$
transversely in a finite (possibly empty) collection of
simple closed curves (recall that
$M_{\infty }$ has been obtained as a limit after an intrinsic
blow-up procedure, rather than an extrinsic one).
Let $\Delta_1(n)$ denote the closure of the component of $\Delta_n \cap
\left[ \rth - (B_1(n) \cup B_2(n))\right] $ that contains $\partial
\Delta_n$. Since $X_n=\rth - (B_1(n) \cup B_2(n))$ is simply connected
and $\Delta_1(n)-\partial \Delta_1(n)$
is properly embedded in $X_n$, then
$\Delta_1(n)-\partial \Delta_1(n)$  separates $X_n$ into
two subdomains. Let
$F_i(n)$ denote the planar domain in $D_i(n)$ with boundary $\Delta_1(n)
\cap D_i(n)$ and let $F'_i(n)$ denote the planar domain in $D'_i(n)$ with
boundary $\Delta_1(n) \cap D'_i(n)$, for $i = 1, 2$. Hence,
\[
\Delta_2(n) =\Delta_1(n) \cup F_1(n) \cup F'_1(n)
\cup F_2(n) \cup F'_2(n)
\]
 is a properly embedded, piecewise smooth surface that
bounds an open region $R_n$ of $\rth$ such that the boundary of $R_n$ is a good
barrier for solving least-area problems in it (the smooth part of
$\partial R_n$ is minimal and the interior angles are convex); see Figure~\ref{figure6}.

Now choose a simple closed curve $\alpha \subset \Omega _{1,2}\subset M_{\infty }$
which bounds one of the
annular ends of $\Omega _{1,2}$. For $n$ sufficiently large, let $\alpha_n$ denote a related
simple closed curve on
$\Delta_n$ such that the $\l _n(\a _n-p_n)$
converge smoothly to $\a$ as $n\to \infty$ and
\[
\alpha_n\subset \mbox{Int}[\Delta _1(n)]\subset \Delta_2(n) =\partial R_n.
\]
We can also assume that the curves $\a$ and $\a_n$ are chosen so that
$\alpha_n \cup D_1'(n)\cup D'_2(n)$ lies on the boundary of its
convex hull.

\begin{assertion}
\label{asser3.8}
Consider a (possibly empty) collection ${\cal T}$ of closed curves
in $$\cup _{i=1}^2(D_i(n)\cup D'_i(n))\cap \Delta _1(n).$$ Then, for $n$ sufficiently large,
$\a _n\cup {\cal T}$ does not bound a
 compact minimal surface in
 the closure $\overline{R_n}$ of $R_n$.
\end{assertion}
\begin{proof}
Assume by contradiction that such a surface $S$ exists.
Since $S$ is compact and minimal, the convex hull property implies that $S$
is contained in the convex hull of its boundary $\partial S=\a _n\cup {\cal T}$;
thus $S$ lies in the closed slab containing the disks $D_1'(n)\cup D'_2(n)$.
Note that for $n$ large, there exists a path
$\be _n\subset \Delta_n - \alpha_n$ joining
$\g_1(n) $ to $\g_2(n)$. After adding two arcs $c_{n,1},c_{n,2}$ to $\be_n$ such that
$c_{n,i}\subset A_i(n)$, we obtain an embedded arc $\wh{\be}_n$ that joins
$\g_1'(n) $ to $\g_2'(n)$, see Figure~\ref{figure6} top.
Let $\wh{\be}'_n$ be a path parallel and
close to $\wh{\be}_n$, lying outside $\overline{R_n}$ in $\R^3$, and with the same
end points as $\wh{\be}_n$. Observe that $\wh{\be}'_n$ lies
the closed slab containing $D_1'(n)\cup D'_2(n)$.
Since $\wh{\be '}_n$ does not intersect $\overline{R_n}$, then $\wh{\be }'_n$
has zero intersection number with $S$. On the other hand, this
intersection number can be computed (mod~2) as the sum of the linking numbers (mod~2)
of $\wh{\be}'_n$ with the boundary components of $S$.
Since $\partial S$ can be assumed to lie  on the boundary of a convex body disjoint from the
end points of $\wh{\be}'_n$, $S$ is contained
in this convex body, and $\wh{\be }'_n$ does not link any of the
curves in ${\cal T}$ but it has linking number one with
$\a _n$, then $\wh{\be}'_n$ must have odd intersection number with $S$,
which is a contradiction. This proves Assertion~\ref{asser3.8}.
\end{proof}

Note that $\a _n$ separates the planar domain $\Delta _n$ into two closed
components, where one of these components $S_n'$ is a planar domain with $\partial S_n'=\a_n$,
see Figure~\ref{figure6}.
Also $\a _n$  separates the planar domain $\Delta _1(n)$ into two closed
components, where  one of these components $S_n$ satisfies $S_n\subset S_n'$. The boundary
of $S_n$ consists of $\a _n$ together with a collection ${\cal T}_n$  of closed planar curves
in $\cup _{i=1}^2(D_i(n)\cup D'_i(n))\cap \Delta _1(n)$. Let
$S_n(1)\subset S_n(2)\subset \ldots $ be a compact exhaustion of $S_n$ by smooth connected
subdomains with $\partial S_n\subset \partial S_n(1)$ and let $\wh{S_n}(k) $ be an area-minimizing
compact surface in $\ov{R_n}$ with $\partial \wh{S_n}(k) = \partial S_n(k) $
in the relative $\Z_2$-homology class of  $S_n(k) $, for all $k\in \N$; 
$\wh{S_n}(k)$ is orientable since either ${S_n}(k)\cup\wh{S_n}(k)$
is the piecewise-smooth boundary of a connected compact region of $\rth$, or else
${S_n}(k)\cup\wh{S_n}(k)$ is the union of some components of ${S_n}(k)$ and a
piecewise-smooth compact region of $\rth$.
A limit of some subsequence of $\{ \wh{S_n}(k) \} _k$
produces a properly embedded, oriented  stable minimal surface
$\wh{S_n}(\infty) \subset \overline{R_n}$ with boundary $\a _n\cup {\cal T}_n$, see~\cite{msy1} for
these standard arguments.
By Assertion~\ref{asser3.8},
$\a _n\cup {\cal T}_n$ does not bound a compact minimal surface in $\overline{R_n}$.
Therefore, the component $S_n(\infty)$ of $\wh{S_n}(\infty)$ which contains $\a _n$ is noncompact.
Let $G_n\colon S_n(\infty)\to \esf^2(1)$
be the Gauss map of $S_n(\infty)$.

Given $n\in \N$, consider the dilation (i.e., the composition of a translation and a homothety) 
$f_n(x)=\lambda_n (x -p_n)$, $x\in \R^3$.
Let $R_0>0$, $n_0\in \N$ be sufficiently large so that the following properties hold
for all $n\geq n_0$:
\ben[1.]
\item $f_n(\partial \Delta _1(n))\subset \ov{\B}(R_0)$ and, without loss of generality,
we may assume that the closed curves
$f_n(\a_n )\subset f_n(\Delta _1(n))\cap \partial\B(R_0)$
converge to $\a =M_\infty\cap \partial\B(R_0)$ as $n\to \infty$.
\item  There exists an increasing sequence of numbers  $R_n>R_0$ that diverge to infinity and such that
for every $n\in \N$, the component $\Sigma_n$ of $f_n(\Delta _1(n))\cap [\B(R_n)-\B(R_0)]$
that contains $f_n(\a _n)$ is a graph over its projection to the plane $Q$, and the $\Sigma _n$
converge smoothly on compact sets of $\R^3$ as $n\to \infty $ to the  annular end of $M_\infty$ bounded by $\a$.
\een
Note that $f_n(S_n(\infty))\cap [\B (R_n)-\ov{\B}(R_0)]$ is either contained in $\Sigma _n$, or it is
disjoint from $\Sigma _n$. By curvature estimates for stable minimal surfaces and  after choosing a subsequence,
the surfaces
\[
\frac{1}{\sqrt{R_n}}\left( \left[ \Sigma_n \cup \left( f_n(S_n(\infty)\right) \right]\cap \left[
\B (R_n)-\ov{\B}(R_0)\right] \right)
\]
converge to a minimal lamination $\cL$ of $\rth-\{\vec{0}\}$ with quadratic
decay of curvature, which contains the leaf $Q -\{\vec{0}\}$. By the
Local Removable Singularity Theorem (Theorem~1.1 in~\cite{mpr10}),
$\cL$ extends to a minimal lamination $\ov{\cL}$ of $\R^3$ with quadratic decay of curvature. As $\ov{\cL}$
contains $Q$,  Corollary~6.3 in~\cite{mpr10} implies that all leaves
of $\ov{\cL}$ are flat, and hence, they are planes parallel to $Q$.

Back to the scale of $E$, consider the compact subdomain
$S'_n(\infty)=S_n(\infty)\cap f_n^{-1}(\overline{\B}(\sqrt{R_n}))$ of $S_n(\infty)$. Then,
the normal
lines to the boundary of $S_n(\infty)-S'_n(\infty)$ make arbitrarily small angles
with the normal line to the plane $Q$ for $n$ sufficiently large. Pick a component $K_n$ of
$S_n(\infty)-S'_n(\infty)$ that intersects the boundary of $S_n(\infty)-S'_n(\infty)$.
Since $K_n$ is stable with finite total curvature~\cite{fi1}, then, for $n$ sufficiently large,
the Gaussian image of $K_n$ must be  arbitrarily close to
one of the two unit normal vectors $\pm V_Q$ to $Q$,
considered to be points of $\esf^2(1)$. As $K_n$ lies in the closure of a complement of $E$, then
property (B1) implies that the planar and catenoid-type ends of $K_n$
have limiting Gaussian images contained in the set $\{(0,0,\pm 1)\}\subset\esf^2(1)$.
But this last set is a positive
distance from $\{ \pm V_Q\} $, which is a contradiction.
This contradiction completes the proof of Proposition~\ref{2limit}.
\end{proof}

  \begin{lemma}
\label{pg1}
Case~(C4) cannot occur.
\end{lemma}
\begin{proof}
Since Case~(C4) forces the surfaces $\l _n(M_n-p_n)$
to have the appearance for $n$ large of a properly embedded, minimal planar domain with
two limit ends, large curvature and fixed
size ``horizontal'' flux\footnote{By ``horizontal flux'' we mean the nonzero component of the flux vector
of $\l _n(M_n-p_n)$ that is parallel to the planes
of the limit parking garage structure
associated to Case~(C4).} (see Traizet and Weber~\cite{tw1}, or~\cite{mpe2,mpr14}),
the proof of this lemma follows from a straightforward
adaptation of the proof  of Proposition~\ref{2limit}.
\end{proof}

\begin{proposition}
\label{cat}
$M_\infty$ is not a 
catenoid.
\end{proposition}
\begin{proof}
Reasoning by contradiction, assume $M_{\infty }$ is a catenoid.
By Lemma~\ref{flux}, $M_\infty$ has a vertical axis and
  the simple closed curves $\gamma_n \subset E$ defined just before Lemma~\ref{flux}
can be chosen to be
horizontal convex curves with vertical flux.
For $n$ large, we can choose
a compact unstable annulus $C_n\subset E$ with
$\gamma _n\subset \Int(C_n)$ so that $C_n$ is
arbitrarily close to a rescaling of a fixed, large, compact
unstable piece $C$ of a vertical catenoid. We may also assume that $\partial C_n$
consists of two convex curves in horizontal planes.
Let $D_n\subset \R^3$ denote the open convex
horizontal disk with $\partial D_n=\g_n$.

There are three different possible
 topological configurations for $\gamma_n
$ in $E$, after choosing a subsequence (see Figure~\ref{fig8}).
\begin{figure}
\begin{center}
\includegraphics[width=12.5cm]{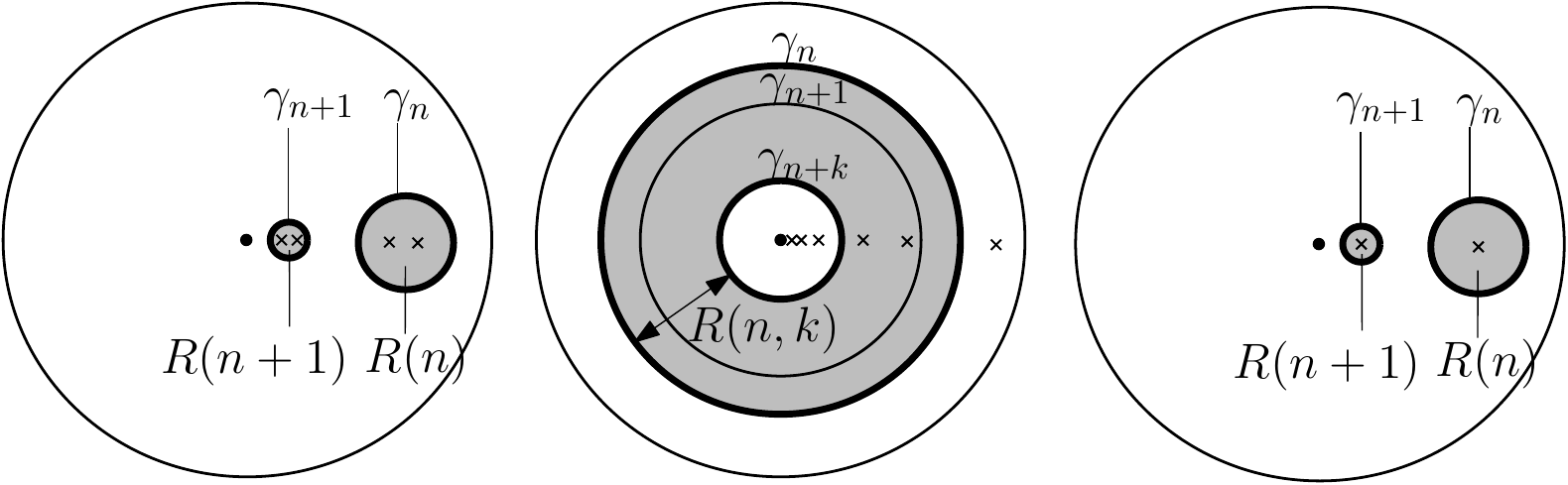}
\caption{Cases (D1) (left), (D2) (center) and (D3) (right) for Proposition~\ref{cat}.}
\label{fig8}
\end{center}
\end{figure}

\begin{enumerate}[(D1)]
\item Each $\gamma_n$ is
the boundary of a proper subdomain $R(n) \subset E$
with a finite number of annular ends greater than $1$.
\item When considered to lie in
$\overline{\D }- \{ 0 \}$, each
$\gamma_n$ is
homologous to $\partial E$. Hence, for $k$
large, the annular domain $R(n, k) \subset E$ bounded by $\gamma_n \cup \gamma_{n+k}$
has a finite positive number of
annular ends.
\item Each $\gamma_n$ bounds a proper annulus $R(n) \subset E$ with
 $\partial R (n) = \gamma_n$.
\end{enumerate}
The proof of Proposition~\ref{cat} will be a case-by-case elimination of each of
these three possibilities (for $n$ sufficiently large).

{\bf We first check that Case~(D1) does not occur.}
 In this case, $\gamma_n$ bounds
a proper, finite topology domain $R(n) \subset E$ with more than one end
and vertical flux.

\begin{assertion}
\label{asser3.11}
The open planar disks $D_1(n),D_2(n)\subset \R^3$
bounded by the curves in $\partial C_n$, are disjoint from $R(n)$.
\end{assertion}
\begin{proof}
If not, the proper surface $R(n)$ intersects the compact region
$W_n\subset \R^3$ bounded by
$C_n\cup D_1(n)\cup D_2(n)$ in a compact component
$\Omega (n)$ with boundary in $D_1(n)\cup D_2(n)$. Observe that
$\Omega (n)$ intersects both $D_{1}(n)$ and $D_{2}(n)$ by
the maximum principle for minimal surfaces.
Let $\widehat{W}_n$ be the closure of the component of $W_n-\Omega (n)$
that contains $C_n$ in its boundary. Since $\partial D_1(n)\cup \partial D_2(n)$
bounds the annulus $C_n$ in $\widehat{W}_n$
and $\partial D_1(n)$ is homotopically nontrivial in $\widehat{W}_n$, then
the Geometric Dehn Lemma
for Planar Domains in Theorem 5 in~\cite{my1} (as adapted in the more general boundary
setting of~\cite{my2}) implies that
$\partial D_1(n)\cup \partial D_2(n)$ is the boundary
of an embedded, least-area minimal annulus in $\widehat{W}_n$. But $\partial
D_1(n)\cup \partial D_2(n)$ also
bounds a stable minimal annulus in the
outer side of $C_n$, since $C_n$ is a good
barrier  that is an unstable minimal annulus. This
contradicts Theorem~1.1 in~\cite{mw1} which states that a pair of
convex curves
in parallel planes can bound at most one compact stable minimal annulus.
This contradiction proves Assertion~\ref{asser3.11}.
\end{proof}

Once we know that $D_i(n)\cap R(n)=\mbox{\O }$
for $i=1,2$, then $\g_n=R(n)\cap \ov{D_n}$,
which implies that
$R(n)\cup D_n$ is a properly embedded surface in $\R^3$.
Hence, $R(n)\cup D_n$ separates
$\R^3$ into two components. In this situation, for $n$ large the standard
L\'opez-Ros argument can be applied to $R(n)$
(since it is a complete embedded minimal surface with finite total
curvature, vertical flux and convex planar boundary
which is the boundary of an open convex planar disk  disjoint from the surface,
see Theorem~2 in~\cite{pro1} for a similar
argument), to conclude that $R(n)$ is an annulus. {\bf Thus,
 Case~(D1) does not occur.}
\vspace{.2cm}

We will use the following property when ruling out Cases (D2) and (D3).
\begin{assertion}
\label{ass3.14a}
Suppose after choosing a subsequence, that $\{ p_n\} _n$
converges to some point $p_{\infty }\in \R ^3$ and Case~(D3) holds for $\g _n$ for all $n\in
\N$. Then, the horizontal plane $L(p_{\infty })
\subset \R^3$ passing through $p_{\infty }$ satisfies that
 $E\cap L(p_{\infty })=\mbox{\rm \O}$, after removing any small compact neighborhood
of $\partial E$.
\end{assertion}
\begin{proof}
Since we are in Case~(D3), then $\g _n$ bounds a proper annulus $R(n)\subset E$.
After replacing
$\g _n$ by one of the boundary curves of the almost perfectly formed catenoid $C_n$, we
have that the new annulus $R(n)\subset E$ with $\partial R(n)$ the replaced boundary curve,
is disjoint from $\Int(C_n)$, and thus, we can assume that the
total absolute curvature of $R(n)$ is arbitrarily small for $n$ sufficiently large. Since the Gauss map
of $R(n)$ is open, almost vertical along $\partial R(n)$ (by Lemma~\ref{flux}),
the image of this Gauss map
has a limiting value $ (0,0,\pm1)$ at the end of $R(n)$, and the spherical image of the Gauss map of $R(n)$
is arbitrarily small, then we deduce that $R(n)$ is the graph of
a function defined on the projection of $R(n)$ to the $(x_1,x_2)$-plane,
and this graph has arbitrarily small gradient.

As we can assume that $\gamma_n \rightarrow p_{\infty }$
as $n \rightarrow \infty$, it follows that
the graphical annuli $R(n)$ converge smoothly away
from $p_\infty$ to the horizontal plane $L(p_{\infty })$
passing through $p_{\infty }$. To finish the proof of the assertion,
it only remains to show that $\Int (E)\cap L(p_{\infty })=\mbox{\rm \O}$.
Arguing by contradiction,
suppose that $L(p_{\infty })$ intersects $E$ at an interior point.
Since $L(p_{\infty })$ is not contained in $E$, then $L(p_{\infty })$ intersects
$E$ transversely at some interior point of $E$. This implies that for $n$
sufficiently large, $R(n)$ intersects $E-R(n)$,
which is impossible since $E$ is embedded. Now the assertion is proved.
\end{proof}

{\bf
We next check that Case~(D2) does not occur for $n$ large.} Arguing by
contradiction, assume that $n$ is large and (D2) holds.
Notice that for $n$ fixed and for $k\geq 1$, the proper subdomains $R(n, k)$  bounded by $\gamma_n \cup
\gamma_{n+k}$ give rise to  an proper exhaustion
of the representative of the limit end of $E$ whose boundary is  $\g_n$.
Rather than choosing $\gamma_n$ near the waist circle of
the forming unstable compact catenoid piece $C_n$, we choose $\g _n$
to be a curve contained in a horizontal plane at a height so
that for each $k$, the (noncompact) subdomain
$R(n, k) \subset E$ contains two unstable, pairwise disjoint, compact almost-catenoidal pieces,
also denoted by $C_n$, $C_{n+k}$, near $\gamma_n$ and
$\g _{n+k}$ respectively, so that $C_n$ is an annular neighborhood of
$\g _n$ (resp. $C_{n+k}$ is a neighborhood of
$\g _{n+k}$) in the new proper domain $R(n,k)$.
We may assume that both boundary curves of $C_n$ and of $C_{n+k}$ are convex horizontal curves
for all $k$. Also, $n$ can be chosen so that for all $k$ sufficiently large, the
almost-catenoid $C_{n+k}$ is much smaller than the scale of the
almost-catenoid $C_n$, see Figure~\ref{figure5}.
\begin{figure}
\begin{center}
\includegraphics[width=14cm]{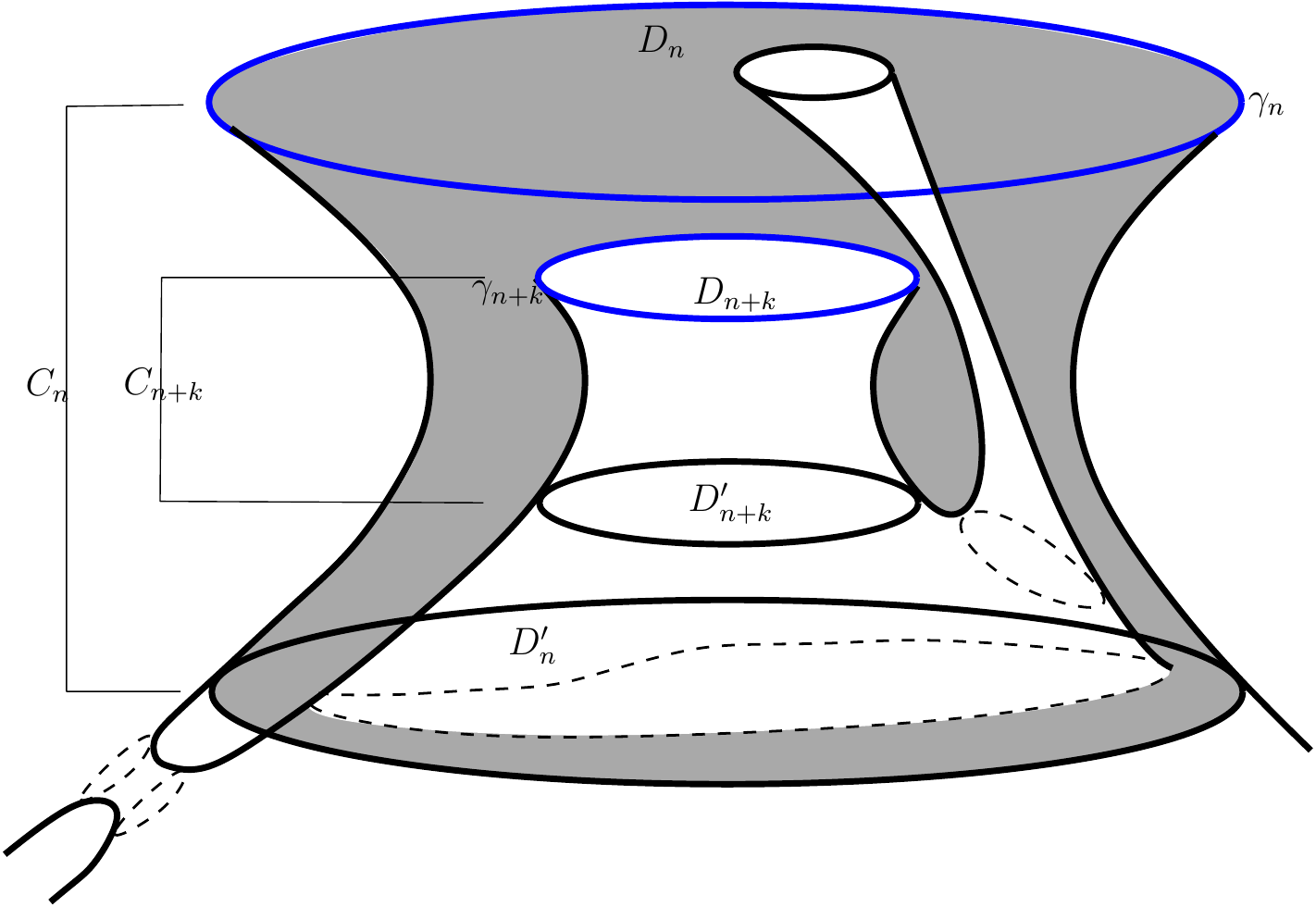}
\caption{
The boundary curve $\g _{n+k}$ of $R(n,k)$
must be contained in the compact region
$W_n\subset  \R^3$ bounded by $C_{n}\cup D_{n}\cup D'_{n}$.
}
\label{figure5}
\end{center}
\end{figure}

Given $n\in \N$, let $D'_n\subset \R ^3$ be the horizontal open disk
bounded by $\partial C_{n}-\g _{n}$ (recall that $D_n$ is the horizontal open disk
bounded by $\g _{n}$). 
We next analyze the intersection of $R(n,k)$ with $D_n,D'_n,D_{n+k},D'_{n+k}$.

\begin{enumerate}
\item[(D2-a)] We may assume that $D_{n+k},D'_{n+k}$
are disjoint from $C_n$ (because the scale of $C_{n+k}$
is much smaller than the scale of $C_n$, and both $C_n,C_{n+k}$ are inside $E$ which
is an embedded surface).

\item[(D2-b)] An analogous reasoning as in the proof of Assertion~\ref{asser3.11}
shows that both $D_{n+k}$, $D'_{n+k}$ are disjoint from $R(n,k)$.
Observe that the boundary curve $\g _{n+k}$
must be contained in the compact region
$W_n\subset  \R^3$ bounded by $C_{n}\cup D_{n}\cup D'_{n}$ as in Figure~5
(otherwise the arguments in the proof of Assertion~\ref{asser3.11} lead to a contradiction).
\end{enumerate}
The maximum principle and the fact that the scale of $C_{n+k}$ is much smaller
than the scale of $C_n$ imply  that $C_{n+k}$ is contained in the interior of $W_n$.
Therefore, the topological balls $W_n$ can be assumed to be concentric, in the following sense:
\begin{enumerate}[$(\star )$]
\item After replacing by a
subsequence and re-indexing, $W_{n+1}\subset \Int (W_n)$.
\end{enumerate}
Since the scales of the catenoids $C_{n}$ are converging to  zero as $n\to \infty $,
Property $(\star )$ implies that the $W_n$ converge to a point
$c_{\infty }\in \R^3$, which satisfies
$\{ c_{\infty }\} =\bigcap _{n\in \N }W_n\subset \mbox{Int}(W_1)$.
Without loss of generality, we may assume that $\partial E\cap W_1=\mbox{\O}$.

We next prove that the surface $E(W_1):=E\cap [\Int(W_1)-\{c_\infty\}]$ has
locally positive injectivity radius in $\Int(W_1)-\{c_\infty\}$.  Otherwise, there is a point
$q\in \Int(W_1)-\{c_\infty\}$ and a sequence of points $q_j\in E(W_1)$, $j\in \N$,
of almost minimal injectivity radius
for $E(W_1)$ in the sense of the Local Picture Theorem on the Scale of Topology,
that diverge in $E(W_1)$ but converge to $q$ as $j\to \infty $. After blowing up
$E(W_1)$ around the points $q_j$ on the scale of the injectivity radius, we find a limit which is
a catenoid (i.e., the other possibilities given by the
Local Picture Theorem on the Scale of Topology are not possible by the
arguments in Proposition~\ref{2limit} and Lemma~\ref{pg1}).
In particular, the catenoid which is forming nearby $q_j$ inside $E(W_1)$ for $j$ large,
is of one of the types (D1),
(D2) or (D3); in this case we will simply say that Case~(D1), (D2) or (D3) holds for $q_j$.
Case~(D1) for $q_j$ is not possible by our previous arguments based on
Assertion~\ref{asser3.11} and the L\'opez-Ros deformation. Also observe that Case~(D2)
cannot occur at $q_j$ for $j$ large, because the $q_j$ are converging to $q
\neq c_{\infty }$, which implies that $q_j$ does not lie in $W_n$ for $n$ large but fixed,
in contradiction with Property~$(\star )$. This implies that for $j$ large, Case~(D3) holds
for $q_j$. Since the $q_j$ converge to $q$ and Case~(D3) holds for $q_j$ for every $j$,
then Assertion~\ref{ass3.14a} insures that the horizontal plane $L(q)$ passing through $q$
in disjoint from $E$ after removing any small compact neighborhood of $\partial E$.
This is impossible, since $L(q)$ intersects $C_1$. This contradiction proves that
$E(W_1)$ has
locally positive injectivity radius in $\Int(W_1)-\{c_\infty\}$.

Since $E(W_1)$ has locally positive injectivity
radius in $\Int(W_1)-\{c_\infty\}$, Remark~2 in~\cite{mr13} ensures that
the closure of $E(W_1)$ in $\Int(W_1)-\{c_\infty\}$ is a
minimal lamination  $\cL$
of $\Int(W_1)-\{c_\infty\}$ that contains $E(W_1)$ as a subcollection of leaves.

We next prove that $\cL$ has no limit leaves in some neighborhood of $c_{\infty }$.
Otherwise, the sublamination $\cL'$ of limit leaves of $\cL$ is not empty,
and $\cL'$ consists of stable leaves by Theorem~1 in~\cite{mpr18}. By
Corollary~7.1 in~\cite{mpr10}, $\cL'$ extends across $c_{\infty }$
to a lamination of $\Int(W_1)$.
Thus, there exists a stable minimal surface $L_1\subset \Int (W_1)$
passing through $c_{\infty }$
such that $L_1-\{ c_{\infty }\}$ is a leaf of $\cL'$. Since $L_1$ is stable
and $C_n$ is unstable, then $L_1$ is disjoint from $C_n$ for all $n\geq 2$.
Therefore, for $\ve >0$ small enough, the
ball $\B (c_{\infty },\ve )$ of center $c_{\infty }$ and radius $\ve $ intersects
$L_1$ in a component $\Omega _1$ which is a disk that separates $\B (c_{\infty },\ve )$.
Take $n\in \N$ large enough so that $W_n\subset \B (c_{\infty },\ve )$, which
exists since $\{ c_{\infty }\} =\bigcap _{n\in \N }W_n$.
As $\Omega _1$ contains $c_{\infty }\in \Int (W_n)$ but $\Omega _1\cap
C_n=\mbox{\O }$ and $W_n\cap \partial \Omega _1=\mbox{\O }$,
then $\Omega _1\cap (D_n\cup D'_n)$ is nonempty. Without loss of generality,
we may assume that  $\Omega _1$ intersects $D_n\cup D'_n$ transversely and so,
there exists a simple closed curve $\be $
in $\Omega _1\cap (D_n\cup D'_n)$. This contradicts
 the maximum principle applied to the subdisk of  $\Omega _1$ bounded by $\be $.
This contradiction proves that $\cL$ has no limit leaves in some neighborhood of $c_{\infty }$.

Since $\cL$ has no limit leaves in some neighborhood of $c_{\infty }$, we may assume
that in some small compact neighborhood $N$ of $c_\infty$ in $\R^3$,
$\cL\cap N=[E-\{c_\infty\}]\cap N$ and $[E-\{c_\infty\}]\cap N$
is a properly embedded minimal surface in $N-\{c_\infty\}$ of genus zero.
But properly embedded minimal surfaces of finite genus in
a punctured Riemannian ball extend smoothly across the puncture
(see for example, Corollary~2.7 in~\cite{mpr11}
for this minimal lamination extension result).
This is clearly not possible because the Gaussian curvature of $E$
is not bounded in any neighborhood of $c_\infty$.
{\bf This contradiction proves that Case~(D2) does not occur for $n$ large.}
\par
\vspace{.2cm}

 {\bf Finally we check that Case~(D3) does not occur}, which will finish the proof of
Proposition~\ref{cat}.  By Lemmas~\ref{flux}, \ref{pg1} and Proposition~\ref{2limit}
and from the previously considered cases, we may assume that
all local pictures $M_n$ of $E$ on the scale
of topology (defined by properties (C1)-\ldots -(C4))
produce, after blowing-up, limiting catenoids
with vertical axes, and the horizontal almost waist circles $\g _n\subset E$ are in Case~(D3)
for all $n\in \N $ (after passing to a subsequence).
Consider  the related sequences $\{ M_n\} _n$, $\{ \g _n\} _n$.
We can assume that for all $n$, $M_n$ contains a compact piece of an almost perfectly
formed unstable catenoid $C_n$ containing $\g_n$, where $C_n$ is a shrunken image of a
large compact portion of an
almost-catenoid whose boundary consists of simple closed convex horizontal planar curves.
Since we are in Case~(D3), then $\g _n$ bounds a proper annulus $R(n)\subset E$.
After replacing
$\g _n$ by one of the boundary curves of the almost perfectly formed catenoid $C_n$, we
have that the new annulus $R(n)\subset E$ bounded by $\g _n$ satisfies the
following properties (see the proof of Assertion~\ref{ass3.14a}):
\begin{enumerate}[(E1)]
\item $R(n)$ is the graph of a function defined on the projection of $R(n)$ to the
$(x_1,x_2)$-plane, and this graph has arbitrarily small gradient.
\item Length$(\g _n)\to 0$ as $n\to \infty $.
\end{enumerate}

We will next show that Assertion~\ref{asser3.11} holds in this new setting.
\begin{assertion}
\label{asser3.11new} After extracting a subsequence and possibly
 replacing $E$ by another end representative,
for every $n\in \N$, the open planar disks $D_1(n),D_2(n)\subset \R^3$
bounded by the curves in $\partial C_n$, are disjoint from $E$.
\end{assertion}
\begin{proof}
Let $W_n\subset \R^3$ be the compact region
bounded by $C_n\cup D_1(n)\cup D_2(n)$.
After choosing a subsequence and removing a small neighborhood of $\partial E$ from $E$,
we may assume that $W_n\cap \partial E =\mbox{\O}$.
Observe that $E\cap \Int (W_n)$ is
locally simply connected: otherwise, there exists
some point $p_{\infty }\in E\cap \Int(W_n)$
where Case~(D3) holds for $\g_m$ for all $m\in \N$ sufficiently large ($m$ larger than $n$);
in this case, Assertion~\ref{ass3.14a} ensures
that the horizontal plane $L(p_{\infty })$ passing through
$p_{\infty}$ is disjoint from $E$ after removing any compact neighborhood of $\partial E$, which is
impossible since $L(p_{\infty})\cap C_n\neq \mbox{\O}$. Thus,
$E\cap \Int (W_n)$ is locally simply connected.

The arguments in the previous paragraph and Assertion~\ref{ass3.14a} ensure
that there exists an open set $U\subset \R^3$ such that $C_n\subset U$ and
the restriction of the injectivity radius  function of $E$ to $E\cap U$ is bounded away from
zero. Therefore, the closure of $E\cap U$ relative to the open set $U$ is a minimal lamination of $U$.
As $E\cap \Int (W_n)$ is locally simply connected, the closure of $E\cap \Int (W_n)$
relative to $\Int (W_n)$ is a minimal lamination of $\Int (W_n)$. Consequently,
the closure of $E\cap [U\cup \Int (W_n)]$ is a minimal lamination of $U\cup \Int (W_n)$.
Since $C_n$ is unstable, then $C_n$ is not contained in a limit leaf of this lamination,
which implies that the distance from $C_n$ to the closure $\ov{E\cap \Int (W_n)}$
of $E\cap \Int (W_n)$ is positive.

As $C_n$ is unstable, we can find a compact unstable
subannulus $C_n'\subset \Int(C_n)$ such that
$\partial C'_n$ consists of two convex horizontal curves
that bound open planar disks $D'_1(n),D'_2(n)\subset \R^3$.
Let $W'_n\subset W_n$ be the compact region
bounded by $C'_n\cup D'_1(n)\cup D'_2(n)$. It follows from the previous
paragraph that the closure of $E\cap \Int (W_n)$ relative to $\Int (W_n)$
is a minimal lamination of $\Int(W_n)$, that is at a positive distance from $C_n$.
In particular, the closure of $E\cap \Int (W_n)$ relative to $\Int (W_n)$
intersected with $W'_n$ is a compact, possibly empty, set $X$ in $W_n'$.

Suppose the assertion fails for some $n$, that is, $E$ intersects $D_1(n)\cup D_2(n)$.
Then, $E\cap \Int(W_n)\neq \mbox{\O}$ and thus, we can assume
 $E\cap \Int(W'_n)\neq \mbox{\O}$ by choosing $C_n'$ sufficiently close to $C_n$.
 In particular, $X\neq \mbox{\O}$. As $X$ is a compact union of minimal surfaces in
 $W_n'$, then the maximum principle applied to $x_3$ gives that each component of
 $X$ intersects both disks $D_1'(n),D_2'(n)$.
 Since $X$ is a good barrier for solving Plateau type problems in $W'_n$, and
 $\partial C'_n$ does not bound minimal disks in $W_n'-X$, then there exists
a least area annulus $A'\subset W_n'$ with boundary $\partial A=\partial C'_n$.
This is impossible, by the same reasoning as  in the proof of Assertion~\ref{asser3.11}.
This completes the proof of Assertion~\ref{asser3.11new}.
\end{proof}

Arguing by contradiction, assume that Case~(D3) occurs for all $n$.
By our earlier considerations, there would exist an infinite collection
of pairwise-disjoint almost-catenoids $C_n$ forming on $E$ of the
type described in Case~(D3) and that satisfy the conclusions of
Assertion~\ref{asser3.11new}. Also, we can assume that
the logarithmic growths of the associated graphs $R(n)$ all have the same sign,
say negative.

Consider the piecewise smooth graphical planes $P_n=D_2(n)\cup R(n)$, where
$D_2(n)$ is the lower open disk given in Assertion~\ref{asser3.11new}.
Note that as $D_2(n)\cap E=\mbox{\O}$, then $E-R(n)$ is contained in
the component of $\R^3-P_n$ above $P_n$. It follows that the connected surface
$E-\cup_n R(n)$ must lie above each of the $P_n$. By elementary separation properties,
this situation is not possible as it would imply that $P_1$ lies above $P_2$ and $P_2$ lies above
$P_1$.  {\bf This contradiction completes the proof that Case~(D3) does not occur.}
So, Proposition~\ref{cat} is proved.
\end{proof}

By  Lemma~\ref{pg1}, Propositions~\ref{2limit}, \ref{cat}
 and the paragraph before Remark~\ref{rem3.6},
we conclude that the injectivity radius function $I_E$
 is bounded away from zero outside of some (and thus, every)
intrinsic $\ve$-neighborhood  of $\partial E$.
Therefore, Theorem~\ref{thmmr} insures that $E$ is properly embedded in $\R^3$,
which completes the proof of Theorem~\ref{propos3.4}.
\end{proof}

\section{The proof of Theorem~\ref{thm1.3}.}
\label{sec4}
Let ${\bf e}$ be a simple limit end of genus zero of a complete, embedded minimal
surface $M\subset \R^3$ with compact
boundary (possibly $\partial M=\mbox{\O}$).
By Theorem~\ref{propos3.4}, we can choose a representative $E$ of
${\bf e}$  such that $E$ is properly embedded in $\R^3$. The arguments
at the end of Section~\ref{secprelim} show
that after relabeling, properties (A1), (A2) hold for
$E$. As explained in the second paragraph of the proof of Theorem~\ref{propos3.4},
each  simple end of $E$ has and an annular end representative
with finite total curvature and is asymptotic to an
end of a plane or catenoid, which after a
fixed rotation of $M$ in $\R^3$, is a graph over
its projection to  the $(x_1,x_2)$-plane.
Since $E$ is properly embedded in $\R^3$, it
follows from the Ordering Theorem~\cite{fme2} and Theorem~1.1
in~\cite{ckmr1} that the limit end of $E$, after a possible
rotation by $\pi$ around the $x_1$-axis,
is the top end of $E$.


Lemma 3.6 in~\cite{ckmr1} implies that a limit end of a properly
embedded minimal surface with compact boundary in $\rth$ cannot
have a representative that lies
above the end of a catenoid with positive logarithmic growth. Therefore, since
the limit end of $E$ is its top end and the middle ends of $E$ are
asymptotic to planes and catenoidal ends, none of the catenoidal
ends in $E$ have positive logarithmic growth. This proves items~1 
 and~2 
 of Theorem~\ref{thm1.3}.

\begin{lemma}
\label{lemma3.15}
There exists a divergent sequence of points $q_{n} \in
E$ such that $\frac{I_{E}(q_{n})}{|q_{n}|} \rightarrow
0$ as $n \rightarrow \infty$, where $I_{E}$ is the injectivity radius
function of $E$.
\end{lemma}
\begin{proof}
Otherwise, there exists $c>0$ such that $I_{E}(\cdot )\geq c\, |\cdot |$ in $E$,
away from a compact neighborhood of $\partial E$.
Since $\partial E$ is compact, $E$ is properly
embedded and $E$ does not have finite total curvature, then Theorem~1.2
in~\cite{mpr10} implies that there exists a divergent sequence of
points $y_n \in E$ such that $K_{E} (y_n)  |y_n|^2 \rightarrow
-\infty$ as $n \rightarrow \infty$. Consider the sequence of positive numbers
$\sigma_n =\frac{1}{|y_n|}\to 0$. Since
\[
\frac{I_{\sigma _nE}(\sigma_nx)}{|\sigma _nx|}=\frac{I_{E}(x)}{|x|},
\]
we conclude that the sequence of surfaces $\{ \sigma _nE\} _n$ has locally positive
injectivity radius in the open set $\R^3-\{ \vec{0}\} $
in the sense of Definition~\ref{defLSC}, or
equivalently, the sequence of compact genus-zero minimal surfaces
$\{ (\sigma _nE)\cap \ov{\B}(n)\} _n $ is locally simply connected in $\R^3-\{ 0\} $,
see the first paragraph after Remark~\ref{remark3.2}. Since
the surfaces $(\sigma _nE)\cap \ov{\B}(n)$ have genus zero with compact
boundary and the
Gaussian curvature of $(\sigma _nE)\cap \ov{\B}(n)$ at the point $\sigma _ny_n\in
\partial \B (1)$ diverges as $n\to \infty $, then item~2 of Theorem~2.2
in~\cite{mpr8} implies  that
after passing to a subsequence, $\{ (\sigma _nE)\cap \ov{\B}(n)\} _n$ converges to a
minimal lamination ${\cal L}$  of $\R^3-\{ \vec{0}\} $, outside of a
nonempty singular set of convergence $S({\cal L})\subset {\cal L}$
(this is the closed subset of points
$x\in {\cal L}$ such that the supremum of the absolute Gaussian curvature of
$(\sigma _nE)\cap \B (x,\ve )$ is not bounded in $n$, for any $\ve >0$), and
the following property holds:
\begin{enumerate}[(F)]
\item The closure $\ov{\cL}$ of $\cL$ in $\R^3$ is a foliation of $\R^3$ by planes,
and the closure $\ov{S({\cL})}$ of $S({\cL})$ consists of one or two complete lines
orthogonal to the planes in $\ov{\cL}$.
\end{enumerate}

Since the limit end of $E$ is its top end and its annular ends are
catenoidal with nonpositive logarithmic growth, it follows that ${\cal L}$
is contained in the closed upper halfspace
$\{ x_3\geq 0\} $ minus the origin. This contradicts property (F) above,
which completes the proof of Lemma~\ref{lemma3.15}.
\end{proof}

Consider the divergent sequence $\{ q_n\}_n \subset E$ given by Lemma~\ref{lemma3.15}.
We next apply a similar rescale-by-topology argument as
as we did in the proof of Theorem~\ref{propos3.4}
just after property (B1), but instead of using the Local Picture Theorem on the
Scale of Topology as we did there,  we will use the following extrinsic argument.
Given $n\in \N$ large
so that the boundary of $E$ lies in $\B (|q_n|/2)$, consider
the continuous, nonnegative function $h_n\colon \ov{\B }(q_n,|q_n|/2)\cap E\to \R $ given by
\[
h_n(x)=\frac{\mbox{dist}_{\R^3}(x,\partial \B (q_n,|q_n|/2))}{I_{E}(x)}.
\]
$h_n$ vanishes at $\partial \ov{\B }(q_n,|q_n|/2)$. Let $p_n$ be a maximum of $h_n$.
Observe that
\[
h_n(p_n)\geq h_n(q_n)=\frac{|q_n|}{2I_E(q_n)}\to \infty ,
\]
and define
\[
r_n=\frac{1}{2}\mbox{dist}_{\R^3}\left( p_n,
\partial \B (q_n,|q_n|/2)\right)
=\frac{1}{2}h_n(p_n)I_E(p_n).
\]
Then, the sequence of embedded minimal surfaces
of genus zero and compact boundary
\begin{equation}
\label{eq:tildeEn}
\widetilde{E}_n = \lambda _n \left[E\cap \ov{\B}(p_n,r_n) -p_n\right]
\end{equation}
is uniformly locally simply connected in $\R^3$, where
$\l _n=1/I_{E}(p_n)$ (in fact, $\wt{E}_n$ has boundary in the sphere
centered at the origin with
radius $\frac{1}{2}h_n(p_n)\to \infty $ and the injectivity radius function of $\wt{E}_n$ is
at least $1/2$ at points at least at distance $1/2$ from its boundary). By Theorem~2.2 in~\cite{mpr8}
applied to this sequence of surfaces, we deduce that
there exists a minimal lamination $\cL$ of $\R^3$ and a closed subset $S({\cL})
\subset \cL $ such that $\{ \wt{E}_n\} _n$ converges $C^{\be}$, for all $\be
\in (0, 1)$, on compact subsets of $\R^3-S({\cL})$ to $\cL $; here $S({\cL})$
is the singular set of convergence of the $\wt{E}_n$ to $\cL$. Furthermore, exactly one
of the two following cases holds:
%
%

\begin{enumerate}[(G1)]
\item The surfaces $\widetilde{E}_n$ have uniformly bounded Gaussian curvature on compact
subsets of $\R^3$. In this case, $S({\cal L})=\mbox{\O }$ and either $\cL$ is a collection of
planes (this case cannot occur since the injectivity radius function of $\wt{E}_n$
at the origin is 1 for each $n\in \N$),
 or $\cL$ consists of a single leaf $M_{\infty }$, which is properly embedded in $\R^3$
with genus zero. Furthermore, in this last case $\wt{E}_n$ converges
smoothly on compact sets in $\R^3$ to $M_{\infty }$ with multiplicity one and exactly
one of the following three cases holds for $M_{\infty }$:
\begin{enumerate}[(a)]
\item $M_{\infty }$ has one end and it is asymptotic to a helicoid (in this case,
Theorem~0.1 in \cite{mr8} insures that $M_{\infty }$ is a helicoid). Again,
this case cannot occur
as the injectivity radius function of $\wt{E}_n$ at the origin is 1 for each $n\in \N$.
\item $M_{\infty }$ has nonzero finite total curvature. In this case, $M_{\infty }$ is a catenoid
by the main result in~\cite{lor1}.
\item $M_{\infty }$ has two limit ends. In this case, $M_{\infty }$ is a Riemann minimal example
by~\cite{mpr6}.
\end{enumerate}

 \item $\cL$ has the structure of a limiting parking garage in the following sense:
$\cL$ is a foliation of $\R^3$ by parallel planes and $S({\cL})$ consists of one or two
lines orthogonal to the planes in $\cL$ (called columns of the limiting parking garage
structure), and as $n\to \infty $, a pair of highly sheeted multivalued graphs
forms inside $\wt{E}_n$ around each of the lines in $S({\cL})$. Furthermore,
if $S({\cL})$ consists of two lines $l,l'$, then $l$ intersects $\overline{\B}(1)$,
$l'$ is at distance~1 from $l$ and the pairs of multivalued graphs inside the
$\wt{E}_n$ around different lines are oppositely handed. In fact, 
$S({\cL})$ cannot consist of a single line; a proof of
this property can be found by a direct adaptation of the second paragraph of the proof
of Lemma~3.4 in~\cite{mpr8}.
\end{enumerate}

\subsection{Finding horizontal planes $P_n$ and ``concentric'' curves $\wh{\G}(n)\subset E\cap P_n$.}
\begin{lemma}
\label{lemma4.2}
After possibly replacing $E$ by another end representative,
there exists a sequence $\{ P_n\} _{n\in \N\cup \{0\}}$ of
horizontal planes with $x_3(P_n)<x_3(P_{n+1})$ and $x_3(P_n)\to \infty $, such that
each $P_n$ intersects $E$ transversely and $P_n\cap E$ contains
a simple closed curve $\wh{\G}(n)$ with the following properties:
\begin{enumerate}
\item $\partial E=\wh{\G}(0)\subset P_0$.
\item When viewed in $\ov{\D}-\{ 0\} $, each $\wh{\G}(n)$ with $n\in \N$
is topologically parallel to $\partial E$.
\item Given $n\in \N$, let $\Omega _n\subset \ov{\D}(*)$ be the finite
topology subdomain whose boundary is $\wh{\G}(n)\cup \partial E$.
Then, $\Omega _n\subset \Omega _{n+1}$ for all $n$.
\item When viewed in $\R^3$, $\Omega _n$ lies below the plane $P_n$.
\item $E$ lies locally above $P_0$ along $\partial E$.
\item  If Case~(G1) occurs then:
\ben
\item For each $n\in \N\cup \{ 0\} $, $\wh{\G}(n)$ bounds a compact convex disk $D_n\subset P_n$ 
whose interior is disjoint from $E$. Furthermore, the
$D_n$ all lie in the same side of $E$.
\item The limit tangent plane at infinity of $M_{\infty }$ is horizontal.
\een
\item If Case~(G2) occurs, then the planes in the limit parking garage structure are horizontal.
\end{enumerate}
\end{lemma}
\begin{proof}
We first claim that if $P$ is a horizontal plane such that $\partial E\subset \{ x_3<x_3(P)\} $, then
$P\cap E$ contains exactly one compact component that is nonzero in $H_1(\ov{\D}-\{ 0\} )$ ($P\cap E$ might
contain infinitely many compact components that bound disks in $\ov{\D}-\{ 0\} $, each one containing finitely
 many annular ends of $E$). To see this, note that $P\cap E$ contains at least one compact component that is
 nonzero in $H_1(\ov{\D}-\{ 0\} )$ since $\partial E$ lies below $P$, the limit end of $E$ is its top end and $E$ is connected.
If $P\cap E$ contains two compact components both nonzero in $H_1(\ov{\D}-\{ 0\} )$, then we can choose
two of such components $\G,\G'$ satisfying that $\G\cup \G'$ is the boundary of
a compact annulus $A(\G,\G')\subset \ov{\D}-\{ 0\} $ such that  $\Int (A(\G,\G'))\cap x_3^{-1}(x_3(P))$ does not 
contain components which are nonzero in $H_1(\ov{\D}-\{ 0\} )$ and when
viewed in $\R^3$, $A(\G,\G')\cap E$ locally lies above $P$ along $\G\cup \G'$. Observe that $A(\G,\G')$ 
contains finitely many (annular)
ends of $E$, each of which has nonpositive logarithmic growth. Therefore,
$A(\G,\G')-x_3^{-1}(-\infty ,x_3(P))$ is a parabolic surface with 
boundary, and $x_3|_{A(\G,\G')-x_3^{-1}(-\infty ,x_3(P))}$ is a bounded nonconstant harmonic 
function with constant boundary values, which is impossible. This proves our claim.

{\bf Assume that Case~(G2) occurs for the limit of the
$\wt{E}_n$.} Recall that a limiting parking garage structure in $\R^3$ with two
oppositely handed vertical columns closely resembles geometrically and
topologically a Riemann minimal example with almost horizontal flux vector
and finite positive injectivity radius; we refer the reader to the
paper~\cite{mpr14} for further explanations.

 Let $l,l'$ be the straight lines which are the columns of
 the limiting parking garage structure, and let
$\widetilde{c}_n= \lambda_n(c_n-p_n)\subset
\widetilde{E}_n $ be a connection loop for
the forming parking garage structure; this means that
$\widetilde{c}_n$ is a closed curve, which
approximates arbitrarily well (for $n$ large enough) a path that
starts at a point in the first column, travels on one level of the
limiting parking garage to the second column, goes ``up'' one level
(remember that we do not know that the columns $l,l'$ are vertical) and then
travels back again on this level ``over'' the previous arc until arriving at the
first forming column, and then goes ``down'' one level until it closes up.

We claim that when viewed in $\ov{\D}-\{ 0\} $, $c_n$ cannot bound a disk; to see this,
note that if $c_n$ bounds a disk in $\ov{\D}-\{ 0\} $, then $c_n$ bounds a
finite topology domain $\Delta_n$ in $E$ with vertical flux. Since for $n$ large
the flux of $\widetilde{E}_n$ along $\wt{c}_n$ is arbitrarily close to a nonzero vector orthogonal to $l$,
we conclude that $l,l'$ are horizontal. This implies that there are points in the interior of $\Delta_n$
whose heights are strictly greater than the maximum height of $c_n$.
Since the ends of $\Delta_n$ are graphical with nonpositive logarithmic growth,
we find a contradiction with the maximum principle for $x_3|_{\Delta_n}$. Therefore, our claim holds.

We next prove that $l,l'$ are vertical lines. Pick a plane $\wt{P}$ in the limiting
parking garage structure, orthogonal to $l,l'$ and for $n$ large, let $P_n$ be a
plane such that $\l _n(P_n-p_n)$ converges to $\wt{P}$ as $n\to \infty$, such
that the height of $P_n$ does not coincide with the height of any planar end of
$E$. Choose two connection loops $c_n,c'_n\subset E$ lying at different sides of $P_n$. Since both
$c_n,c'_n$ are homologically nontrivial in $\ov{\D}-\{ 0\} $
by the last paragraph, then $c_n,c'_n$ are topologically parallel in
$\ov{\D}-\{ 0\} $ and thus, there exists an annular
domain $A(c_n,c'_n)\subset \ov{\D}-\{ 0\} $ bounded by $c_n\cup c'_n$.
Observe that we can choose $c_n,c'_n$ so that $A(c_n,c'_n)$
contains annular ends of $E$ (by the convex hull property).
If $l,l'$ were not vertical, then for $n$ large $A(c_n,c'_n)\cap E$
 would contain interior points
whose heights are strictly greater than the maximum height
of $c_n\cup c'_n$, which is a contradiction as in the previous paragraph.
Therefore, $l,l'$ are vertical lines, which proves item~7 of the lemma.

We continue assuming that Case~(G2) occurs. By Sard's theorem,
 we can assume that $P_n$ intersects transversely $E$.
Identifying $A(c_n,c'_n)\cap E$ with its image minimal surface in $\R^3$, we deduce that the intersection set
$A(c_n,c'_n)\cap x_3^{-1}(x_3(P_n))$
consists of a nonzero finite number of Jordan curves contained in the interior of
$A(c_n,c'_n)$. By elementary separation properties, there exists at least
one component $\wh{\G}(n)$ of
$A(c_n,c'_n)\cap x_3^{-1}(x_3(P_n))$ which is
topologically parallel to $c_n$ in $A(c_n,c'_n)$;
in fact, $\wh{\G}(n)$ is unique by the arguments in the first paragraph of this
proof. Thus, $\wh{\G}(n)\subset E$ satisfies item~2 of the lemma.

Note that the curves $\wh{\G}(n)$ can be chosen (after passing to a subsequence)
so that the finite topology domains $\Omega _n\subset \ov{\D}(*)$ bounded by
$\wh{\G}(n)\cup \partial E$ satisfy $\Omega _n\subset \Omega _{n+1}$
for all $n$, so item~3 of the lemma holds by construction.
Without loss of generality, we may assume that $c_n\subset \Omega _n$.
Given $n\in \N\cup \{ 0\}$ and $k\in \N$, the
annulus $A(\wh{\G}(n),\wh{\G}(n+k))\subset \ov{\D}-\{ 0\} $ bounded by
$\wh{\G}(n)\cup \wh{\G}(n+k)$ satisfies that $A(\wh{\G}(n),\wh{\G}(n+k))\cap E$
is a finitely punctured annulus and
$A(\wh{\G}(n),\wh{\G}(n+k))\cap E$ lies below the horizontal plane
at height $\max \{ x_3(\wh{\G}(n)),x_3(\wh{\G}(n+k))\} $
(by the maximum principle applied to $x_3|_{A(\wh{\G}(n),\wh{\G}(n+k))\cap E}$,
since the annular ends of $E$
have nonpositive logarithmic growth). As $E$ contains points
of arbitrarily large heights because the limit end of $E$
is its top end, we conclude that the heights of the planes $P_n$ are not
bounded from above. After passing to a subsequence,
we can assume that $x_3(P_n)<x_3(P_{n+1})$ and $x_3(P_n)\to \infty $
as $n\to \infty $. This implies that after replacing $E$ by a
representative of the same limit end bounded by the curve $\wh{\G}(0)$, we
can assume that item~1 of the lemma holds provided that Case~(G2) occurs.

Observe that the finite topology domain $\Omega _n$ equals $
A(\wh{\G}(0),\wh{\G}(n))$, hence item~4 holds by the last paragraph.
By transversality, this implies that $E-\Omega _n$ lies locally above
$P_n$ along $\wh{\G}(n)$. In particular, $E$ lies locally above
$P_0=\{ x_3=x_3(\partial E)\} $ along $\partial E$ and item 5 of the
lemma holds provided that Case~(G2) occurs. Thus, the proof
of Lemma~\ref{lemma4.2} is finished if Case~(G2) holds.

{\bf Next assume that Case~(G1) occurs for the limit of the
$\wt{E}_n$ with $M_{\infty}$ being a Riemann minimal example.} The previous arguments can be adapted to prove that:
\begin{itemize}
\item If $\widetilde{c}_n=\l_n(c_n-p_n)\subset \wt{E}_n$ converges to
a circle $C$ in the Riemann minimal example $M_{\infty }$, then $c_n$ winds
once around $0$ in $\overline{\D}-\{ 0\}$ (adapt the arguments in the fourth
paragraph of the present proof and use that if the flux of a Riemann minimal
example is vertical, then its planar ends are not horizontal).
\item The limit tangent plane at infinity for $M_{\infty }$ is vertical
(adapt the arguments in the fifth paragraph of the present proof).
\item There exists a sequence of horizontal planes $P_n$ such that $\{ \l_n(P_n-p_n\} \}_n$
converges to $\{ x_3=x_3(C)\} $, and compact components $\wh{\G}(n)$ of $E\cap P_n$ that
are Jordan curves which, when viewed in $\overline{\D}-\{ 0\}$, wind once around $0$ (adapt the arguments in the sixth paragraph above).
\item The finite topology domain $\Omega _n\subset \overline{\D}(\ast)$ bounded by
$\wh{\G}(n)\cup \partial E$ can be chosen so that $\Omega_n\subset \Omega_{n+1}$ for all $n\in \N$, and all of the remaining properties of
Lemma~\ref{lemma4.2} hold (follow verbatim the arguments in the seventh paragraph of this proof).
\end{itemize}

{\bf Finally suppose that Case~(G1) occurs for the limit of the
$\wt{E}_n$ with $M_{\infty}$ being a catenoid.}
Let $\wt{P},P_n\subset \R^3$ be parallel planes so that $\wt{P}$ intersects $M_{\infty }$ in its waist circle
$\wt{\g}$, and for each $n$ $P_n\cap E$ contains a convex Jordan curve $\g_n$ such that $\{ \l_n(\g_n-p_n)\} _n$
converges to $\wt{\g}$ as $n\to \infty$.
\begin{claim}
\label{claim4.3}
For $n$ sufficiently large, $\g_n$ is nonzero in $H_1(\ov{\D}-\{ 0\} )$.
\end{claim}
\begin{proof}
Assume that $\g_n$ bounds a disk $\Delta $
in $\ov{\D}-\{ 0\} $. By the convex hull property, $\Delta $ contains a finite positive number of ends of $E$, all of which
are annular with finite total curvature and vertical (possibly zero) flux. As $\g_n$ is
convex, a standard application of the L\'opez-Ros deformation argument shows that $\Delta $ contains exactly one end of
$E$. This annular end of $E$ has negative logarithmic growth for $n$ sufficiently large, as the flux of $M_{\infty }$ along
$\wt{\g}$ is nonzero. The same reason gives that $M_{\infty }$ is a vertical catenoid, and thus, $\wt{P},P_n$ are
horizontal planes. For $n$ sufficiently large,
consider a compact annular neighborhood $A(\g_n)$ of $\g_n$ in $E$ with the following properties:
\begin{enumerate}[(H1)]
\item $A(\g_n)$ is bounded by two compact, convex curves in horizontal planes and
the lower boundary curve of $A(\g_n)$ bounds an annular end $R(n)$ of $E$ of catenoidal type (with negative
logarithmic growth).
\item $A(\g_n)$ is unstable and the sequence $\l_n(A(\g_n)-p_n)$ converges smoothly with multiplicity one to a
large compact piece of $M_{\infty }$ containing $\wt{\g}$.
\end{enumerate}
Let $D_n$ (resp. $D_n'$) be the compact horizontal disk in $\R^3$ whose boundary is the lower (resp. upper)
boundary component of $A(\g_n)$. Thus, $\partial D_n=\partial R(n)$.
By the same arguments as in the proof of Assertion~\ref{asser3.11}, the compact region $W_n\subset \R^3$
bounded by $A(\g_n)\cup D_n\cup D_n'$, satisfies that $W_n\cap E=A(\g_n)$
(note that we can assume that $n$ is sufficiently large so that $\partial E$ does not intersect $W_n$).
As $E$ is connected and proper, we deduce that $E- R(n)$ is disjoint from the piecewise smooth,
properly embedded topological plane $ R(n)\cup D_n$. As the limit end of $E$ is its top end,
we deduce that $E- R(n)$ lies entirely above $ R(n)\cup D_n$. In particular, $ R(n)$ is the lowest end of $E$.
As this can only happen once for the $\g_n$, this proves Claim~\ref{claim4.3}.
\end{proof}

We continue assuming that Case~(G1) occurs with $M_{\infty}$ being a catenoid.
By Claim~\ref{claim4.3}, we can
assume that $\g_n$ is nonzero in $H_1(\ov{\D}-\{ 0\} )$ for each $n\in \N$. Let $\Omega _n\subset \ov{\D}(*)$
be the subdomain with finite topology and $\partial \Omega _n=\partial E\cap \g_n$. Adapting the arguments
in the fifth paragraph of this proof (with $\Omega _n$ instead of $A(c_n,c'_n)$) we conclude that the catenoid
$M_{\infty }$ is vertical, and thus, $\wt{P},P_n$ are horizontal planes. As for $n$ large we can assume that
$\partial E$ lies below $P_n$, the claim in the first paragraph of the proof of Lemma~\ref{lemma4.2} shows that $\g_n$ is the
unique compact component of $P_n\cap E$ that is nonzero in $H_1(\ov{\D}-\{ 0\} )$. We now
define $\wh{\G}(n):=\g_n$. Once here, items 1-6 in Lemma~\ref{lemma4.2} are easy to prove by direct adaptation
of the arguments in paragraphs six and seven above. We leave the details to the reader.
\end{proof}

For the remainder of this section, we will assume that $E$ satisfies the properties
stated in Lemma~\ref{lemma4.2}.

\begin{definition} \label{def4.3}
{\rm
Since $E$ is proper,
Theorem~3.1 in~\cite{ckmr1} implies that $(x_3|_E)^{-1}([t,\infty ))$ is a
parabolic manifold with boundary, i.e., it has full harmonic measure
on its boundary. In this situation, the
Algebraic Flux Lemma for parabolic manifolds (Meeks~\cite{me23}) ensures that
if we define
\begin{equation}
\label{VE}
V_{E}:=\int _{\{ x_3=t\} }\frac{\partial x_3}{\partial \eta } \, \in [0,\infty ],
\end{equation}
where $ \eta$ is the inward pointing conormal to $(x_3|_E)^{-1}([t,\infty))$, then $V_E$
is independent of $t\geq \max (x_3|_{\partial E})=x_3(P_0)$, where $P_0$ is the horizontal plane
defined in Lemma~\ref{lemma4.2}. We call $V_{E}$
the {\it vertical flux component} of $E$.
}
\end{definition}

In what follows, we will use the notation
\begin{equation}
\label{notation}
T=T_H+T_V
\end{equation}
for the decomposition of a vector $T\in \R^3$
in its horizontal and vertical components.

\begin{corollary}[Flux Estimates]
\label{fluxestimates}
Let $\Omega_n$ be the subdomains of $E$ defined in Lemma~\ref{lemma4.2},
and let $\be _n\in (-\infty,0]$ be the sum of the (nonpositive) logarithmic growths
of the simple ends of $\Omega_n$.  Let $\eta$ denote the outward
pointing conormal vector to $\Omega_n$ along $\wh{\G}(n)$ and define the associated flux vector
$$\mbox{F}(\wh{\G}(n)):=\int_{\wh{\G}(n)} \eta.$$
 Then, for each $n\in \N$:
\ben
\item \label{eq:Fn}
$F(\wh{\G}(n))=F_{E}-2\pi\beta_n e_3$,
where $F_E$ is the flux of $E$ given in~\eqref{FM} and $e_3=( 0,0,1)$.
\item   $F(\wh{\G}(n))_H=(F_E)_H.$
Furthermore, after a normalization of $E$ by replacing it by its image
under a rotation around the $x_3$-axis,  $F_E=(h,0,\tau)$ for some $h,\tau\in (0,\infty)$,
where $h=|(F_E)_H|$ and $\tau=|(F_E)_V|$.
\item Case~(G2) does not occur.
\item Let $\ds \be_{\infty }=\lim _{n\to \infty }\be _n\in [-\infty ,0]$. If  $\beta_{\infty }$  is finite,
then $V_E\ e_3 =(F_E)_V- 2\pi\beta_{\infty } e_3$, where $V_E$ is defined in (\ref{VE}).
\item Case~(G1-c) (i.e., $M_{\infty}$ is a Riemann minimal example) occurs if and only if
$\be_{\infty }$ is finite.  In this case,
$\ds \l_\infty=\lim_{n\to \infty}\l_n$ exists and is a positive number, and
$M_{\infty}$ is the scaled Riemann minimal example with horizontal limit tangent plane at infinity
that has injectivity radius 1 and  flux vector $\l_{\infty } (h,0,\tau-2\pi \be_{\infty })$.
\item Case~(G1-b) (i.e., $M_{\infty}$ is a catenoid)  occurs if and only
if  $\be_{\infty }=-\infty$ (equivalently, $\ds \lim_{n\to \infty}\l_n=0$).
\een
\end{corollary}
\begin{proof}
Item~1 follows from the divergence theorem applied to the
harmonic coordinate functions of $E$, using
the fact that the flux contributions for catenoidal ends of $\Omega_n$
are all vertical with negative logarithmic growth.

The first statement in item~2 follows from taking horizontal components in~item~1;
we next prove the second statement in item~2.
First suppose that Case~(G1) holds. By item (6-a) of Lemma~\ref{lemma4.2},
the boundary curves of $\Omega_n$ are convex planar curves that bound
horizontal disks $D_n$ whose interiors are disjoint from $E$, and
the $D_n$ all lie on the same side of $E$.
Since for $n$ large $\Omega_n$ is not an annulus, then if $(F_E)_H=0$,
then the L\'opez-Ros deformation argument applied to $\Omega_n$
would lead to a contradiction. The fact that $(F_E)_V\neq 0$ follows directly from the
maximum principle for $x_3$, since $E$ is not contained in a horizontal plane.
This proves the second statement in item~2
when Case~(G1) holds.  Thus item~2 will hold once we
prove item~3.

In the case that (G2) holds, the connection loop $c_{n}\subset \Omega_n$
(defined in the proof of Lemma~\ref{lemma4.2})
is homologous in $\ov{\D}-\{0\}$ to $\partial E$.
Since the planes in the limiting parking garage are horizontal by item~7 of
Lemma~\ref{lemma4.2}, then the ratio $\frac{|F(c_n)_V|}{|F(c_n)_H|}$
of the length of the vertical component $F(c_n)_V$ over the length of the
horizontal component $F(c_n)_H$ of the flux vector $F(c_n)$
converges to zero as $n\to \infty $.
As $F(c_n)_H=(F_E)_H$ by the divergence theorem, then
$|F(c_n)_V|$ tends  to $0$ as $n\to \infty$.  This is impossible, since
the arguments in obtaining item~\ref{eq:Fn} show that $|F(c_n)_V|\geq |(F_E)_V|>0$.
This contradiction gives that items~2 and 3 hold.

We next prove item~4. Taking vertical components in
the equality of item~1 and using that the limit $\be_{\infty }$ of the $\be_n$ is assumed to be finite,
we have that $\lim _nF(\wh{\G}(n))_V$ exists and equals $(F_{E})_V-2\pi\beta_{\infty} e_3$.
Hence it remains to show that
\begin{equation}
\label{eq:VE1}
V_E \ e_3 =\lim _nF(\wh{\G}(n))_V.
\end{equation}
To see this, we will describe $E\cap P_n$ for $n\in \N$ given. Observe that if $C$ is a noncompact component of $E\cap P_n$, then
$C$ is a noncompact embedded arc and each of the two ends of $C$ diverges to the same annular
end of $E$, which is therefore a planar end asymptotic to $P_n$. Hence, after moving slightly the height
of $P_n$, we can assume that every component of $E\cap P_n$ is compact.
Next consider a (compact) component $C$ of $E\cap P_n$. By item~4 of Lemma~\ref{lemma4.2},
$C\subset E-\mbox{Int}(\Omega_n)$. By the claim in the first paragraph of the proof of Lemma~\ref{lemma4.2},
either $C=\wh{\G}(n)$ or $C$ bounds a disk in $\ov{\D}-\{ 0\} $.

Assume that $E$ contains a planar annular end. By embeddedness of $E$, all annular ends above $E$ (with the
ordering given by the Ordering Theorem) must be also planar. After replacing $E$ by another end representative
of its limit end, we
can assume that all the ends of $E$ are planar. In this case, $E\cap P_n=\wh{\G}(n)$ (otherwise,
there exists a component $C$ of $E\cap P_n$ such that $C$ bounds a disk $\Delta_C$ in $\ov{\D}-\{ 0\} $
by the last paragraph, and we contradict the maximum principle applied to $x_3|_{\Delta_C}$
as all the (finitely many) ends of $E$ in $\Delta _C$ are planar). Since
$E\cap P_n=\wh{\G}(n)$ for each $n\in \N$, then  (\ref{VE}) computed for $t=x_3(P_n)$ gives that $V_E\ e_3=F(\wh{\G}(n))_V$
for each $n\in \N$, from where (\ref{eq:VE1}) follows directly.

By the arguments in the last paragraph, we can assume that all the annular ends of $E$ have negative logarithmic growth.
Fix $n\in \N$. As $E-\Omega _n$ lies locally above $P_n$ along $\wh{\G}(n)$ (by item~4 of Lemma~\ref{lemma4.2})
and every annular end of $E$ in $E-\Omega_n$ is represented by a punctured disk that lies entirely below $P_n$, then
we conclude that $E\cap P_n$ consists of $\wh{\G}(n)$ together with infinitely many compact components $C_i(n)$,
$i\in \N$, each of which bounds a disk $\Delta_{C_i(n)}$ in $\ov{\D}-\{ 0\} $ that contains a finite positive number of catenoidal type
ends of $E$. Therefore,  (\ref{VE}) computed for $t=x_3(P_n)$ gives that
\begin{equation}
\label{VE2}
V_E\ e_3=F(\wh{\G}(n))_V+\sum _{i\in \N}F(C_i(n)),
\end{equation}
where $F(C_i(n))$ is the (vertical) flux vector of $E$ along $\partial \Delta_{C_i(n)}$ computed with the unit conormal
vector that points outwards from $\Delta_{C_i(n)}$ along its boundary. Observe that given $n,i\in \N$, the divergence theorem
gives that $F(C_i(n))$ equals $e_3$ times a finite sum of logarithmic growths of annular ends of $E$. As the
sequence of domains $\{ \Omega_n\} _n$ forms an increasing exhaustion of $E$, then given $n,i\in \N$, there exists
$k\in \N$ sufficiently large so that all annular ends in $\Delta _{C_i(n)}$ lie in the closure of $\Omega_{n+k}$ in $\ov{\D}-\{ 0\} $.
This observation and (\ref{VE2}) imply that (\ref{eq:VE1}) holds, and the proof of item~4 is complete.

We next show items~5 and 6. Using item~2
we have
\[
\l_nF(\wh{\G}(n))_H=\l_n(F_E)_H=\l_n(h,0,0)
\]
for each $n\in \N$. If Case~(G1-c) occurs, then the left-hand-side of the
last equation tends to the nonzero horizontal component of the flux $F(M_{\infty })$
of $M_{\infty }$, which implies that the $\l _n$ converge to a finite positive
number $\l _{\infty}$. Taking vertical components in item~\ref{eq:Fn} we have
\begin{equation}
\label{eq:aaa}
\l_nF(\wh{\G}(n))_V=\l_n[(F_E)_V-2\pi \be_ne_3].
\end{equation}
Taking $n\to \infty $ in (\ref{eq:aaa}), we obtain $\langle F(M_{\infty }),e_3\rangle =\l _{\infty }(\tau -2\pi \beta _{\infty })$,
hence $\be_{\infty }$ is finite (and negative, as $\be _n$ is nonpositive for every $n$) and
$F(M_{\infty})=\l _{\infty }(h,0,\tau -2\pi \beta _{\infty })$.

If Case~(G1-b) happens, then the horizontal component of the flux of $M_{\infty }$ is zero and thus,
a similar reasoning shows that $\frac{|F(\wh{\G}(n))_H|}{|F(\wh{\G}(n))_V|}\to 0 $, hence
the $\be_n$ diverge to $-\infty $ and the $\l_n$ converge to zero. This finishes the proof of the corollary.
\end{proof}

In the remainder of this section, we will assume that $E$ satisfies the normalization
stated in Corollary~\ref{fluxestimates}, and we will also use the notation in that corollary.

\begin{lemma}
\label{lemma4.5}
Suppose $\{ 	p'_n\} _n\subset E$ is a divergent sequence such that $\{ I_E(p'_n)\} _n$ is bounded.
Then, $\beta_{\infty }$ is finite and a subsequence of the surfaces
$E-p'_n$ converges smoothly on compact sets of $\R^3$
with multiplicity one to the Riemann minimal example with horizontal ends
and flux vector $(h,0,\tau-2\pi \be_{\infty })$.
\end{lemma}
\begin{proof}
First assume that the Gaussian curvature of the sequence
$\{ E-p'_n\} _n$ is locally bounded in $\R^3$.
Then, a subsequence of $\{ E-p'_n\}_n$ converges to a minimal
lamination $\mathcal{L}$ of $\R^3$ with a nonsimply connected leaf $L$
passing through the origin and genus zero. By Theorem~7 in~\cite{mpr3},
$L$ is proper. By the Halfspace Theorem, $L$ is the unique leaf of $\mathcal{L}$. Since $L$ has genus zero, then $L$ is either a catenoid or a
Riemann minimal example; in particular, the convergence of
$E-p'_n$ to $L$ is of multiplicity one.
Similar arguments as those in the proof of Lemma~\ref{lemma4.2}
imply that:
\begin{enumerate}[({I}1)]
\item If $L$ is a catenoid (resp. a Riemann minimal example), then the waist curve of $L$ (resp.
 each circle contained in $L$) is the limit as $n\to \infty $ of closed curves
$\a _n\subset \overline{\D}(*)$ that wind once around the limit end $\vec{0}$ of $E$ in the parameter domain
$\overline{\D}(*)$ of $E-p'_n$.
\item  The annular ends of $L$ are horizontal.
\end{enumerate}
Note that the horizontal component $F(\a_n)_H$ of the flux of $E-p'_n$ along
$\a_n$ is independent of $n$ and nonzero (by item~2 of Corollary~\ref{fluxestimates}),
which is clearly impossible if $L$ is a vertical catenoid.
This proves that $L$ is a Riemann minimal example. The fact that the flux of $L$ is  $(h,0,\tau-2\pi \be_{\infty })$
comes from taking limits in the fluxes of the curves $\a_n$ and using the arguments in the proof of item~5 of Corollary~\ref{fluxestimates}.
This completes the proof of the lemma provided that the Gaussian curvature of
$\{ E-p'_n\} _n$ is locally bounded in $\R^3$.

Now assume that the Gaussian curvature of $\{ E-p'_n\} _n$ fails to be locally bounded in $\R^3$.
As $I_E$ is bounded away from zero outside every $\ve$-neighborhood of $\partial E$ by
Theorem~\ref{propos3.4}, then Theorem~2.2 in~\cite{mpr8}
ensures that after choosing a subsequence, $\{ E-p'_n\} _n$
converges to a minimal parking garage structure with two columns (the
one-column case of a limiting parking garage structure is ruled out because
$I_E(p'_n)$ is bounded from above by assumption). Similar
arguments as in the proof of item~3 of Corollary~\ref{fluxestimates} lead
to a contradiction, which completes the proof of the lemma.
\end{proof}

\subsection{Analysis of the Case~(G1) when $M_{\infty }$ is a Riemann minimal example.}
\label{sec4.2}

In this section, we will prove that Theorem~\ref{thm1.3}
holds provided that Case~(G1) occurs and
that the limit surface $M_{\infty }$ of the surfaces $\wt{E}_n$ given by (\ref{eq:tildeEn})
is a Riemann minimal example. 

Let $\mathcal{R}$ be the Riemann example with horizontal ends
and flux vector $(h,0,\tau-2\pi \be_{\infty })$, which is just a fixed rescaling of $M_{\infty }$
by item~5 of Corollary~\ref{fluxestimates}.
Recall that $\mathcal{R}$ is invariant under the $\pi $-rotation about infinitely many
horizontal straight lines $L_k$, $k\in \Z$, that intersect the surface orthogonally
(the lines $L_k$ are parallel to the
lines in which $\mathcal{R}$ intersects horizontal planes at the heights of its planar ends, and the heights of
$L_k$ are ordered by $k\in \Z$).
Given $k\in \Z$, let $A_1(k),A_2(k)\in \mathcal{R}$ the two points in which $L_k$ intersects $\mathcal{R}$.
For $i=1,2$, let $J_i^{\mathcal{R}}\subset \mathcal{R}$ be the integral curve of
the gradient of the third coordinate function $x_3$ of $\mathcal{R}$, passing through the points $A_i(k)$
for all $k\in \N$.
$J_2^{\mathcal{R}}$ is the reflected image of $J_1^{\mathcal{R}}$ with respect to the
vertical plane of symmetry of $\mathcal{R}$, and both $J_1^{\mathcal{R}},J_2^{\mathcal{R}}$
are properly embedded, periodic Jordan arcs, see Figure~\ref{fig10}  for a picture in a fundamental
region of $\mathcal{R}$.
 If we parameterize $\mathcal{R}$ conformally by a cylinder
$\esf^1\times \R$ so that $x_3$ corresponds to the projection over the second factor, then $J_i^{\mathcal{R}}$
corresponds to $\{ \theta_0 \} \times \R$ for certain $\theta_0\in \esf^1$. Observe that the image of
$J_1^{\mathcal{R}}$ through the Gauss map $N_{\mathcal{R}}$ of $\mathcal{R}$
is a simple closed curve $C\subset \esf^2$, and
if we parameterize $J_1^{\mathcal{R}}$ by $x_3$,
then the derivative of the argument of $g_{\mathcal{R}}(J_1^{\mathcal{R}}(x_3))$
is a positive (or negative) periodic function, where $g_{\mathcal{R}}$ denotes the
stereographic projection of $N_{\mathcal{R}}$ from the north pole of $\esf^2$; this last property follows from the
well-known fact that the Gauss map of a minimal surface and its conjugate minimal surface are the same, and
the conjugate surface of a Riemann example is another Riemann example, where the integral curves of the gradient
of $x_3$ correspond to circles in the conjugate surface.

\begin{proposition}
\label{propos4.6}
Let $C\subset \esf^2$ be closed curve defined in the last paragraph.
Then, after replacing $E$ by another end representative of the limit end,
the inverse image of $C$ through the Gauss map of $E$ consists of two disjoint,
proper Jordan arcs $J_1,J_2$ satisfying the following properties:
\begin{enumerate}
\item $I_E$ restricted to $J_1\cup J_2$ is bounded from above, and
\[
\limsup _{x\in J_1\cup J_2}I_E(x)=\limsup _{x\in J_1^\mathcal{R}\cup
J_2^{\mathcal{R}}}I_{\mathcal{R}}(x)<\infty.
\]
\item When viewed in $\R^3$, the unit tangent vector along $J_1\cup J_2$ makes an angle with the horizontal planes
which is bounded away from zero.
 \end{enumerate}
\end{proposition}
\begin{proof}
After a small perturbation of $C$ (by the Sard-Smale theorem), we can assume that the Gauss map $N$ of $E$ is transverse to $C$.
In particular, $N^{-1}(C)$ consists of a proper
(possibly disconnected) 1-dimensional submanifold of $E$; after replacing $E$ by a subend, we may assume that
the geometry of $E$ near $\partial E$ is close to the one of $\mathcal{R}$ and thus,
$N^{-1}(C)$ intersects $\partial E$ transversely at two points. Observe that
the tangent plane to $E$ along $N^{-1}(C)$ is bounded away from the horizontal.
The proposition will be a consequence of three assertions.

\begin{assertion}
\label{ass4.7}
If $I_E(p'_n)\to \infty$ for a sequence of points
$p'_n\in N^{-1}(C)$, then for $n$ large $N^{-1}(C)$ makes an angle
with the horizontal at $p'_n$ which is bounded away from zero.
\end{assertion}
\begin{proof}
Let $J(n)$ denote the component of $N^{-1}(C)$ that contains $p'_n$. Arguing by contradiction,
we may assume that the tangent line to
$J(n)$ at $p'_n$ makes an angle less than $\frac1n$ with the horizontal and $I_E(p'_n)$ is much greater than $n$.
Since $I_E(p'_n)\to \infty$, Proposition~1.1 in~\cite{cm35} ensures that
after replacing by a subsequence, we may assume that $p'_n$ lies in a compact minimal disk
$D_n\subset \ov{\B}(p'_n,n)\cap E$ with $\partial D_n \subset \partial \overline{\B}(p'_n,n)$.
Since the vertical component of the flux $V_E=\tau-2\pi \be_{\infty }$
is finite and the tangent plane to $E$ along $N^{-1}(C)$ is bounded away from the
horizontal, then for $n$ large, there exist constants an $C,R_0\in (0,n/3)$ depending on $V_E$
and there is a point $q'_n\in \B(p'_n,R_0)\cap E$ where the absolute Gaussian curvature of $E$ is
at least $C$.  It then follows from Theorem~0.1 in~\cite{cm23} that a subsequence of
the disks
\[
\Sigma_n=\frac{1}{\sqrt{n}}(D_n-p'_n)\subset \ov{\B}(\sqrt{n})
\]
converges on compact subsets of $\rth$
to a minimal parking garage structure $\cF$ with
a single column being a straight line $L$ passing through the origin and orthogonal to the planes in $\cF$,
see also Meeks~\cite{me25,me30}.  Since $V_E$ is finite, $L$ is the $x_3$-axis.

It follows from~\cite{me20} that for $n$ large, $\Sigma_n(1)=\Sigma_n\cap \ov{\B}(\vec{0},1)$
contains a compact arc $\a_n$ along which $\Sigma_n(1)$ has vertical tangent
spaces, and $\a_n$ is converging $C^1$ to the line segment  $\{(0,0,t) \mid t\in [-1,1]\}$;
this last result can be found in Meeks~\cite{me20}.  Some further refinements by
Meeks, P\'erez and Ros (Corollary~4.27 in~\cite{mpr14})
give that at every point $x\in\a_n$ and for $n$ large enough, $\Sigma _n(1)$ nearby $x$
can be closely approximated in the $C^2$-norm by compact domains of a shrunk vertical
helicoid $H_x$ whose axis contains $x$, and in the complement  of a small tube $T_n$
around $\a_n$ containing the forming vertical (scaled) helicoids, the remaining surface
$\Sigma_n(1)-T_n$ consists of almost horizontal multigraphs with an arbitrarily large number
of sheets for $n$ large.  In particular, the corresponding curve
$\frac{1}{\sqrt{n}}(J(n)-p'_n)$
lies in $T_n$. Since the inverse image $J_H$ of $C\subset\esf^2$ by the Gauss map
on a vertical helicoid makes an angle with the horizontal
that is bounded away from zero (because if we use conformal coordinates $\rho e^{i\theta}$
defined on $\C-\{ 0\} $ for the helicoid so that the polar angle $\t$ corresponds to height in $\R^3$,
then $J_H$ can be parameterized by $\t$, and the angle of $J_H(\t)$ with the horizontal has constant
positive derivative with respect to $\t$).
Thus, for $n$ sufficiently large the same property holds for $J(n)$ near $p'_n$.
This contradiction completes the proof of Assertion~\ref{ass4.7}.
\end{proof}

Let $c_0,c_1\subset \mathcal{R}$ be the horizontal circles passing through the points
$A_1(0),A_1(1)$ defined just before the statement of
Proposition~\ref{propos4.6}. Let {\it Cyl} be a solid, compact
vertical cylinder whose axis passes through the branch point of $N_{\mathcal{R}}$ at height
$\frac{1}{2}(x_3(c_0)+x_3(c_1))$, of radius $r>0$ large enough so that
$[J^{\mathcal{R}}_1\cup J^{\mathcal{R}}_2]\cap x_3^{-1}((x_3(c_0),x_3(c_1)))$
is contained in the interior of {\it Cyl}, and such that the top and bottom disks in $\partial
\mbox{\it Cyl}$ contain the circles $c_1, c_0$, respectively, see Figure~\ref{fig10}.
\begin{figure}
\begin{center}
\includegraphics[width=9cm]{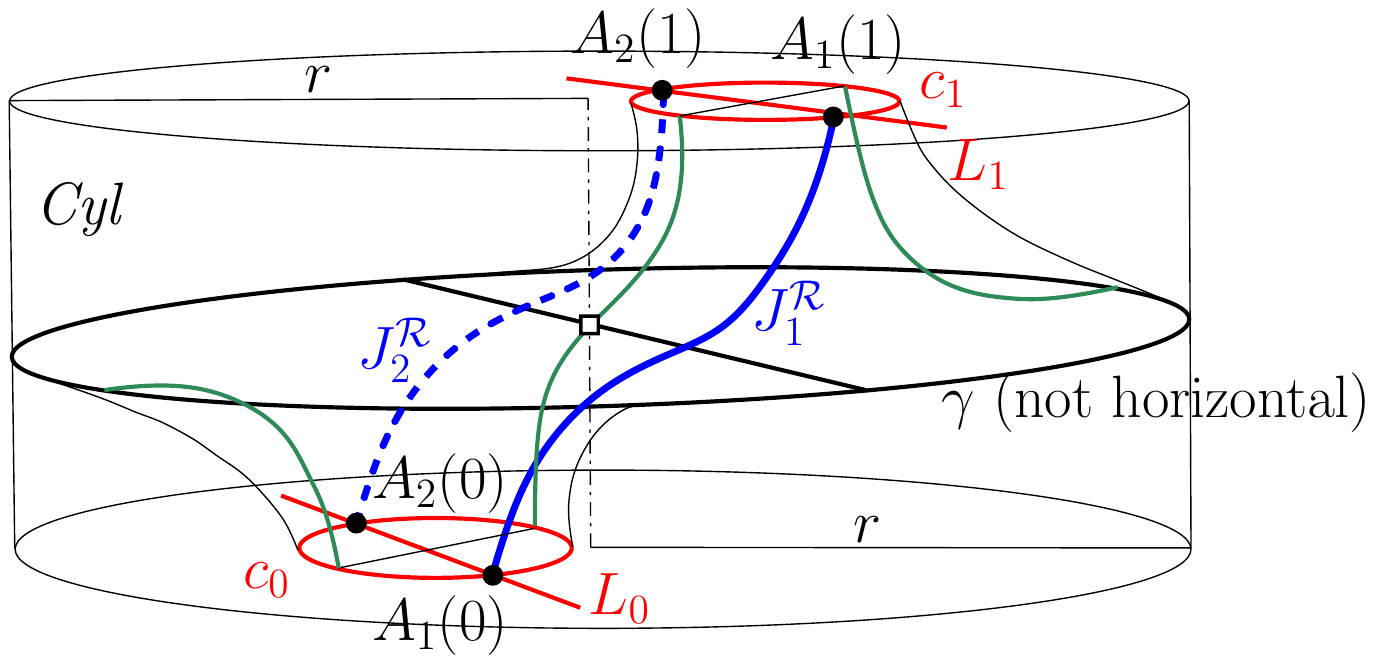}
\caption{The small square in the center of the figure represents the branch point of the Gauss map $N_{\mathcal{R}}$ of $\mathcal{R}$
whose height is the average of the heights of the circles $c_0,c_1$. The green curves represent the intersection
of $\mathcal{R}\cap \mbox{\it Cyl}$ with the symmetry plane of $\mathcal{R}$.}
\label{fig10}
\end{center}
\end{figure}
By construction, the side of {\it Cyl} intersects $\mathcal{R}$ in an almost horizontal closed curve $\g$
(for $r$ large enough) that winds once around the axis of {\it Cyl}, and $N_{\mathcal{R}}(\g)$
is a closed spherical curve arbitrarily close to $(0,0,1)\in \esf^2$ that winds twice around $(0,0,1)$ (we can
assume that $(0,0,1)$ is the extended value of $N_{\mathcal{R}}$ at the planar end of $\mathcal{R}$
between the heights of $c_0$ and $c_1$). Let $\mathcal{R}_0$ be the compact subdomain $\mathcal{R}\cap
\mbox{\it Cyl}$ of $\mathcal{R}$.

Given $\ve >0$ small, let $\mathcal{R}_0(\ve)\subset E$ be a compact subdomain which is
$\ve$-close to $\mathcal{R}_0$ in the Hausdorff distance in $\R^3$, which exists since
the smooth limit of translations of $E$ is $\mathcal{R}$. We may assume that
$\partial \mathcal{R}_0(\ve)$ consists of three components $c_0(\ve),c_1(\ve),\g (\ve)$,
so that $c_0(\ve),c_1(\ve)$ are horizontal convex curves and $\g (\ve)$ is a Jordan curve whose
image by $N$ is at positive spherical distance from $C$.
We may take $\ve $ sufficiently small so that $N^{-1}(C)\cap \mathcal{R}_0(\ve)$
consists of two disjoint arcs, each one joining $c_0(\ve)$ to $c_1(\ve)$.
 The complement of $\mathcal{R}_0(\ve)$ in $\ov{\D}-\{ 0\} $ consists of three components,
namely two annular components of which one contains $\partial E$ and another one contains $0$, and a
disk component $\Delta$. By the convex hull property, $\Delta $ contains a finite positive
number of annular ends of $E$. Observe that we can choose an infinite sequence of
pairwise disjoint domains of the type $\mathcal{R}_0(\ve)$ in $E$, so that the sequence collapses to the origin when
viewed in $\ov{\D}-\{ 0\} $.

\begin{assertion}
\label{ass4.8}
After replacing $E$ by a limit subend, every such a domain $\Delta $ contains
exactly one annular end of $E$.
\end{assertion}
\begin{proof}
Arguing by contradiction, assume that we have a sequence
$\Delta_n$ of such domains so that $\Delta _n$ contains at least two annular ends of $E$. As
the limiting normal vector of $E$ at its annular ends is vertical and
$N^{-1}(C)\cap \partial \Delta _n=\mbox{\O}$, then a simple continuity argument gives that
$N^{-1}(C)\cap \Delta _n$ contains a finite positive number of components, each of which is a Jordan curve.
Choose one of these Jordan curves $\be_n\subset N^{-1}(C)\cap \Delta_n$. If for each $n\in
\N$ there exists a point $p'_n\in \be_n$ so that the sequence $\{ I_E(p'_n)\} _n$ is bounded,
then Lemma~\ref{lemma4.5} implies that after extracting a subsequence, the $E-p'_n$ converge
 smoothly with multiplicity one to $\mathcal{R}$. Note that the sequence
$\{ (\Delta _n\cap E)-p'_n\} _n$ also converges to $\mathcal{R}$ (because the
intrinsic distance from
$\g_n(\ve)$ to $\be_n$ goes to infinity as $n\to \infty $). This is impossible, as every closed
curve in $\Delta _n\cap E$ has vertical flux but $\mathcal{R}$ does not have this property.
Therefore, the sequence of numbers $\{ \min I_E(x)\ | \ x\in \be_n\} _n$ goes to $\infty $.
In this situation, Assertion~\ref{ass4.7} gives a contradiction by taking for each $n\in \N$
a point $x_n\in \be_n$ of maximum height in $\R^3$ (which exists since $\be_n$ is a
Jordan curve). This finishes the proof of Assertion~\ref{ass4.8}.
\end{proof}

By Assertion~\ref{ass4.8}, after replacing $E$ by a limit subend, we assume that every $\Delta$-domain
as defined in the description just before the statement of Assertion~\ref{ass4.8}, contains one end of $E$.
By the last sentence before Assertion~\ref{ass4.8},
these $\Delta $-domains occur in a sequence collapsing to the limit end of $E$.
By the Gauss-Bonnet formula, the total Gaussian curvature of the annulus $\Delta\cap E$ is arbitrarily small
by choosing $\mathcal{R}_0(\ve)$ appropriately. Therefore, $\Delta \cap E$ is a graph over its projection
into the $(x_1,x_2)$-plane, of a function with small length of its gradient (the maximum of the length of
the gradient of the graphing function occurs at $\partial \Delta$). By gluing $(\Delta \cap E)\cup \mathcal{R}_0(\ve)$
with the two planar disks bounded by $c_0(\ve)\cup c_1(\ve)$, we obtain a piecewise smooth
topological plane~$\Pi $, which is properly embedded in $\R^3$. Take a maximal collection $\{ \Pi_n\} _n$ of
such topological planes, so that $\Pi _n\cap \Pi _m=\mbox{\O}$ if $n\neq m$. $\R^3-\bigcup_{n\in \N}\Pi_n$
consists of a countable union of open components, each of which is a topological slab $S_n$.

Given such a slab $S_n$, observe that the closure of $N^{-1}(C)\cap S_n$ is a compact 1-manifold with four boundary
points. Therefore, $N^{-1}(C)\cap S_n$ consists of a finite number of Jordan curves plus two arcs. We first check
that for $n$ sufficiently large, $N^{-1}(C)\cap S_n$ does not contain Jordan curve components. Otherwise,
there exists a sequence of points $p'_n\in N^{-1}(C)\cap S_n$
where the tangent line to $N^{-1}(C)\cap S_n$
is horizontal. By Assertion~\ref{ass4.7}, $I_E(p'_n)$ must be
bounded. Thus, Lemma~\ref{lemma4.5} gives that after
extracting a subsequence, the $E-p'_n$ converge smoothly to
$\mathcal{R}$, which is impossible since on $\mathcal{R}$ the
corresponding set $N_{\cal R}^{-1}(C)=J_1^{\mathcal{R}}\cup J_2^{\mathcal{R}}$ satisfies that the angle
with the horizontal planes is bounded away from zero.
Therefore, $N^{-1}(C)\cap S_n$ does not contain Jordan curve components
that for $n$ sufficiently large, and the same argument proves that the two compact arcs in
the closure of $N^{-1}(C)\cap S_n$
make an angle with the horizontal planes which is bounded away from zero; in particular, each of these arcs joins
two boundary components of $S_n$. After replacing $E$ by a subend, we can assume
that $N^{-1}(C)$ consists of two proper arcs $J_1,J_2$ satisfying item~2 of the
proposition. In particular, each $J_i$ can be parameterized by the $x_3$-coordinate, $i=1,2$.

\begin{assertion}
\label{asser4.9}
Given $\de>0$, there exists $x_{3,0}=x_{3,0}(\de)\in \R$ such that if $x_3\geq x_{3,0}$, then
for $i=1,2$ it holds
\[
I_E(J_i(x_3))\leq \de +\limsup _{x\in J_1^\mathcal{R}\cup
J_2^{\mathcal{R}}}I_{\mathcal{R}}(x).
\]
\end{assertion}
\begin{proof}
Arguing by contradiction, suppose that the assertion fails. Then, there exists $\de >0$ and
a sequence of heights $t_n\to \infty $ so that $I_E(J_i(t_n))>\de +c$, where
$c=\limsup _{x\in J_1^\mathcal{R}\cup J_2^{\mathcal{R}}}I_{\mathcal{R}}(x)$.
After passing to a subsequence, we can assume that
given $n\in \N$, there exists $t'_n\in (t_n,t_{n+1})$ so that the related point $J_i(t_n')$ lies in a region of
the form $\mathcal{R}_0(\ve)$ where $E$ is $\ve$-close to a compact portion of $\mathcal{R}$.
Then, after taking $\ve $ much smaller than $\de $, we can assume that $I_E(J_i(t_n'))\leq
\frac{\de}{2}+c$. By continuity of $I_E\circ J_i$, there exists $t_n''\in (t_n,t_n']$ such that $I_E(J_i(t_n''))=
\frac{\de}{2}+c$ for each $n\in \N$. Applying Lemma~\ref{lemma4.5} to
$p_n:=J_i(t_n'')$ we deduce that the $E-J_i(t_n'')$ converge (after extracting a subsequence)
smoothly to $\mathcal{R}$, which is impossible since the value of the injectivity radius
of  $E-J_i(t_n'')$ at the origin is $\frac{\de}{2}+c$ and the injectivity radius is a
continuous function with respect to smooth limits (see e.g. Ehrlich~\cite{ehr1} and Sakai~\cite{sa2}).
Now the assertion is proved.
\end{proof}
Finally, item~1 of Proposition~\ref{propos4.6} follows directly from Assertion~\ref{asser4.9}. Item~2
of Proposition~\ref{propos4.6} follows from Assertion~\ref{asser4.9}, Lemma~\ref{lemma4.5} and the fact
that the unit tangent vector along the curves $J_1^{\mathcal{R}}\cup J_2^{\mathcal{R}}$ makes an angle
with the horizontal planes which is bounded away from zero. This
completes the proof of the proposition.
\end{proof}

A direct consequence of Lemma~\ref{lemma4.5} and Assertion~\ref{asser4.9}
is that for every divergent sequence of points $p'_n\in J_1\cup
J_2$, the surfaces $E-p'_n$ converge smoothly to $\mathcal{R}$
after passing to a subsequence. This property together with
Assertion~\ref{ass4.8} imply that after replacing $E$ by a subend, $E$ consists of an infinite
number of noncompact pieces $M_n$,  each of which is has the topology of a pair of
paints with a point
removed (this puncture is one annular end of $E$), and the two compact boundary components
of $c_{0,n},c_{1,n}$ of $M_n$ can be taken arbitrarily close to translated copies of
the horizontal circles $c_0,c_1\subset \mathcal{R}$ defined in the paragraph just after the
proof of Assertion~\ref{ass4.7}. Furthermore, $c_{1,n}=c_{0,n+1}$ and
$M_n\cap M_{n+1}=c_{1,n}$ for all $n\in \N$.

We next explain why Theorem~\ref{thm1.3} holds in the Case~(G1) when $M_{\infty }$
is a Riemann minimal example. The main properness statement of Theorem~\ref{thm1.3}
was proven in Section~\ref{sec3}. Items 1, 2 of Theorem~\ref{thm1.3} were proven in the
second paragraph of this section~\ref{sec4}.  Item 3 of Theorem~\ref{thm1.3} follows from
Lemma~\ref{lemma4.2} and Corollary~\ref{fluxestimates}. In particular, items 1, 2, 3
of Theorem~\ref{thm1.3} also hold in the Case~(G1) when $M_{\infty }$ is a vertical catenoid.
Assume from now on that Case~(G1) occurs and $M_{\infty }$ is a Riemann minimal example. 
Item~4 of Theorem~\ref{thm1.3}
is a consequence of the last paragraph. The same description of $E$ as a union of domains
$M_n$ implies that the Gaussian curvature of $E$ is bounded, which is
item~5 of Theorem~\ref{thm1.3}. The next proposition completes the proof of Theorem~\ref{thm1.3}
in the Case~(G1) when $M_{\infty }$ is a Riemann minimal example.

\begin{proposition}
\label{propos4.11}
If Case~(G1) occurs and $M_{\infty }$ is a Riemann minimal example, then
 item~6 of Theorem~\ref{thm1.3} holds.
\end{proposition}
\begin{proof}
Suppose that the proposition fails.
As each of the annular ends of $E$ has finite total curvature,
$E$ is conformally diffeomorphic to $\widehat{D} =\overline{\D} - \{ x \in
\D \mid |x| \leq a \}$ for some $a\in (0, 1)$, with a countable
discrete set of points $\{ e_n \}_{n \in \N}$ removed and where $|
e_n | \searrow a$ as $n \rightarrow \infty$.

By the above decomposition of $E$ as a countable union of
regions $M_n$, there exists $\de>0$ and a sequence $f_n\colon \esf^1 \times [0, \de ]\to E$
(here $\esf^1$ is the unit circle) of conformal embeddings with $f_n(\esf^1
\times [0,\de ])$ being arbitrarily close to a region $\mathcal{R}_0(\ve)$ of `Riemann type'
to which one attaches an annular end of $E$ (that might have negative logarithmic growth,
arbitrarily close to zero). Observe that the $f_n$ have pairwise disjoint images in $\R^3$
for different values of $n$.

Consider on $\widehat{D}$ the usual flat metric $g_0$. Next we
will show that the $g_0$-area of $f_n(\esf^1 \times [0, \de ])$ is at least $2\pi a^2\de $, which gives the
desired contradiction since the $g_0$-area of $\widehat{D}$ is
finite and we have an infinite number of such pairwise disjoint
embeddings $f_n$ in $\widehat{D}$.

To compute the $g_0$-area of $f_n(\esf^1 \times [0, \de ])$,
we will apply the coarea formula to the smooth function
$h_n\colon f_n(\esf^1 \times [0, \de ])\to \R$ that satisfies
\begin{equation}
\label{eq:cf1}
(h_n\circ f_n)(\t,t)=t,\quad \mbox{for all }(\t ,t)\in \esf^1 \times [0, \de ].
\end{equation}
Thus,
\begin{equation}
\label{eq:cf2}
{\rm Area}(f_n(\esf^1 \times [0,\de ]),g_0) =
 \int _0^{\de}\left( \int _{h_n^{-1}(t)}\frac{ds_t}{| \nabla _0h_n|}\right) dt,
\end{equation}
where $ds_t$, $|\nabla _0h_n|$ denote respectively the length element
of the simple closed curve $h_n^{-1}(t)=f_n(\esf^1 \times \{ t\} )$ and the
gradient of $h_n$, both computed with respect to $g_0$. Since $f_n$ is a
conformal diffeomorphism onto its image endowed with $g_0$, we deduce that
\begin{equation}
\label{eq:cf3}
v:=\frac{1}{\left| \frac{\partial f_n}{\partial t}(\t,t) \right|}\frac{\partial f_n}{\partial t}(\t,t)
\end{equation}
is a unit normal vector to the curve $f_n(\esf^1 \times \{ t\} )$ at the point
$f_n(\t,t)$. Hence, (\ref{eq:cf1}) and (\ref{eq:cf3}) give
\[
|\nabla _0h_n|(f_n(\t,t))=(dh_n)_{f_n(\t,t)}(v)=\frac{1}{\left| \frac{\partial f_n}{\partial t}(\t,t) \right|},
\]
which implies that the right-hand-side of (\ref{eq:cf2}) equals
\begin{equation}
\label{eq:cf4}
 \int _0^{\de}\left( \int _{f_n(\esf^1\times\{ t\})}\left| \frac{\partial f_n}{\partial t} \right|ds_t\right) dt
=\int _0^{\de}\left( \int _{\esf^1\times\{ t\}}\left| \frac{\partial f_n}{\partial t} \right| \, \left| \frac{\partial f_n}{\partial \t} \right| d\t \right) dt.
\end{equation}
Using again the conformality of $f_n$ in the right-hand-side of (\ref{eq:cf4})
and the Cauchy-Schwarz inequality, we obtain
\[
{\rm Area}(f_n(\esf^1 \times [0,\de ]),g_0) = \int _0^{\de}
\left( \int _{\esf ^1\times \{ t\} }\left| \frac{\partial
f_n}{\partial \theta }\right| ^2d\theta \right) dt \geq
\frac{1}{2\pi }\int _0^{\de} \left( \int _{\esf ^1\times \{ t\} }\left|
\frac{\partial f_n}{\partial \theta }\right| d\theta \right) ^2dt
\]
\[
=\frac{1}{2\pi }\int _0^{\de}[\mbox{length}(f_n(\esf ^1\times\{ t \}))]^2dt
\stackrel{(\star)}{\geq }
\frac{1}{2\pi }\int _0^{\de}(2\pi a)^2dt=2\pi a^2\de ,
\]
where in $(\star)$ we have used that $f_n(\esf^1 \times \{ 0 \})$ is a loop in
$\widehat{D} $ that is parallel to $\partial \widehat{D}$. This  completes
the proof of the proposition.
\end{proof}

\subsection{Analysis of the Case~(G1) when $M_{\infty }$ is a catenoid.}
\label{sec4.3}

We will devote this section to prove Theorem~\ref{thm1.3}
provided that Case~(G1) occurs and that the limit surface $M_{\infty }$ of the
surfaces $\wt{E}_n$ given by (\ref{eq:tildeEn}) is a vertical catenoid.

Without loss of generality, we will assume that the waist circle of
$M_{\infty }$ is  the unit circle in the $(x_1, x_2)$-plane.
Recall the following properties demonstrated above for each $n\in \N$:
\begin{enumerate}[(J1)]
\item There exists a horizontal plane $P_n$ so that
$P_n\cap E$ contains a convex Jordan curve $\wh{\G}(n)$ and the $\l_n
(\wh{\G}(n)-p_n)$ converge as $n\to \infty $ to the waist circle $\wt{\g}$
of $M_{\infty}$.  Moreover, the heights of $P_n$ diverge increasingly to $\infty $.
\item $\wh{\G}(0)=\partial E$, $\wh{\G}(n)$ is topologically parallel to $\partial E$ in
$\ov{\D}-\{ 0\} $ (items 1 and 2 of Lemma~\ref{lemma4.2}) and $\wh{\G}(n)$
is the unique compact component of $P_n\cap E$ that is topologically parallel to $\partial E$ in
$\ov{\D}-\{ 0\} $ (claim in the first paragraph of the proof of Lemma~\ref{lemma4.2}).

\item When viewed in $\ov{\D}(*)$, $\wh{\G}(n)$ bounds a noncompact
domain $E(\wh{\G}(n))$ which is an end representative of the limit end of
$E$; we take $E(\wh{\G}(n))$ as the
closure of the component of  $E-\wh{\G}(n)$ such that $E(\wh{\G}(n))\cap
\partial E=\mbox{\rm \O }$.

\item When viewed in $\R^3$, $\wh{\G}(n)$ bounds a compact convex disk
$D_n\subset P_n$ whose interior is disjoint from $E$, and
the $D_n$ all lie in the same side of $E$ (item~(6a) of Lemma~\ref{lemma4.2}).
We will denote by $W$ the closure of the component of $\R^3-(E\cup D_0)$ that contains
$D_n$ for $n\geq 1$.

\item $F(\wh{\G}(n))=F_{E}-2\pi\beta_n e_3$ and $F(\wh{\G}(n))_H=(F_E)_H$ is not zero (items 1 and 2 of
Corollary~\ref{fluxestimates}).

\item $V_E=\infty $, $\be_{\infty}=\sum_n\be _n=-\infty $ and $\l_n\to 0$ as $n\to \infty $
(items~4 and 6 of Corollary~\ref{fluxestimates}).
In particular, the annular ends of $E$ all have strictly negative logarithmic growths
(because if one of these ends were asymptotic to a plane, then all annular ends of $E$
above this last one would be planar as well, and $F(\wh{\G}(n))_V$ would then be
independent of $n$, which contradicts that $\be_{\infty }=-\infty $ after taking vertical
components in the first formula of (J5)).
\item $\{ I_E(p'_n)\} _n$ is unbounded for every divergent
sequence $\{ p'_n\} _n\subset E$ (Lemma~\ref{lemma4.5}).
\end{enumerate}

\begin{proposition}
\label{lemma4.15}
Given $\ve >0$ small, there exist compact annular subdomains $\Delta _n=
\Delta _n(\ve )\subset E$
bounded by horizontal convex curves, such that for $n$ sufficiently large:
\begin{enumerate}
\item There exist numbers $ \l'_n>0$ converging to zero
and points $p'_n\in \R^3$ such that
the Hausdorff distance
between $\l'_n(\Delta _n-p'_n)$ and $M_{\infty }(\ve )=\{ x\in M_{\infty }\ : \
|x_3|\leq 1/\ve \} $ is less than $\ve $, and $\l'_n(\Delta _n-p'_n)$ can be written as
a normal graph over its projection to $M_{\infty }$ with $C^2$-norm less than $\ve $ and the
boundary curves of $\l'_n(\Delta _n-p'_n)$ are contained in the planes $\ds \{x_3=\pm 1/\ve\}$.

\item For all $n\in \N$, the boundary curves of $\Delta _n$ are topologically
parallel to $\partial E$ in $\ov{\D}-\{ 0\} $. Therefore, we may assume that the $\Delta _n$ are ordered
so that for each $n$, $\Delta _n$ is contained in the component of $\ov{\D }(*)-
\Delta _{n+1}$ that contains $\partial E$.

\item The closed horizontal slab in $\R^3$ that contains $\Delta _n$ is strictly below the one
that contains $\Delta _{n+k}$ for all $n,k\in \N$, $k\neq 0$.

\item Except for a finite number of components,
each component $\Omega $ of ${\displaystyle E- \bigcup_{n\in \N} \Delta_n}$
is topologically a plane with two disks removed, and $\Omega $
is the graph of a function $u$ defined over the projection of $\Omega $ to the
$(x_1, x_2)$-plane, with $|\nabla u|<1$.
\item The Gaussian curvature $K_E$ of $E$ is asymptotically zero.
\end{enumerate}
\end{proposition}
\begin{proof}
Items 1 and 2 follow from the facts that the sequence $\{ \wt{E}_n\} _n$
defined by (\ref{eq:tildeEn}) converges smoothly on compact subsets of $\R^3$ with multiplicity~1
to the vertical catenoid $M_{\infty }$, that $\l _n\to 0$,
and that the convex horizontal curves $\wh{\G}(n)$ defined in (J1)
are topologically parallel to $\partial E$ in $\ov{\D}-\{ 0\}$.

We next prove item~3. Let $A(\wh{\G}(n),\wh{\G}(n+1))$ be the subdomain
of $E$ bounded by $\wh{\G}(n)\cup \wh{\G}(n+1)$.
Since the ends of $A(\wh{\G}(n),\wh{\G}(n+1))$ have negative logarithmic
growths and $\{ x_3(P_n)\} _n$ is increasing,
the maximum principle applied to the function $x_3|_{A(\wh{\G}(n),\wh{\G}(n+1))}$ implies
that $A(\wh{\G}(n),\wh{\G}(n+1))$ is contained in the halfspace
$\{ x_3\leq x_3(\wh{\G}(n+1))\}$. As
$A(\wh{\G}(n-1),\wh{\G}(n))$ is contained in the halfspace
$\{ x_3\leq x_3(\wh{\G}(n))\}$ and
$A(\wh{\G}(n-1),\wh{\G}(n))\cap A(\wh{\G}(n),\wh{\G}(n+1))=\wh{\G}(n)$, then we conclude
that $A(\wh{\G}(n),\wh{\G}(n+1))$ locally lies above $P_n$ along $\wh{\G}(n)$. Observe that
$A(\wh{\G}(n),\wh{\G}(n+1))$ contains the portion $\Delta _n^+$ of $\Delta _n$ that lies above $P_n$,
and it also contains the portion $\Delta _{n+1}^-$ of $\Delta _{n+1}$ that lies below $P_{n+1}$.
By the maximum principle applied to $x_3|_{A(\wh{\G}(n),\wh{\G}(n+1))-[\Delta _n^+
\cup \Delta ^-_{n+1}]}$, we deduce that the lower boundary curve of $\Delta _{n+1}$ lies strictly
above the upper boundary curve of $\Delta _n$, which implies that item~3 holds.

Next we prove item~4 of the proposition.
For $\ve>0$ small and fixed, choose a
maximal collection of pairwise disjoint domains $\{ \Delta_n \}_{n\in \N \cup \{ 0\} }\subset E$
which satisfy items~1, 2 and 3. After replacing $E$ by a subend representative of the limit end,
we may assume that $\partial E$ is thre bottom boundary component of $\Delta_0$.

We will first show that if a
component $\Omega $ of $E- \bigcup_{n\in \N \cup \{ 0\} } \Delta_n$ is topologically a plane
with two disks removed, then $\Omega $
is the graph of a function $u$ defined over the projection of $\Omega $ to the
$(x_1, x_2)$-plane, with $|\nabla u|<1$. To see this, note that
if $\ve >0$ is sufficiently small, the total geodesic curvature of $E$ along
each of the two components
of $\partial \Omega$ is arbitrarily close to $-2\pi $. As we are assuming that
$\Omega$ has exactly one end and this end has finite total curvature, then
the Gauss-Bonnet formula gives that $\Omega$ has arbitrarily small total Gaussian
curvature by taking $\ve $ sufficiently small. Therefore, the Gaussian image of $\Omega$
lies in a small neighborhood of one of the poles, say the north pole, of the unit sphere.
Hence, the projection of $\Omega$ to the $(x_1,x_2)$-plane is a proper submersion
which is injective on each of the two boundary components of $\Omega$. In this
setting, a straightforward covering space type argument implies that $\Omega$ is a graph
of a smooth function $u$ defined over the projection of $\Omega$ to the $(x_1,x_2)$-plane.
The fact that $|\nabla u|<1$ follows from the fact that the Gaussian image of $\Omega$ lies
in a small neighborhood of the north pole. {\bf Therefore, to prove item~4 we must show that
except for a finite number of components  of $E- \bigcup_{n\in \N \cup \{ 0\} } \Delta_n$,
all these components have
the topology of a plane minus two disks.} Observe that item~2 of this proposition
implies that every component of  $E- \bigcup_{n\in \N \cup \{ 0\} } \Delta_n$ is a planar domain
with two boundary components and a finite number of annular ends with negative
logarithmic growth.

Let $\{ \Omega _n\} _{n\in \N}$ be the collection of components of  $E- \bigcup_{n\in \N \cup \{ 0\} }\Delta_n$.
Enumerate these components so that $\Omega_n$
is the component of  $E- \bigcup_{n\in \N \cup \{ 0\} } \Delta_n$
with boundary components
\begin{equation}
\label{eq:alphabeta}
\a _n=\Omega_n\cap \Delta_{n-1},\quad
\be _n=\Omega _{n}\cap \Delta_n.
\end{equation}
Fix for each $n\in\N$ a dilation 
$f_n\colon \rth\rightarrow\rth$ so that the
Hausdorff distance between $f_n(\Delta_n)$ and
$M_{\infty }(\ve)$ is minimized, and so, by the
definition of $\Delta_n$, this Hausdorff distance is less than $\ve$.

\begin{lemma}
\label{claim4.12}
After passing to a subsequence,
the surfaces $ f_n(\Omega_n)$ converge with multiplicity one to the
representative $M_{\infty }\cap \{ x_3\leq -1/\ve \} $ of the
bottom end of $M_{\infty }$. In fact, the surfaces
$f_n(E)$ converge smoothly on compact subsets of $\R^3$ to $M_{\infty }$.
\end{lemma}

\begin{proof}
We first show that if the sequence of curves $\{ f_{n}(\a_{n})\} _{n}$ diverges to
infinity in $\rth$, then the lemma holds. 
Following the notation in property (J3) above, we denote by
$E(\a_n)$ the closure of the component of  $E-\a_n$ such that $E(\a_n)\cap
\partial E=\mbox{\rm \O }$.
By construction, after passing to a subsequence, we may assume that the curves
$f_n(\be_{n})$ converge in the $C^2$-norm to a convex horizontal curve
$\wh{\be}$.
Take a divergent sequence $\{ R_{n}\} _{n}$ of positive
numbers so that for each $n$, the boundary $f_{n}(\a_{n})$ of $f_{n}(E(\a_{n}))$
lies outside of the closed ball $\ov{\B}(R_{n})$ centered at the origin. Then, item~3 of
Theorem~2.2 in~\cite{mpr8} applied to the sequence of
compact minimal surfaces $\{ f_{n}(E(\a_{n}))\cap \ov{\B}(R_{n})\} _{n}$
implies that after extracting a subsequence, the $f_{n}(E(\a_{n}))\cap \ov{\B}(R_{n})$
converge smoothly on compact subsets of $\R^3$
with multiplicity one to a connected, properly embedded, nonflat minimal surface $\widehat{M}_{\infty}$
of genus zero, that is either a catenoid, a helicoid or a Riemann minimal example.
Clearly, $\widehat{M}_{\infty}$ contains the curve $\wh{\be}$. $\widehat{M}_{\infty}$ cannot be a helicoid since
$\widehat{M}_{\infty }$ has nonzero vertical flux along $\wh{\be}$;
the same argument shows that either $\widehat{M}_{\infty}$ is a vertical catenoid or a
Riemann minimal example with vertical flux.
Our earlier arguments imply that $\widehat{M}_{\infty}$ must be a vertical catenoid.
Since the limit set of the sequence $\{ f_{n}(E(\a_{n}))\cap \ov{\B}(R_{n})\} _{n}$
equals the limit set of $\{ f_{n}(E)\} _{i}$ (this follows from the properties that 
$f_n(E-E(\a_n))$ lies in the halfspace $\{ x_3\leq x_3(f_n(\a_n))\} _n$ and 
$x_3(f_n(\a_n))\to -\infty $ as $n\to \infty $), then we conclude the Lemma in the special 
case that the $f_{n}(\a_{n})$ diverge to infinity in $\rth$.

We next divide the proof into two parts: in the first one, we will prove the lemma assuming that,
after choosing a subsequence, $\Omega _n$ contains just one annular end for every $n$. 
In the second part we will suppose that, after choosing a subsequence, $\Omega _n$ contains more than 
one annular end for every $n$. 

Assume that  $\Omega _n$ contains just one annular end for every $n$. We will demonstrate that
the curves $f_n(\a_n)$ diverge to infinity in $\rth$. Arguing by contradiction, assume that after
choosing a subsequence, $f_n(\a_n)$ lies in a compact set of $\R^3$ independently of $n\in \N$. 
Recall that the logarithmic growths of the annular ends of $E$ are bounded (between
the negative logarithmic growth of the lowest end of $E$ and zero), and that
the dilation $f_n$ has homothetic factor going to zero as $n\to \infty $. Therefore,
the logarithmic growths of the unique annular end of $f_n(\Omega_n)$ is arbitrarily small in absolute
value for $n$ sufficiently large. Since the flux of $f_n(\Omega _n)$ along $f_n(\be_n)$ is converging to the
nonzero flux of $M_{\infty }$ along $\wh{\be}$ and the flux of $f_n(\Omega _n)$ along its annular end
is arbitrarily small, then the divergence theorem implies that the flux of $f_n(\Omega _n)$ along $f_n(\a_n)$ converges
to the negative of the flux of $M_{\infty }$ along $\wh{\be}$, that is nonzero. This property
and the fact that the convex planar curve $f_n(\a_n)$ lies in a compact set independent of $n$, imply that
the $f_n(\a_n)$ converge (after passing to a subsequence) to a convex, horizontal planar curve $\wh{\a}$
as $n\to \infty$. Also recall that $f_n(\Omega _n)$ is a graph over its projection to the $(x_1,x_2)$-plane,
by the fourth paragraph in the proof of Proposition~\ref{lemma4.15}.
Therefore, curvature estimates for stable minimal surfaces imply that the $f_n(\Omega _n)$ converge to
a minimal graph over the complement in the plane $\{ z=0\}$ of the two disks bounded by $\Pi (\wh{\a}),\Pi(\wh{\be})$, 
where  $\Pi (x,y,z)=(x,y,0)$. As this minimal graph has vertical flux and two convex boundary curves, a standard application
of the L\'opez-Ros deformation argument leads to contradiction. This contradiction shows that the 
 $f_n(\a_n)$ diverge to infinity in $\rth$. By the discussion in the first paragraph of this proof, we now
 conclude that Lemma~\ref{claim4.12} holds if $\Omega _n$ has one end for every $n$ (after choosing a subsequence).

Next assume that $\Omega _n$ has always at least two ends. Again by the discussion in the first paragraph 
of the proof of Lemma~\ref{claim4.12}, it remains to show that the curves $f_n(\a_n)$ diverge to
infinity in $\rth$. Assume this last property fails to hold. 
Since the diameter of the sets $f_n(\a_n)$ are uniformly bounded (because the diameter of 
$f_n(\Delta_{n-1})$ is bounded as the homothetic factor of the dilation $f_n$ is going to zero and
the diameter of $f_{n-1}(\Delta_{n-1})$ is comparable to the one of $M_{\infty }(\ve)$), 
we may assume from this point on that the
curves $f_n(\a_n)$ all lie in a fixed bounded subset of $\rth$.
This bounded set must lie below the plane $\{x_3=-1\} $ if $\ve$ is
chosen sufficient small (by the already proven item~3 of Proposition~\ref{lemma4.15}).
We will find the desired contradiction by analyzing each of the following two 
mutually exclusive situations (after passing to a subsequence):
\begin{enumerate}[(K1)]
\item The diameters of the curves $f_n(\a_{n})$ are not bounded away from zero.
\item The diameters of the curves $f_n(\a_{n})$ are bounded away from zero.
\end{enumerate}

Suppose that Case~(K1) holds. Then, after choosing a subsequence, we may assume that
the curves $f_n(\a_n)$ converge to a point $p\in\rth$ that lies below the
plane $\{ x_3=-1\} $.  Consider the sequence of compact, embedded, minimal planar domains
$\{ f_n(E(\a_n))\cap \ov{\B}(n)\} _n$. We claim that the $[f_n(E(\a_n))\cap \ov{\B}(n)]-\{ p\}$
form a locally simply connected sequence of minimal planar domains in
$\rth - \{ p\}$. Otherwise, our previous arguments show that we
can produce, after blowing-up by topology, a new limit of dilations of the
$[f_n(E(\a_n))\cap \ov{\B}(n)]-\{ p\}$ which is
a vertical catenoid. This means that after extracting a subsequence and for $n$ sufficiently large,
$[f_n(E(\a_n))\cap \ov{\B}(n)]-\{ p\}$ contains a compact subdomain
$C_n$ which is arbitrarily close to a homothetically shrunk copy of a large
compact region of a vertical catenoid, where the homothetic factor can be taken arbitrarily small.
Note that $C_n$ cannot lie in $f_n(\Omega _n)$ because this contradicts
the maximality of the family $\{ \Delta _m\} _m$.
Since $f_n(\Delta _n)$ is $\ve $-close to $M_{\infty }(\ve )$,
we deduce that $C_n$ must lie in $f_n(E(\a_{n+1}))$.
To see that this is impossible, first observe that
for $n$ sufficiently large, the generator of the homology group $H_1(C_n)$ of $C_n$
is topologically parallel to $f_n(\be_n)$ modulo annular ends of $f_n(E)$ (adapt the
arguments as in the proof of Claim~\ref{claim4.3}).
As the vertical component of the flux vector of $f_n(\Delta _n)$ along $f_n(\be_n)$
is larger than some positive number in absolute value (namely, one half of the vertical flux of
$M_{\infty}$) and the annular ends of $f_n(E)$ all have negative logarithmic growths,
we deduce from the divergence theorem that the vertical component of the flux vector of $f_n(C_{n})$ is positive and
bounded away from zero (see the last paragraph of the proof of Corollary~\ref{fluxestimates}),
which contradicts that the length of a generator of $H_1(C_{n})$ tends to zero as $n\to \infty$.
Therefore, the sequence $[f_n(E(\a_n))\cap \ov{\B}(n)]-\{ p\}$
is locally simply connected sequence in $\rth - \{ p\}$. In fact, this argument shows that for all
$n\in \N$ and given a  regular neighborhood $U_n(\de )$ of the boundary of $f_n(E(\a_n))\cap \ov{\B}(n)$ in
$f_n(E(\a_n))\cap \ov{\B}(n)$ with radius $\de >0$, the restriction of the injectivity radius function of
$f_n(E)$ to $[f_n(E(\a_n))\cap \ov{\B}(n)]-U_n(\de)$ is uniformly bounded away from zero
(independently of $n$).

In this setting, item~3 of Theorem~2.2 in~\cite{mpr8} ensures that after passing to a subsequence,
the surfaces $[f_n(E(\a_n))\cap \ov{\B}(n)]-\{ p\}$ converge to
a minimal lamination $\mathcal{L}$ of $\R^3-\{ p\} $ whose closure $\ov{\mathcal{L}}$
in $\R^3$ consists of a single leaf which is a properly embedded minimal surface $L_1$ of genus zero
that is either a helicoid, a catenoid or a Riemann minimal example.
Furthermore, the convergence of the $[f_n(E(\a_n))\cap \ov{\B}(n)]-\{ p\}$ to $L_1$
is smooth on compact sets of $\R^3-\{ p\} $.
Our previous arguments show that $L_1$ is the vertical catenoid $M_{\infty }$.
As $p$ is a point in $L_1$, then $M_{\infty }$
must pass through $p$. We next analyze the intersection of $f_n(E(\a_n))$ with a ball $\B (p,\de )$
of small radius $\de >0$ so that $\B (p,\de)\cap M_{\infty }$ is a graphical disk with boundary
$\G_{\infty }$. For $n$ large,
we can assume that $f_n(\a_n)\subset \B (p,\de)$. Since $f_n(E(\a_n))$ is a properly embedded surface
of genus zero and $f_n(E(\a _n))\cap \partial \B (p,\de )$ consists of a single curve $\G_n$ such that
$\{ \G_n\} _n\to \G_{\infty }$, then we conclude that $f_n(\a _n)\cup \G_n$ bounds a compact annulus
in $f_n(E(\a_n))$; in fact, $f_n(E(\a_n))\cap \ov{\B}(p,\de )$ is this annulus.

We now arrive at the desired contradiction as follows.
Consider a horizontal plane $\Pi$ strictly below the height of $p$.
Since $M_{\infty }$ is the smooth limit of the $[f_n(E(\a_n))\cap \ov{\B}(n)]-\{ p\}$ away from $p$,
we conclude that the horizontal circle $M_{\infty }\cap  \Pi$ is arbitrarily close
to a simple closed, planar convex curve $c_n \subset f_n(E(\a_n))$.
Observe that $c_n$ can be joined to both $f_n(\a_n)$ and $f_n(\be_n)$
by arcs that do not intersect $f_n(\a_n)\cup f_n(\be_n)\cup c_n$ except at their
extrema. Therefore,
$f_n(\a_n)\cup f_n(\be_n)\cup c_n$ is the boundary of a compact planar
domain in $f_n(E(\a_n))$. Since $\a_n$ and $\be_n$ are both nontrivial in
$\ov{\D}-\{ 0\} $, then $c_n$ must bound a disk in $\ov{\D}-\{ 0\} $
and therefore, $c_n$ bounds in $E$ a punctured disk $T_n$ with the number of punctures being
positive and finite (depending on $n$) by the convex hull property.
Recall that the logarithmic growths of the annular ends of $f_n(E_n)$ are arbitrarily small in absolute
value for $n$ sufficiently large. As the flux of $f_n(E(\a_n))$ along $c_n$ is bounded
away from zero, we conclude that the number of punctures in $T_n$ is unbounded as
$n\to \infty $. Since $f_n(T_n)$ lies below the plane $\Pi $, $c_n=\partial [f_n(T_n)]$
is a convex horizontal curve inside $\Pi $ and the flux of any closed curve in $f_n(T_n)$
is vertical, then the L\'opez-Ros deformation argument implies that $T_n$ contains
just one puncture, which is a contradiction for $n$ large. This contradiction proves that
Case~(K1) does not occur.

Finally suppose that Case~(K2) occurs. This assumption implies that after passing to a subsequence,
the following properties hold.
\begin{enumerate}[(L1)]
\item
The surfaces $f_n(\Delta_{n-1})$ converge to a compact minimal annulus $\Delta_{\infty }$
bounded by two horizontal simple closed convex curves, and $\Delta_{\infty }$
 is close in the Hausdorff distance to a compact piece of a vertical catenoid. Similarly,
 $f_n(\Delta_{n})$ converge to a compact minimal annulus $\Delta^{\infty }$
bounded by two horizontal simple closed convex curves, and $\Delta^{\infty }$
 is $\ve$-close in the Hausdorff distance to $M_{\infty }(\ve)$.

 \item The curves $f_n(\a_{n})$ converge to the top boundary
component $\wh{\a}$ of $\Delta_{\infty}$.
\item For all $n\in \N$, the restriction of the injectivity radius function of
$f_n(E)$ to $f_n(E(\a_n))\cap \ov{\B}(n)$ is uniformly bounded away from zero
(independently of $n$; this is a consequence of the arguments in the first paragraph of the proof of Case~(K1)).
\end{enumerate}

We now divide the argument of why Case~(K2) leads to contradiction into two subcases, depending on
whether or not the sequence $\{ f_n(\Omega_n)\} _n$ has locally
bounded second fundamental form in $\R^3$.

First suppose that $\{  f_n(\Omega_n)\} _n$ has locally
bounded second fundamental form in $\R^3$.
By the arguments in the proof of Lemma~1.1 in~\cite{mr8},
after passing to a subsequence, the surfaces $f_n(\Omega_n)$ converge to a
minimal lamination $\mathcal{L}_1$ of $\R^3-(\wh{\a}\cup \wh{\be})$.
In fact, property~(L1) above together with the
definition of $f_n$ imply that
$ \{ f_n(\Int (\Omega_n\cup \Delta _{n-1}\cup \Delta_n)) \}_n $ converges to a minimal lamination $\mathcal{L}_2$
of $\R^3-(\wh{\a}_1\cup \wh{\be}_1)$, where
\[
\wh{\a}_1=\partial \Delta _{\infty }-\wh{\a},\quad \wh{\be}_1=\partial \Delta ^{\infty }-\wh{\be},
\]
and $\mathcal{L}_2$ contains the interior of both $\Delta_{\infty },\Delta^{\infty }$
as portions of its leaves. We will call $L(\Delta_{\infty })$ (resp. $L(\Delta^{\infty })$) the leaf of $\mathcal{L}_2$
that contains the interior of $\Delta_{\infty }$ (resp. of $\Delta^{\infty }$). Note that $L(\Delta _{\infty })$ might
coincide with $L(\Delta^{\infty})$. Also observe that neither $L(\Delta_{\infty })$ nor $L(\Delta^{\infty })$ are
stable, as both $\Delta_{\infty }$, $\Delta^{\infty }$ can be assumed to be unstable by choosing $\ve $
in the statement of Proposition~\ref{lemma4.15} sufficiently small. As $L(\Delta_{\infty }), L(\Delta^{\infty })$
are not stable, Theorem~1 in~\cite{mpr18} implies that $L(\Delta_{\infty }), L(\Delta^{\infty })$ are not limit leaves
of $\mathcal{L}_2$.
Also note that every limit leaf of $\mathcal{L}_1$ is contained in a limit leaf of $\mathcal{L}_2$,
and so, limit leaves of $\mathcal{L}_1$ are complete stable minimal surfaces, which are planes.
In particular, limit leaves of $\mathcal{L}_1$  and of $\mathcal{L}_2$ are the same.
This implies that the closure in $\R^3$ of each nonflat leaf of $\mathcal{L}_2$ is proper in $\R^3$,
in a halfspace or in a slab with boundary being limit leaves of $\mathcal{L}_1$. In particular,
the following surfaces with compact boundary are proper in $\R^3$, proper in an open halfspace or proper in an open slab:
\[
\left\{
\begin{array}{ll}
L(\Delta_{\infty })\cup \wh{\a}_1,\quad L(\Delta^{\infty })\cup \wh{\be }_1
& \mbox{if }L(\Delta_{\infty })\neq L(\Delta^{\infty }),\\
L(\Delta_{\infty })\cup \wh{\a}_1\cup \wh{\be}_1&
 \mbox{if }L(\Delta_{\infty })= L(\Delta^{\infty }).
\end{array}\right.
\]
Suppose that $L(\Delta_{\infty })\neq L(\Delta^{\infty })$. 
 In this setting, property (L3) above
 and the intrinsic version of the one-sided curvature estimates by Colding and Minicozzi
 (Corollary~0.8 in~\cite{cm35})
 imply that $L(\Delta_{\infty })\cup \wh{\a}_1$ has bounded Gaussian curvature in
 any small regular neighborhood of the limit set of $L(\Delta_{\infty })\cup \wh{\a}_1$
(see Lemma~1.2 in~\cite{mr8} for a similar argument using the extrinsic version of the
one-sided curvature estimates by Colding and Minicozzi). Therefore,
$L(\Delta_{\infty })\cup \wh{\a}_1$ is proper in $\R^3$, and the same holds for $L(\Delta^{\infty })\cup \wh{\be}_1$ by similar
arguments. In the case that $L(\Delta_{\infty })= L(\Delta^{\infty })$, the same reasoning gives that
$L(\Delta_{\infty })\cup \wh{\a}_1\cup \wh{\be}_1$ is proper in $\R^3$.

Observe that $L(\Delta_{\infty })$ has genus zero and one or two boundary curves,
each of which bounds an open convex horizontal disk
disjoint from $L(\Delta_{\infty })$ (this follows from the arguments in the proof of Assertion~\ref{asser3.11}).
If $L(\Delta _{\infty })$ has one boundary curve, then we contradict the Halfspace Theorem, as $L(\Delta_{\infty})$
lies in the halfspace $\{ x_3\leq 0\} $ (because $f_n(\Omega _n)$ has the same property for all $n$)
and $L(\Delta_{\infty })$ contains interior points with heights strictly greater than its boundary curve. Therefore,
$L(\Delta _{\infty })$ has two boundary curves (equivalently, $L(\Delta _{\infty })=L(\Delta^{\infty })$). By the
convex hull property, $L(\Delta _{\infty })$ is noncompact. As $L(\Delta _{\infty })-\Delta^{\infty}$ is contained
in $\{ x_3\leq 0\}$, then $L(\Delta _{\infty })$ has horizontal limit tangent plane at infinity.
Since both $\a_n,\be_n$ have the same horizontal component of their fluxes, and
the homothetic factors $\l_n$ in (\ref{eq:tildeEn}) tend to zero,
then we deduce that the fluxes of $L(\Delta _{\infty })$ along its
boundary curves are vertical. This implies that $L(\Delta_{\infty })$ has vertical flux,
since it has genus zero.
Therefore, $L(\Delta_{\infty })$ cannot have a finite positive number of  ends by the L\'opez-Ros deformation argument.

The previous paragraph implies that $L(\Delta_{\infty })$ has infinitely many ends. Let $D$ be a positive number such
that the boundary of $L(\Delta_{\infty })$ is contained in the ball $\B (D)$ of radius $D$ centered at the origin.
We claim that for every sequence $\{ \mu _n\} _n$ of positive numbers going to zero, the restriction
to $\mu _n[L(\Delta _{\infty })-\B(2D)]$ of the injectivity radius function of $\mu _nL(\Delta _{\infty })$
is greater that some positive constant (independent of $n$) times the distance to the origin. Otherwise,
there exists a sequence of points $x_n\in \mu _n[L(\Delta _{\infty })-\B(2D)]$ such that
\[
\frac{I_{\mu _nL(\Delta_{\infty })}(x_n)}{|x_n|}\to 0 \quad \mbox{as $n\to \infty$,}
\]
where $I_{\mu _nL(\Delta_{\infty })}$ stands for the injectivity radius function of $\mu _nL(\Delta_{\infty })$.
Since the last quotient is invariant under rescaling, the sequence $x_n/|x_n|$ lies in the unit sphere and the boundary
of $\mu _nL(\Delta_{\infty })$ shrinks to the origin as $n\to \infty $, we produce a sequence of blow-up points
on the scale of topology on $\Omega _n$, which is impossible by previous arguments (maximality of $\{ \Delta_m\} _m$). 
This proves our claim.

Since $L(\Delta_{\infty })$ has infinitely many ends, then $L(\Delta_{\infty })$ has infinite total curvature.
As $L(\Delta_{\infty })$ has compact boundary, then Theorem~1.2 in~\cite{mpr10} ensures that
$L(\Delta_{\infty })$ does not have quadratic decay of curvature, i.e., there exists a divergent sequence
$y_n\in L(\Delta_{\infty })$ such that
\[
|K_{L(\Delta_{\infty })}|(y_n)\cdot |y_n|^2\to \infty \quad \mbox{as $n\to \infty $,}
\]
where $K_{L(\Delta_{\infty })}$ denotes the Gaussian curvature of $L(\Delta_{\infty })$. Taking
$\mu _n=1/|y_n|$ and using the claim in the last paragraph, we conclude by Theorem~2.2 in~\cite{mpr8}
that after passing to a subsequence, the $\mu _nL(\Delta_{\infty })$ converge to a minimal parking garage
structure of $\R^3$. This is impossible, since the limit set of the $\mu _nL(\Delta_{\infty })$ lies in
$\{ x_3\leq 0\} $. This contradiction implies that $L(\Delta_{\infty })$ cannot have infinitely many ends, and thus,
case (K2) does not occur in the special case that $\{  f_n(\Omega_n)\} _n$ has locally
bounded second fundamental form in $\R^3$.

By the last sentence, it remains to prove that
the sequence $\{ f_n(\Omega _n)\} _n$ has locally
bounded second fundamental form in $\R^3$ provided that case (K2) happens. 
By property (L3) above and the 1-sided curvature estimates by
Colding-Minicozzi, we conclude that the norms of the second
fundamental forms of the $ f_n(\Omega _n\cup \Delta _{n-1}\cup \Delta _n)$ are bounded on some small
fixed compact regular neighborhood $W$ of $\Delta _{\infty }\cup \Delta ^{\infty}$.
Arguing by contradiction,
suppose, after extracting a subsequence, that there is a sequence of
points $q_n\in f_n(\Omega_n)-W$ that converges to
a point $q\in \rth -W$ where the norms of the second fundamental forms of the $f_n(\Omega_n)$
are greater than $n$.  By property (L3), Colding-Minicozzi theory in~\cite{cm24} (see also Figure~2 in \cite{mpr11})
implies that after extracting a subsequence, the following properties hold:
\begin{enumerate}[(M1)]
\item There is a positive number $\de$ less than one half of the distance in $\R^3$ from $q$ to $W$,
a relatively closed subset $\cS_{q,\de}$ of $\B(q,\de)$ and a minimal lamination $\cL_{q,\de}$ of $\B(q,\de)-\cS_{q,\de}$
such that $\B(q,\de)\cap f_n(\Omega_n)$ consists of disks with their boundary curves in
$\partial \B(q,\de)$ and a subsequence of these disks converges $C^{\alpha }$, $\a \in (0,1)$, to
$\cL_{q,\de}$ in $\B(q,\de)-\cS_{q,\de}$.
\item For each point $s\in \cS_{q,\de}$, there is a limit leaf of $\cL_{q,\de}$ that is a minimal disk punctured at $s$,
and the closure in $\B(q,\de)$ of the collection of all the limit leaves of $\cL_{q,\de}$ forms a minimal lamination
${\cal F}_{q,\de}$ of $\B(q,\de)$.
\end{enumerate}

We refer the reader to description (D) in Section 3 of~\cite{mpr11} for details on observations (M1), (M2).
Furthermore, straightforward diagonal arguments using
this just described local structure of the limit set of the $f_n(\Omega _n)$ near points in $\rth$ where the
norms of their second fundamental forms are becoming unbounded,
demonstrate that there exists a possibly singular minimal lamination $\cL'$ of $\rth -W$
and a relatively closed set $\cS\subset \cL'$ in $\rth-W$ (the set of points where $\cL'$ fails to
admit a local lamination structure) such that
after extracting a subsequence, the $f_n(\Omega _n)$ converge $C^\a$ ($0<\a<1$) on compact subsets of $\rth -[W\cup
\cS \cup S(\cL')]$ to $\cL'$, where $S({\cal L}')\subset {\cal L}'-\cS$ is the set of points where $\cL'$
admits a local lamination structure but the second fundamental forms of the surfaces $f_n(\Omega_n)$
blow up as $n\to \infty $ ($S({\cal L}')$ is called the singular set of convergence of the sequence).
In fact, since the second fundamental forms of the
$f_n(\Omega _n\cup \Delta_n\cup \Delta _{n-1})$
are uniformly bounded in $W$, we conclude that $\cL'$ can be extended to a possibly singular minimal lamination
$\cL'_1$ of $\R^3-(\wh{\a}_1\cup \wh{\be }_1)$ and the surfaces  $f_n(\Omega _n\cup \Delta_n\cup \Delta _{n-1})$ converge $C^{\a }$ to
$\cL'_1$ in $\R^3-[\wh{\a}_1\cup \wh{\be }_1\cup \cS\cup S({\cal L})]$. Furthermore, the singular set (resp. the singular set of
convergence) of $\cL'_1$ equals the singular set $\cS$ (resp. the singular set of convergence $S({\cal L}')$) of $\cL'$.
Additionally, the closure in $\rth -[\wh{\a}_1\cup \wh{\be }_1]$
of the sublamination of limit leaves of $\cL'_1$ is a (regular) minimal lamination ${\cal F}$ of $\rth-[\wh{\a}_1\cup \wh{\be }_1]$
with $\cS \cup S(\cL')\subset {\cal F}$. Observe that the leaves $L(\Delta _{\infty }),L(\Delta ^{\infty })$ of $\cL'_1$ that
contain respectively $\Int(\Delta _{\infty })$, $\Int(\Delta ^{\infty })$, are unstable and thus, they are not leaves of ${\cal F}$.
In fact, $\Delta _{\infty }\cup \Delta ^{\infty }$ can be assumed to lie at a positive distance from ${\cal F}$ after slightly changing the
compact domains $\Delta _{\infty },\Delta ^{\infty }$. This implies that the leaves of ${\cal F}$ are complete in $\R^3$, and since
they are stable, then these leaves are planes.

Since ${\cal F}$ contains $\cS\cup S({\cal L}')$, then the norms of the second fundamental forms
of the surfaces $f_n(\Omega_n)$ are locally bounded in the open set $\R^3-{\cal F}$, which is a countable
union of open slabs and open halfspaces.
As the top boundary component of $L(\Delta ^{\infty })$ is $\wh{\be }_1$,
it follows that $L(\Delta ^{\infty })$
is contained in the halfspace $\{x_3\leq x_3(\wh{\be }_1)\}$ and $L(\Delta ^{\infty })$ is proper in the open slab
$A$ of $\rth$ with boundary $\{x_3= x_3(\wh{\be }_1)\}\cup P$, where $P$ is the plane in $\cF$
with largest $x_3$-coordinate (which exists since $\cS\cup S({\cal L}')\neq \mbox{\O}$).
Observe that $L(\Delta^{\infty })$ is proper in $A$ (otherwise there exists a plane in $\mathcal{F}\cap A$).

Next we will show that $L(\Delta^{\infty })\cup \partial L(\Delta^{\infty })$ is incomplete.
Arguing by contradiction, suppose that $L(\Delta^{\infty })\cup \partial L(\Delta^{\infty })$ is complete.
As
the injectivity radius function of $f_n(\Omega_n)\cup \Delta _n\cup \Delta _{n-1}$ restricted to $f_n(\Omega_n)$
is uniformly bounded away from zero (otherwise we could find a sequence of blow-up points in $\Omega _n$, which is
impossible) and $L(\Delta^{\infty })\cup \partial L(\Delta^{\infty })$ is assumed to be complete,
then the injectivity radius function of $L(\Delta^{\infty })$ is bounded away from zero outside any neighborhood
of its boundary. In this setting, Theorem~\ref{proper1end} implies that $L(\Delta^{\infty })$ is proper in $\R^3$.
To find the desired contradiction, we distinguish two cases; first suppose that $\partial L(\Delta^{\infty })=\wh{\be}_1$.
In this case, $L(\Delta^{\infty })\cup \wh{\be}_1$ has full harmonic measure by Lemma~2.2 in~\cite{ckmr1}. But
the third coordinate function of $L(\Delta^{\infty })\cup \wh{\be}_1$ is a bounded harmonic function with
constant boundary values $x_3(\wh{\be}_1)$ and values at interior points strictly below $x_3(\wh{\be}_1)$,
which is a contradiction. Second, suppose that
$\partial L(\Delta^{\infty })=\wh{\a}_1\cup \wh{\be}_1$; in this case, $L(\Delta^{\infty })$ has finitely many
ends by the same Lemma~2.2 in~\cite{ckmr1}, and thus, these ends are asymptotic to horizontal planes. Now
the L\'opez-Ros deformation argument applied to $L(\Delta^{\infty })\cup \partial L(\Delta^{\infty })$ 
leads to contradiction as this noncompact
embedded minimal surface has vertical flux and two convex horizontal boundary components. Therefore,
$L(\Delta^{\infty })\cup \partial L(\Delta^{\infty })$ must be incomplete.

Since $L(\Delta^{\infty })\cup \partial L(\Delta^{\infty })$ is incomplete and $L(\Delta^{\infty })$ is
proper in $A$, then
$L(\Delta^{\infty })$ contains a proper arc $\tau\colon [0,1)\to L(\Delta^{\infty })$ of finite length with 
its limiting endpoint $q\in \cS\cap P$;  previous arguments also imply that $L(\Delta ^{\infty })$
has vertical flux. If there exists another point $q'\in (\cS\cap P)-\{ q\} $ where $L(\Delta ^{\infty })$ fails to be complete,
one can construct a sequence of
connection loops $\sigma_k\subset L(\Delta ^{\infty })$, $k\in \N$, that converge as $k\to \infty $ with multiplicity 2
away from $\{q,q'\}$ to an compact embedded arc $\sigma$ in $P-(\mathcal S-\{ q,q'\} )$ that joins
$q$ to $q'$, and such that the fluxes of these connection loops on $L(\Delta ^{\infty })$ converge to a nonzero
horizontal vector, which contradicts that $L(\Delta ^{\infty })$ has vertical flux.
Therefore, $L(\Delta ^{\infty })$ only has $q\in \cS\cap P$ as a point of incompleteness.
By the extrinsic 1-sided curvature estimates of Colding-Minicozzi,
there is an $\ve'>0$ small such that the intersection of $L(\Delta ^{\infty })$ with the
$\ve'$-neighborhood $P(\ve')$ of $P$ is a disk that contains a pair of disjoint $\infty$-valued graphs $\Sigma_1,\Sigma_2$
(with respect to polar coordinates in $P$ centered at $q$, each $\Sigma_i$ is an $\infty $-valued graph over
an annulus in $P$ centered at $q$ with inner radius $1$ and arbitrarily large radius, $i=1,2$) and both $\Sigma_1$,
$\Sigma _2$ spiral together into $P$ by above. Furthermore, $\Sigma_1$ and $\Sigma _2$ can be joined by
curves in $L(\Delta ^{\infty })\cap P(\ve')$ with uniformly bounded length. In this setting,
Corollary~1.2 in \cite{cm26} (see especially the paragraph just after this corollary)
leads to a contradiction. This contradiction completes the proof that Case~(K2) cannot occur.

Since we have discarded Cases~(K1) and (K2) above, then 
the curves $f_n(\a_n)$ diverge to infinity in $\rth$. Thus, the first paragraph in
the proof of Lemma~\ref{claim4.12} ensures that the conclusions of Lemma~\ref{claim4.12} hold.
\end{proof}

Recall that we had called $\Omega _n$, $n\in \N$, to the components of $E- \bigcup_{m\in \N\cup \{ 0\} } \Delta_m$,
where the index $n$ is chosen so that (\ref{eq:alphabeta}) holds, and that in order to prove item~4 of 
Proposition~\ref{lemma4.15}, it suffices to find a contradiction with the following assumption:
\begin{enumerate}[$(\clubsuit)$]
\item The number of components $\Omega _n$ with has at least two annular ends is infinite.
\end{enumerate}
Suppose that $(\clubsuit)$ holds, and let $\Omega _{n(i)}$, $i\in \N$, denote the subsequence of the $\Omega _n$ 
with at least two annular ends each.
Since any path in $\Omega _{n(i)}$ joining two consecutive annular ends
intersects the inverse image by the Gauss map of $E$ of the horizontal equator in the sphere,
we conclude that there exists some point $x_{i} \in
\Omega _{n(i)}$ where the tangent plane to $\Omega _{n(i)}$ is vertical. 
We can assume that $x_{i}$ is chosen so
that it is an extrinsically closest such point to the upper boundary
component $\be _{n(i)}$ of $\Omega _{n(i)}$. Let
 $d_{\ov{\partial }}(i)>0$ be
the extrinsic distance from $f_{n(i)}(x_i)$ to $f_{n(i)}(\be_{n(i)})$,
 where $f_n$ is the dilation defined just before  Lemma~\ref{claim4.12}.
Since $f_{n(i)}(\Omega _{n(i)})$ converges as $i\to \infty $ to
$M_{\infty }\cap \{ x_3\leq \frac{-1}{\ve }\} $ by  Lemma~\ref{claim4.12}
(recall that $M_{\infty}$ is the vertical catenoid whose waist circle is the unit circle in the $(x_1,x_2)$-plane)
and the tangent plane to $f_{n(i)}(\Omega _{n(i)})$ at $f_{n(i)}(x_i)$ is vertical, 
it follows that $d_{\ov{\partial }}(i)\to \infty $ as $i\to \infty $.

Now we apply a homothety to obtain the surface
\begin{equation}
\label{eq:Thetatilde}
\ov{\Theta}_{n(i)} =\frac{1}{d_{\ov{\partial }}(i)}
[f_{n(i)}(\Omega _{n(i)})-f_{n(i)}(b_i)],
\end{equation}
where $b_i$ is a closest point to $x_i$ in $\be_{n(i)}$. The surface $\ov{\Theta}_{n(i)}$
is an embedded, minimal planar domain passing through the origin,
with two horizontal, almost circular boundary components
and a positive number of ends (at least two), all with negative logarithmic growth;
after extracting a subsequence, let ${\bf x}\in \R^3$ be the limit of the points 
$\frac{1}{d_{\ov{\partial }}(i)}[f_{n(i)}(x_i)-f_{n(i)}(b_i)]$
and note that ${\bf x}$ is at a distance 1 from the origin.
Let
\[
\ov{\a }_{n(i)}=\frac{1}{d_{\ov{\partial }}(i)}[f_{n(i)}(\a_{n(i)})-f_{n(i)}(b_i)], 
\quad \ov{\be }_{n(i)}=\frac{1}{d_{\ov{\partial }}(i)}[f_{n(i)}(\be_{n(i)})-f_{n(i)}(b_i)]
\]
be the respective lower and upper boundary components of $\ov{\Theta}_{n(i)}$.

Since the lengths of $ \ov{\a}_{n(i)},\ov{\be}_{n(i)}, $ are shrinking to zero, 
then after extracting a subsequence, the
$\ov{\be}_{n(i)}$ converge to $\vec{0}$ and the
$\ov{\a}_{n(i)}$ either converge to a point $q(\a)\in \R^3$ or they diverge in $\rth$.  It follows that
$\{ \ov{\Theta }_{n(i)}\cap \ov{\B}(i)\} _i$ is a sequence of compact genus zero minimal surfaces
which is locally simply connected in $\R^3-W$, where $W=\{\vec{0},q(\a)\}$  in the case that $q(\a)$ exists
and $W=\{\vec{0}\}$ otherwise.
By Theorem~2.2 in~\cite{mpr8}, after choosing a subsequence, the surfaces $\ov{\Theta }_n\cap \ov{\B}(n)$
converge to a minimal lamination $\cL$ of $\rth-W$
and $\cL$ extends to a minimal lamination $\ov{\cL}$ of $\rth$. Notice that $\ov{\cL}$ contains a complete leaf
$L_{\bf x}$ passing through ${\bf x}$.
Since $\ov{\Theta}_{n(i)}$ is contained in the halfspace $\{ x_3\leq 1\} $ for $i$ large
(since $\frac{1}{d_{\ov{\partial }}(i)}f_{n(i)}(\Delta _{n(i)})$ shrinks to $\vec{0}$),
then Theorem~2.2 in~\cite{mpr8} implies that all of the leaves in $\ov{\cL}$ are horizontal planes and that the 
sequence of norms of the second fundamental forms of the surfaces $\ov{\Theta }_{n(i)}\cap \ov{\B}(i)$ is locally bounded 
in $\rth-W$. In particular, $L_{\bf x}$ is a horizontal plane. Since the tangent plane of $\ov{\Theta }_{n(i)}\cap \ov{\B}(i)$ 
is vertical at $x_i$ for each~$i$, the sequence $\{ \ov{\Theta }_{n(i)}\cap \ov{\B}(i)\} _i$ cannot have uniformly bounded curvature
in any fixed sized neighborhood of ${\bf x}$, which implies that ${\bf x}=q(\a)$ (in particular, $q(\a )$ exists).

We next explain how to refine the arguments in the last paragraph to conclude the following property.
\begin{claim}
\label{claim4.14}
Once we restrict to the subsequence $\{ \ov{\Theta }_{n(i)}\cap \ov{\B}(i)\} _i$ that limits to $\cL$,
for any sequence of points $y_i\in \Omega_{n(i)}$ with vertical tangent plane 
(not necessarily the closest such points in $\Omega_{n(i)}$ to the upper boundary component $\be_{n(i)}$), the points
$\frac{1}{d_{\ov{\partial }}(i)}f_{n(i)}(y_i)$ converge to~${\bf x}$.
\end{claim}
\begin{proof}
Let $d_{\ov{\partial }}(y_i,i)$ be the extrinsic distance from $f_{n(i)}(y_i)$ to $f_{n(i)}(\be _{n(i)})$, which is
attained at some point $f_{n(i)}(b'_i)$ with $b'_i\in \be_{n(i)}$. Observe that
$d_{\ov{\partial }}(y_i,i)\geq d_{\ov{\partial }}(i)$ and the arguments before Claim~\ref{claim4.14}
prove that after choosing a subsequence, the curves 
\[
\frac{1}{d_{\ov{\partial }}(y_i,i)}[f_{n(i)}(\a_{n(i)})-f_{n(i)}(b'_i)]
\]
converge to the same (subsequential) limit ${\bf y}$ of the points 
\[
\frac{1}{d_{\ov{\partial }}(y_i,i)}[f_{n(i)}(y_i)-f_{n(i)}(b'_i)],
\]
which in turn is a 
point in the unit sphere. Since the curves $\frac{1}{d_{\ov{\partial }}(y_i,i)}[f_{n(i)}(\a_{n(i)})-f_{n(i)}(b'_i)]$ converge to the same 
limit point as the $\ov{\a}_{n(i)}$ (this last limit was called $q(\a)$ in the preceeding
paragraph), we conclude that ${\bf y}=q(\a)={\bf x}$ and that $\frac{d_{\ov{\partial }}(y_i,i)}{d_{\ov{\partial }}(i)}$ tends
to $1$ as $n\to \infty $, from where we obtain that the 
$\frac{1}{d_{\ov{\partial }}(i)}[f_{n(i)}(y_i)-f_{n(i)}(b'_i)]$ converge to ${\bf x}$.
The fact that the whole original sequence $\{ \frac{1}{d_{\ov{\partial }}(i)}[f_{n(i)}(y_i)-f_{n(i)}(b'_i)]\} _i$ converges to ${\bf x}$
(i.e., we do not need to pass to a subsequence of the $y_i$ once we restrict to
the subsequence that produces the convergent sequence $\{ \ov{\Theta }_{n(i)}\cap \ov{\B}(i)\} _i$ to $\cL$)
follows from arguing by contradiction and passing to further subsequences. 
\end{proof}

Now return to the definition of the point $x_i\in \Omega _{n(i)}$. Define the related point
$z_{i}$ as an extrinsically  farthest point in $\Omega _{n(i)}$ to its lower boundary component
$\a_{n(i)}$  where the tangent plane to $\Omega _{n(i)}$ is vertical; since the set of all such points $z$ 
is compact in $\Omega_{n(i)}$ and a positive distance from $\a_{n(i)}$, the point $z_i$ exists. 
We now apply  our previous arguments with $z_{i}$ in place of $x_i$: consider for each $i\in \N$
the related surface
\[
\underline{\Theta }_{n(i)}=\frac{1}{d_{\underline{\partial}}(z_i,i)}
[f_{n(i)}(\Omega _{n(i)})-f_{n(i)}(a_{i})],
\]
where $d_{\underline{\partial }}(z_i,i)>0$ is the extrinsic distance
from $f_{n(i)}(z_{i})$ to $f_{n(i)}(\a_{n(i)})$ and $a_{i}$ is a point in $\a_{n(i)}$ closest to $z_i$. 
Observe that we do not know if $d_{\underline{\partial }}(z_i,i)\to \infty $ as $i\to \infty $
but we do know (from the previous paragraph) that the distances in $\rth$ from the origin to the top boundary 
component $\frac{1}{d_{\underline{\partial}}(z_i,i)}\left( f_{n(i)}(\be _{n(i)})-f_{n(i)}(a_{i})\right) $ of $\underline{\Theta }_{n(i)}$ 
diverges to infinity  as $i\to \infty $.

\begin{claim}
\label{claim4.15bis}
The sequence $d_{\underline{\partial }}(z_i,i)$ is bounded independently of $i$.
\end{claim}
\begin{proof}
Arguing by contradiction, suppose after choosing a subsequence, $d_{\underline{\partial }}(z_i,i)\geq 2i$.
By Claim~\ref{claim4.14}, the extrinsic distance from $z_i$ to $\be_{n(i)}$ is much larger than
the extrinsic distance from $z_i$ to $a_{i}$. Therefore, we can assume that the curve 
\[
\frac{1}{d_{\underline{\partial }}(z_i,i)}[f_{n(i)}(\be_{n(i)})-f_{n(i)}(a_i)]
\]
lies outside the ball $\B(i^2)$.
Consider the compact minimal surfaces $\underline{\Theta }_{n(i)}\cap \ov{\B}(i)$, which form 
a locally simply connected sequence of surfaces in $\R^3-\{ \vec{0}\} $ by our previous arguments. 
After extracting a subsequence, let ${\bf z}\in \R^3$ be the limit of the points 
\[
\frac{1}{d_{\underline{\partial }}(z_i,i)}[f_{n(i)}(z_i)-f_{n(i)}(a_i)],
\]
which is a point in the unit sphere. As before, Theorem~2.2 in~\cite{mpr8} implies that after passing to a subsequence, the 
$\underline{\Theta }_{n(i)}\cap \ov{\B}(i)$ converge to a 
minimal lamination $\mathcal{L}_+$ of $\R^3-\{ \vec{0}\} $, which extends to a lamination 
$\ov{\mathcal{L}_+}$ of $\R^3$ with a leaf $L_{\bf z}$ passing through ${\bf z}$. 

Now consider the surfaces 
\[
\Sigma_i=\frac{1}{d_{\underline{\partial}}(z_i,i)}\left[
f_{n(i)}(\Omega _{n(i)}\cup \Delta_{n(i)-1}\cup \Omega_{n(i)-1})-f_{n(i)}(a_i)\right], \quad i\in \N.
\]
For each $i$, $\Sigma _i$ is a noncompact planar domain bounded by two convex horizontal curves, 
that we call 
\[
\partial \Sigma _i^+=\frac{1}{d_{\underline{\partial}}(z_i,i)}
[f_{n(i)}(\be_{n(i)})-f_{n(i)}(a_i)],
\quad 
\partial \Sigma _i^-=\frac{1}{d_{\underline{\partial}}(z_i,i)}
[f_{n(i)}(\a_{n(i)-1})-f_{n(i)}(a_{i})],
\]
and $x_3(\partial \Sigma _i^-)<x_3(\partial \Sigma _i^+)$. Previous arguments show that 
the sequence of curves $\{ \partial \Sigma_i^-\} _i$ either converges to some point $q_{-}\in \R^3$ (that
would then lie in $\{ x_3\leq 0\} $, possibly being $\vec{0}$), or else $\{ \partial \Sigma_i^-\} _i$
diverges in $\R^3$. Maximality of the family $\{ \Delta_m\}_m$ implies as above that 
the sequence of surfaces $\{ \Sigma_i\cap \ov{\B}(i)\} _n$ is locally simply connected in 
$\R^3-W$ where $W=\{ \vec{0},q_{-}\} $ if $q_{-}$ exists, and $W=\{ \vec{0}\} $ otherwise.
Therefore, Theorem~2.2 in~\cite{mpr8} implies that after passing to a subsequence, the 
$\Sigma_i$ converge to a minimal lamination $\mathcal{L}$ of $\R^3-W$, which extends to a lamination 
$\ov{\mathcal{L}}$ of $\R^3$. Note that $\ov{\cL}$ contains $\ov{\cL}_+$ as a sublamination.
Also observe that the same arguments applied to the sequence of surfaces 
\[
\frac{1}{d_{\underline{\partial}}(z_i,i)}\left[
f_{n(i)}(\Omega_{n(i)-1})-f_{n(i)}(a_i)\right].
\]
give that  the surfaces $\frac{1}{d_{\underline{\partial}}(z_i,i)}\left[
f_{n(i)}(\Omega_{n(i)-1})-f_{n(i)}(a_i)\right]\cap \ov{\B}(i)$ converge 
to a minimal lamination $\mathcal{L}_-$ of $\R^3-W$, 
which extends to a lamination $\ov{\mathcal{L}_-}$ of $\R^3$. 
Moreover, $\ov{\cL_-}$ contains a leaf that passes 
through the origin, and $\ov{\cL}_-$ is a sublamination of $\ov{\cL}$.
By construction, the $\frac{1}{d_{\underline{\partial}}(z_i,i)}\left[ f_{n(i)}(\Omega_{n(i)-1})-f_{n(i)}(a_i)\right]$ 
lie in $\{ x_3\leq 0\} $ for each $i\in \N$, 
and thus, $\ov{\cL_-}$ is also contained in $\{ x_3\leq 0\} $. In particular, $\{ x_3=0\} $ is a leaf of $\ov{\cL}_-$. 
In this setting, Theorem~2.2 in~\cite{mpr8} implies 
that all leaves of $\mathcal{L}_-$ are horizontal planes, and thus, the same theorem gives that all leaves of $\ov{\cL}$ 
are horizontal planes. In particular, $L_{\bf z}$ is a horizontal plane.

Since the tangent plane to $\Omega_{n(i)}$ at $z_i$ is vertical, then the convergence of the 
$\underline{\Theta}_{n(i)}\cap \ov{\B}(i)$ to $L_{\bf z}$ cannot be smooth around ${\bf z}$. 
This property and Theorem~2.2 in~\cite{mpr8} 
imply that $\ov{\cL}_+$ is a foliation of $\R^3$ by horizontal planes and the $\underline{\Theta }_{n(i)}\cap \ov{\B}(i)$ 
converge to $\ov{\cL}$ outside the origin and one or two vertical lines (this is the singular set of convergence), one of which
passes through~${\bf z}$. This is impossible, since the compact surfaces $\underline{\Theta}_{n(i)}\cap [\ov{\B}(i^2/2)-\B(2)]$ have
uniformly bounded Gaussian curvature (this follows since the last surfaces do not have vertical tangent planes, and so
they are locally graphical hence stable, and by curvature estimates for stable minimal surfaces). Now Claim~\ref{claim4.15bis}
is proved.
\end{proof}

As a consequence of Lemma~\ref{claim4.12}, the diameter of the compact surface
\[
\frac{1}{d_{\underline{\partial}}(z_i,i)}\left[ f_{n(i)}(\Delta_{n(i)-1})-f_{n(i)}(a_i)\right]
\]
tends to zero as $i\to \infty $. In particular, the diameter of the top boundary curve of the last surface
tends to zero, which implies that 
\[
\frac{1}{d_{\underline{\partial}}(z_i,i)}\left[ f_{n(i)}(\a_{n(i)})-f_{n(i)}(a_i)\right] \to \vec{0} \quad
\mbox{as $i\to \infty $.}
\]
On the other hand, Claim~\ref{claim4.14} implies that 
\[
\frac{1}{d_{\underline{\partial}}(z_i,i)}\left[ f_{n(i)}(\be_{n(i)})-f_{n(i)}(a_i)\right] \quad
\mbox{diverges in $\R^3$ as $i\to \infty $.}
\]
Therefore, Theorem~2.2 in~\cite{mpr8} implies that after passing to a subsequence, the 
$\underline{\Theta }_{n(i)}\cap \ov{\B}(i)$ converge to a 
minimal lamination $\mathcal{L}_+$ of $\R^3-\{ \vec{0}\} $. From this point, we can repeat 
verbatim the arguments in the proof of Claim~\ref{claim4.15bis} to obtain a contradiction.
This contradiction shows that property $(\clubsuit)$ cannot hold, and so, item~4 of Proposition~\ref{lemma4.15}
is proven.

Finally, item~5 of Proposition~\ref{lemma4.15} follows from the fact that
the Gaussian curvature functions of the domains $\Omega_n$ and $\Delta_n$
become uniformly small as $n \rightarrow \infty$. Now the proof of
Proposition~\ref{lemma4.15} is complete.
\end{proof}

\begin{proposition}
 \label{propos3.13}
 Items 4, 5 and 6 of Theorem~\ref{thm1.3} hold in the Case~(G1) when $M_{\infty }$ is a catenoid.
 In particular, Theorem~\ref{thm1.3} holds in this case.
 \end{proposition}
 \begin{proof}
 Recall that in the paragraph just before Proposition~\ref{propos4.11}, we explained that items 1, 2, 3 
 of Theorem~\ref{thm1.3} hold in the Case~(G1) when $M_{\infty }$ is a catenoid.
The description of $E$ as a union of domains $\Delta _n$ and $\Omega_n$ as given in
Proposition~\ref{lemma4.15} implies that item~4 of Theorem~\ref{thm1.3} holds. Item~5
of the same theorem follows from item~5 of Proposition~\ref{lemma4.15}. Finally, the arguments in the proof of Proposition~\ref{propos4.11} can be easily adapted to the
current situation, with the only change of the annular regions of ``Riemann type''
by similar annular regions of ``catenoid type'', namely regions of the type of the compact
annuli $\Delta_n=\Delta_n(\ve)$ that appear in Proposition~\ref{lemma4.15}, each of which
contains the image of a conformal embedding $f_n(\esf^1\times [0,\de ])$ for some $\de >0$
 independent
of $n$ (here we are using the notation in the proof of Proposition~\ref{propos4.11}). This finishes the proof of Proposition~\ref{propos3.13}.
  \end{proof}

\section{The proof of Theorem~\ref{thm1.2}.} 
\label{sec5}
Suppose $M \subset \rth$ is a complete, embedded
minimal surface with finite genus, an infinite number of ends and
compact boundary.

We first check that $M$ has at most two simple limit ends. Arguing
by contradiction, suppose $M$  has at least three simple limit ends,
say { $\bf e_1, e_2, e_3$}. By Theorem~\ref{thm1.3}, we can choose
three pairwise disjoint, properly embedded representatives $E_1,
E_2, E_3 \subset M$, representing {$\bf e_1, e_2, e_3$}
respectively, such that each one satisfies, after a possible
rotation, the conclusions of Theorem~\ref{thm1.3}. Embeddedness of $M$ implies that
after a rotation, the annular ends of $E_1,E_2,E_3$ may be assumed to be asymptotic to ends of
horizontal planes and catenoids with vertical axes. Furthermore, after a possible
reindexing, we may assume that the ends $E_1, E_2$ are
simple top limit ends, that $\partial E_1$ is a simple closed
curve in the $(x_1, x_2)$-plane and that $\partial E_2$ has constant
positive $x_3$-coordinate.

Let $D_{E{_1}} \subset \{ x_3=0\} $ be the planar disk with $\partial D_{E_{1}} =
\partial E_1$ and let $X_1$ be the closure of
the component of $\rth - (E_1 \cup D_{E{_1}})$ that lies above
$D_{E_1}$ locally near $D_{E_1}$. Similarly, we can define a horizontal disk $D_{E_{2}}$ with
$\partial D_{E_{2}} = \partial E_2$ and the related closed component
$X_2$ of $\rth - (E_2 \cup D_{E_2})$ above $D_{E_2}$.

An elementary topological analysis applied to  the topological
picture of $E_1$ and $E_2$ given in item~4 of
Theorem~\ref{thm1.3} shows, after possibly reindexing $E_1$ and
$E_2$ and replacing $E_1$ and $E_2$ by representing subends, that
$D_{E_2} \cap E_{1} \neq \mbox{\rm \O }$ and $X_2$ contains a
representative $E'_1 \subset E_1$ of the limit end of $E_1$ with
$\partial E'_{1} \subset D_{E_{2}} \subset \partial X_2$. Let $X_3$
be the closure of the component of $X_2 - E'_1$ which has $\partial
E_2$ in its boundary. The piecewise smooth surface $\partial X_3$ is
a good barrier for solving least-area problems in $X_3$ (Meeks and
Yau~\cite{my2}), see Figure~\ref{figure3}.
\begin{figure}
\begin{center}
\includegraphics[width=8.9cm]{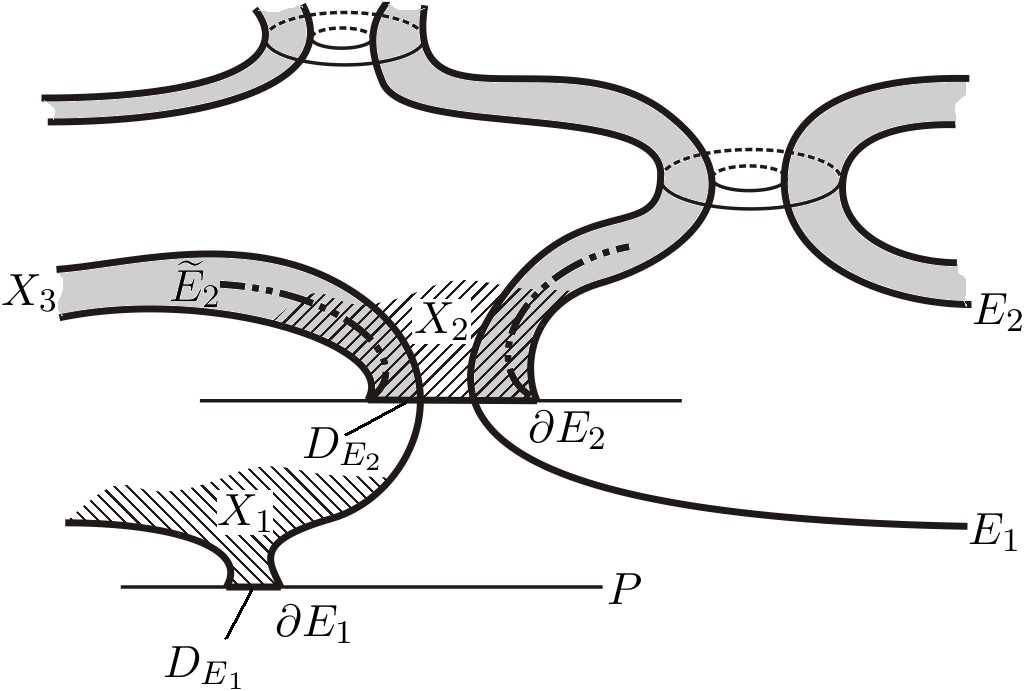}
\caption{The area-minimizing surface $\widetilde{E}_2$ is trapped
between two simple top limit ends $E_1,E_2$.} \label{figure3}
\end{center}
\end{figure}

Let $\widetilde{E}_{2}$ be a noncompact, properly embedded surface
of least-area in $X_3$ with $\partial \widetilde{E}_2 = \partial E_2
\subset \partial X_3$. By a result of Fischer-Colbrie~\cite{fi1},
the orientable, stable minimal surface $\widetilde{E}_2$ has finite
total curvature. Since $\widetilde{E}_2$ is contained in $X_2$, it
must have a finite number of ends, all of which are annuli and which
are parallel to the planar and catenoidal ends of $E_2$. Since
points of $\widetilde{E}_2$ near $D_{E_2}$ have $x_3$-coordinates
which are larger than the constant value $x_3(D_{E_2})$,
$\widetilde{E}_2$ must have a highest end which has positive
logarithmic growth by the maximum principle applied to the harmonic
function $x_3|_{\widetilde{E}_2}$. Hence, $\widetilde{E}_2$ has a
catenoid-type end representative $F$ of positive logarithmic growth.
Since the annular ends of $E_1$ have nonpositive logarithmic
growth, then none of the annular ends of $E_1$ lie above $F$. This
implies that $E_1$ lies below the region of $\R^3$ bounded by the
union of $F$ and a horizontal disk with boundary in $F$. Since $E_1$
also lies above some catenoid end of negative logarithmic growth,
the results in~\cite{ckmr1} imply that $E_1$ has quadratic area
growth. By the monotonicity formula, each annular end in $E_1$
contributes with at least $\frac{\pi }{2}R^2$ to the area growth of
$E_1$ in each ball $\B (R)$ for $R>0$ large. Hence, $E_1$
has a finite number of ends. This contradiction shows that $M$
cannot have more than two simple limit ends, which is item~1
of Theorem~\ref{thm1.2}.

Next we prove item~2 of Theorem~\ref{thm1.2}. If $M$ is
properly embedded in $\rth$, then the results in~\cite{ckmr1} imply
$M$ has one or two limit ends, which are the top and/or bottom ends
in the ordering of the ends of $M$. On the other hand, if $M$ has
one or two limit ends, then these limit ends are simple limit ends,
and so, these limit ends have proper representatives by
Theorem~\ref{thm1.3}. The remaining finite number of ends of $M$ are
then annuli, each of which is proper (see Theorem~\ref{thmmr} in
Section 4). Hence, $M$ is proper, which proves item~2 in
Theorem~\ref{thm1.2}.

Concerning item~3, suppose now that $M$ has a countable number
of limit ends. A result proven in pages 288, 289 of~\cite{mpe1}
states that the space of ends of $M$ embeds topologically as a
totally disconnected, closed subset $A$ of the closed unit interval
$I = [0,1]$. Since the set of limit points $L_A$ of $A$ is a closed
countable subspace of the metric space $I$ (and hence $L_A$ is
complete), Baire's theorem implies that $L_A$ contains a countable
dense set of isolated points (see Lemma~\ref{Baire} below). In
particular, if $L_A$ has at least three points, then $M$ has at
least three simple limit ends. Since $M$ cannot have more than two
simple limit ends by item~1 of Theorem~\ref{thm1.2}, then
$L_A$ consists of one or two points, and so $M$ has one or two limit
ends, each of which is a simple limit end. Hence, part~3-A of
Theorem~\ref{thm1.2} holds. As $M$ has at most two limit ends,
then item~2 of Theorem~\ref{thm1.2} implies that $M$ is
properly embedded in $\R^3$, which is part~{3-B}.

If $M$ has exactly two limit ends, then these limit ends are simple.
By the proof of item~1 of Theorem~\ref{thm1.2}, we deduce
that after a rotation of $M$, these simple limit ends have
representatives $E_1, E_2$, where $E_1$ is a top limit end of $M$
and $E_2$ is a bottom limit end of $M$. By item~1 of
Theorem~\ref{thm1.3}, the  annular ends of $E_1$ have nonpositive
logarithmic growths and the annular ends of $E_2$ have nonnegative
logarithmic growths. Thus, the embeddedness of $M$ implies that all
the annular ends of $M$ must have zero logarithmic growth, which
means that they are planar, and item~{3-C} is proved.

Now assume $\partial M=\mbox{\rm \O }$. Since $M$ has finite genus, then
the main result in~\cite{mpr4} insures that $M$ has two limit ends
and is recurrent for Brownian motion, which is part~{3-D}.

We finally prove item~{3-E} of Theorem~\ref{thm1.2}. Assume
$\partial M\neq \mbox{\rm \O }$. If the annular ends of $M$ are
horizontal and planar, there exists a horizontal plane $P$ that
intersects $M$ transversely in a finite number of simple closed
curves, and $P$ can be chosen to lie above $\partial M$. Hence, the
closure $\Sigma $ of each component of $M-P$ is a properly embedded
minimal surface with compact boundary and $\Sigma $ is contained in
a closed halfspace of $\rth$. Theorem~3.1 in~\cite{ckmr1} implies
that such a $\Sigma $ is a parabolic surface with boundary. Since
there are a finite number of such closed components $\Sigma $, and
the union of these components along related compact boundary
components is $M$, we conclude that $\partial M$ has full harmonic
measure, and so item~{3-E} holds provided that all of the
annular ends of $M$ are horizontal and planar.

If there exists an annular end with nonzero (say negative)
logarithmic growth, then this end is asymptotic to the end of a
negative half catenoid, and so, there exists a horizontal plane $P$
whose intersection with this catenoidal end is an almost circle, and
the end has a representative $E$ with $\partial E\subset P$ such that $E$ is
graphical over the outside of the open planar disk $D\subset P$
whose boundary is $\partial E$. We may assume that $P$ is at height zero.
The complement of the topological
plane $E\cup D$ in $M$ consists of several components,
each one with nonempty boundary contained in $M\cap D$. Since $M$
is proper, then $M\cap D$  is compact. In particular, $M-(E\cup
D)$ has a finite number of components. Let $\Sigma
_1,\ldots ,\Sigma _k$ be the components of $M-(E\cup
D)$ which lie below $E\cup D$. For each $i=1,\ldots
,k$, the surface with boundary $\Sigma _i$ is parabolic, since its
third coordinate function is a proper negative harmonic function. By
items~{3-A} and {3-C}, the surface $M$ has exactly one limit
end. Since a limit end of a properly embedded minimal surface in
$\R^3$ cannot lie below a catenoidal end of negative logarithmic
growth (see Lemma~3.6 in~\cite{ckmr1}), then the limit end of $M$
has a representative of genus zero $E_T$ which lies above $E\cup D$. 
In particular, the limit end of $M$ is its top
end. By item~{6} of Theorem~\ref{thm1.3}, the representative
$E_T$ is parabolic. Let $\Omega $ be the closure of one of the
(finitely many) components of $M-(E_T\cup \Sigma _1\cup \ldots \cup
\Sigma _k)$. Since $\Omega $ has a finite number of ends, each of
which is asymptotic to an end of a plane or half catenoid, then
$\Omega $ has quadratic area growth. Therefore, $\Omega $ is also a
parabolic surface with boundary. As $M$ is a finite union of
parabolic surfaces with boundary along their common compact
boundaries, we deduce that $M$ has full harmonic measure on its
boundary. This finishes the proof of Theorem~\ref{thm1.2}.
\hfill\penalty10000\raisebox{-.09em}{$\Box$}\par\medskip

For the sake of completeness, we prove the following elementary fact which
was needed in the above proof.
\begin{lemma} \label{Baire}
Suppose $X$ is a complete countable metric space, $L(X) \subset X$ is the
set of limit
points of $X$ and $S(X) = X-L(X)$ is the open set of nonlimit points of $X$. Then:
\begin{enumerate}
\item $S(X)$ is dense in $X$.
\item $L(X)$ is a complete countable metric space, and so, its set
$S(L(X))$ of isolated points is dense in $L(X)$.
\end{enumerate}
\end{lemma}
\begin{proof}
Let $L(X) = \{ p_1, ...,p_n, ... \}$ be a listing, possibly finite
or empty, of the set of limit points of $X$. If $L(X) =
\mbox{\rm \O }$, then $\overline{S(X)} = S(X)= X$, and so, item~1 holds.
Otherwise, consider the subsets $X_n = X - \{ p_1,
...,p_n \}$ and note that each $X_n$ is an open dense subset of $X$.
The intersection of this countable collection of sets is
equal to $S(X)$ and must be dense in $X$ by Baire's theorem. Hence,
$S(X)$ is dense in $X$, which proves item~1 in the lemma.

Since $S(X)$ is an open set and $X$ is a complete countable metric space, then
$L(X) = X - S(X)$ is a closed countable set which is complete in the
induced metric. Hence, by item~1, $S(L(X))$ is dense in $L(X)$.
\end{proof}

\section{The proof of Corollary~\ref{corolnew}.} 
\label{sec6}
This last section is devoted to the following result, which has Corollary~\ref{corolnew} stated in
the Introduction as a special case.

\begin{corollary}
\label{corolnew2}
Suppose that $M\subset \R^3$ is a connected properly embedded minimal surface with
compact boundary and a limit end of genus zero.
Then $M$ is recurrent for Brownian motion if its boundary is empty, and otherwise
its boundary has full
harmonic measure.
\end{corollary}
\begin{proof}
Suppose for the moment that the corollary holds when
the surface $M$ has nonempty boundary.  Then,
in the special case that the boundary of $M$ is empty,  consider a
compact disk $D\subset M$ and note that $M-\Int(D)$ has full harmonic measure 
by assumption, which implies that $M$ is recurrent for Brownian motion.  Thus, it suffices
to prove the corollary in the special case that $M$ has
nonempty boundary.

Assume now that $\partial M\neq\mbox{\O}$.
Let $\widehat{E}\subset M$ be an end representative of a limit end of $M$ of genus zero.
Since $\widehat{E}$ is proper in $\rth$
with compact boundary, then
item~2 of Theorem~\ref{thm1.2} implies that $\widehat{E}$ has one or two simple limit ends.
Let  $E\subset \widehat{E}$ be an end representative of a simple limit end of $M$ of genus zero.
After a fixed rotation of $M$ and a replacement of $E$ by a subend representative of
its limit end, we may assume that $E$ satisfies the conclusions of Theorem~\ref{thm1.3}, and
$\partial M\subset \{ x_3<0\} $.

We claim that there exist a pair of horizontal open disks $D_1,D_2\subset \R^3-E$ with the following properties.
\begin{enumerate}[(N1)]
\item $\partial D_i\subset E$, $i=1,2$, and $0\leq x_3(D_1)<x_3(D_2)$.
\item $D_1\cap E=\mbox{\O}$ and if we denote by $X_1$ the closure of the mean convex region of $\R^3-(E\cup D_1)$,
then $D_2\subset \R^3-X_1$. In particular, $D_2\cap E=\mbox{\O}$.
\item Define $X_2$ as the closure of the mean convex region of $\R^3-(E\cup D_2)$. Then, $M-E$ is disjoint from
$X_1\cup X_2$. In particular, $M-E$ is contained in the halfspace $\{ x_3\leq x_3(D_2)\} $.
\end{enumerate}
To prove the claim and following the discussion in Sections~\ref{sec4.2} and~\ref{sec4.3}, 
we will explain how to construct the disks $D_1,D_2$ in each of the cases given by (G1) with $M_{\infty }$
being a Riemann minimal example with horizontal limit tangent plane at infinity, 
or $M_{\infty}$ being a vertical catenoid. In the first case, we simply take $D_1,D_2$ as
the horizontal disks bounded by almost-circles $c_0(\ve),c_1(\ve)$ contained in the boundary of a
piece $\mathcal{R}_0(\ve)\subset E$ as defined in the paragraph just before Assertion~\ref{ass4.8}. 
In the case  (G1) with $M_{\infty }$ being a vertical catenoid, we take $D_1,D_2$ as the convex horizontal disks bounded
by $\a_n$ and $\be_n$, respectively (here we are using the notation in (\ref{eq:alphabeta})).
Properties (N1), (N2) hold from item~4 of
Theorem~\ref{thm1.3}. Concerning item~(N3), if $M-E$ intersects $X_1$ then one can find a contradiction
by adapting the arguments in paragraphs four and five of the proof of Theorem~\ref{thm1.2}. Hence
$M-E$ is disjoint from $X_1$ and similarly, $M-E$ is disjoint from $X_2$.

As $M-\Int(E)$ is contained in a closed halfspace by item~(N3) and $M-\Int(E)$ is proper, then $M-\Int(E)$ is a 
parabolic surface with compact boundary by Theorem~3.1 in~\cite{ckmr1}. By item~6 in Theorem~\ref{thm1.3},
$E$ is also a parabolic surface with compact boundary. Therefore, $M=(M-\Int(E))\cup E$ is a 
parabolic surface with compact boundary, i.e., $\partial M$ has full harmonic measure.\end{proof}

\vspace{.2cm}

\center{William H. Meeks, III at  profmeeks@gmail.com\\
Mathematics Department, University of Massachusetts, Amherst, MA 01003}
\center{Joaqu\'\i n P\'{e}rez at jperez@ugr.es \qquad\qquad Antonio Ros at aros@ugr.es\\
Department of Geometry and Topology and Institute of Mathematics
(IEMath-GR), University of Granada, 18071, Granada, Spain}

\bibliographystyle{plain}

\bibliography{bill}

\end{document}